\documentclass[12pt,a4paper]{article}
\pdfoutput=1
\usepackage{slashbox,amssymb,amsmath,amsthm}
\usepackage{latexsym,fullpage,color,times,tikz}
\usepackage{helvet,mathrsfs,pifont,bm}
\usepackage{stmaryrd,accents,shadow}
\usepackage{pgf,tikz}
\usepackage[pdftex]{hyperref}
\hypersetup{plainpages=True, pdfstartview=FitV,
colorlinks=true,linkcolor=blue,citecolor=blue}

\newtheorem{cor}{Corollary}
\newtheorem{thm}{Theorem}
\newtheorem{lmm}{Lemma}
\newtheorem{prop}{Proposition}

\newcommand\thickbar[1]{\accentset{\rule{.4em}{1pt}}{#1}}
\def\dd#1{{\,\mathrm{d}}#1}
\def\ve{\varepsilon}
\def\tr#1{\lfloor#1\rfloor}
\def\qed{{\quad\rule{1mm}{3mm}\,}}
\def\pf{\noindent \emph{Proof.}\ }

\def\le{\leqslant}
\def\ge{\geqslant}
\def\cl#1{\lceil#1\rceil}
\def\dbbracket#1{\llbracket #1 \rrbracket}

\DeclareMathOperator{\Exp}{\mathrm{Exp}}

\title{Probabilistic analysis of the ($1+1$)-evolutionary algorithm}
\author{Hsien-Kuei Hwang\thanks{Partially supported by a grant from the
    Franco-Taiwan Orchid Program.} \\
    Institute of Statistical Science, \\
    Institute of Information Science\\
    Academia Sinica\\
    Taipei 115\\
    Taiwan
\and Alois Panholzer\thanks{Supported by the Austrian Science
Foundation FWF under
    the Grant P25337-N23.}\\
    Institut f\"ur Diskrete Mathematik\\ und Geometrie\\
    Technische Universit\"at Wien\\
    Wiedner Hauptstra\ss e 8-10/104 \\
    1040 Wien\\
    Austria
\and Nicolas Rolin\\
    LIPN, Institut Galil\'ee\\
    Universit\'e\ Paris 13\\
    93430, Villetaneuse\\ France
\and Tsung-Hsi Tsai \\
    Institute of Statistical Science \\
    Academia Sinica\\
    Taipei 115\\
    Taiwan
\and Wei-Mei Chen\thanks{Partially supported by MOST under
the Grant 103-2221-E-011-113.}\\
    Department of Electronic and Computer Engineering\\
    National Taiwan University of \\ Science and Technology\\
    Taipei 106\\ Taiwan}
\date{\today}

\begin{document}
\maketitle

\begin{abstract}

We give a detailed analysis of the cost used by the
$(1+1)$-evolutionary algorithm. The problem has been approached in
the evolutionary algorithm literature under various views,
formulation and degree of rigor. Our asymptotic approximations for
the mean and the variance represent the strongest of their kind. The
approach we develop is also applicable to characterize the limit
laws and is based on asymptotic resolution of the underlying
recurrence. While most approximations have their simple formal
nature, we elaborate on the delicate error analysis required for
rigorous justifications.

\end{abstract}

\section{Introduction}

The last two decades or so have seen an explosion of application
areas of evolutionary algorithms (EAs) in diverse scientific or
engineering disciplines. An EA is a random search heuristic, using
evolutionary mechanisms such as crossover and mutation, for finding
a solution that often aims at maximizing an objective function. EAs
are proved to be extremely useful for combinatorial optimization
problems because they can solve complicated problems with reasonable
efficiency using only basic mathematical modeling and simple
operators; see \cite{C06,D01, H97} for more information. Although
EAs have been widely applied in solving practical problems, the
analysis of their performance and efficiency, which often provides
better modeling prediction for practical uses, are much less
addressed; only computer simulation results are available for most
of the EAs in use. See for example \cite{BSW02, DJW98,  GKS99, HY01,
HY02}. We are concerned in this paper with a precise probabilistic
analysis of a simple algorithm called ($1+1$)-EA (see below for more
details).

A typical EA comprises several ingredients: the coding of solution,
the population of individuals, the selection for reproduction, the
operations for breeding new individuals, and the fitness function to
evaluate the new individual. Thus mathematical analysis of the total
complexity or the stochastic description of the algorithm dynamics
is often challenging. It proves more insightful to look instead at
simplified versions of the algorithm, seeking for a compromise
between mathematical tractability and general predictability. Such a
consideration was first attempted by B\"ack \cite{B92} and
M\"uhlenbein in \cite{M92} in the early 1990's for the ($1+1$)-EA,
using only one individual with a single mutation operator at each
stage. An outline of the procedure is as follows.

\begin{quote}
\textbf{Algorithm ($1+1$)-EA}
\begin{enumerate}
\item Choose an initial string $\mathbf{x}
    \in \{0, 1\}^n$ uniformly at random
\item Repeat until a terminating condition is reached
\begin{itemize}
    \item Create $\mathbf{y}$ by flipping each bit of $x$
    independently with probability $p$
    \item Replace $\mathbf{x}$ by $\mathbf{y}$ iff $f(\mathbf{y})
    \ge f(\mathbf{x})$
\end{itemize}
\end{enumerate}
\end{quote}

Step 1 is often realized by tossing a fair coin for each of the $n$
bits, one independently of the others, and the terminating condition
is usually either reaching an optimum state (if known) or by the
number of iterations.

M\"uhlenbein \cite{M92} considered in detail the complexity of
($1+1$)-EA under the fitness function \textsc{OneMax}, which counts
the number of ones, namely, $f(\mathbf{x}) = \sum_{1\le j\le n}
x_j$. The expected time needed to reach the optimum value, which is
often referred to as the expected \emph{optimization time}, for
\textsc{OneMax}, denoted for convenience by $\mathbb{E}(X_n)$, was
argued to be of order $n\log n$, indicating the efficiency of the
($1+1$)-EA. B\"ack \cite{B92} derived expressions for the transition
probabilities. Finer asymptotic approximation of the form
\begin{align}\label{EYn-intro}
    \mathbb{E}(X_n) = en\log n +c_1 n + o(n),
\end{align}
was derived by Garnier et al.\ in \cite{GKS99}, where $c_1\approx
-1.9$ when the mutation rate $p=\frac1n$. They went further by
characterizing the limiting distribution of $\frac{X_n-en\log
n}{en}$ in terms of a log-exponential distribution (which is indeed
a double exponential or a Gumbel distribution). However, some of
their proofs, notably the error analysis, seem incomplete (as
indicated in their paper). Thus a strong result such as
\eqref{EYn-intro} has remained obscure in the EA literature.

More recent attention has been paid to the analysis of the
($1+1$)-EA; see for example \cite{AD11,NW10}. We briefly mention
some progresses. Neumann and Witt \cite{NW09} proved that a simple
Ant Colony Optimization algorithm behaves like the ($1 + 1$)-EA and
all results for ($1+1$)-EA translate directly into those for the ACO
algorithm. Sudholt and Witt \cite{SW10} showed a similar translation
into Particle Swarm Optimization algorithms. Moreover, variants such
as ($\mu+1$)-EA in \cite{W06} and ($1+1$)-EA over a finite alphabet
in \cite{DP12} were investigated. The expected optimization time
required by ($1+1$)-EA has undergone successive improvements, yet
none of them reached the precision of Garnier et al.'s result
\eqref{EYn-intro}; we summarize in the following table some recent
findings.
\begin{small}
\begin{center}
\begin{tabular}{|c|c||c|c|} \hline
\multicolumn{2}{|c||}{\textsc{OneMax} function} &
\multicolumn{2}{|c|}{Linear functionals} \\ \hline\hline
$\begin{array}{c}
\text{Doerr et al.\ \cite{DFW10}}\\
\text{(2010)} \end{array} $ & $
\begin{array}{c} \text{lower bound}\\
(1 - o(1))en \log(n)\end{array}$ & $\begin{array}{c}
\text{Jagerskupper \cite{J11}}\\
\text{(2011)} \end{array}$ & $
\begin{array}{c} \text{upper bound}\\
2.02en \log(n) \end{array}$ \\ \hline $\begin{array}{c}
\text{Sudholt \cite{S10}}\\
\text{2010} \end{array}$
& $\begin{array}{c} \text{lower bound}\\
en\log(n) - 2n \log \log(n)\end{array}$ & $\begin{array}{c}
\text{Doerr et al.\ \cite{DJW10}}\\
\text{(2010)} \end{array}$ & $
\begin{array}{c} \text{upper bound}\\1.39en \log(n)
    \end{array}$ \\ \hline
$\begin{array}{c}
\text{Doerr et al.\ \cite{DFW11}}\\
\text{(2011)} \end{array} $ & $en \log(n) - \Theta(n)$ &
$\begin{array}{c}
\text{Witt \cite{W13}}\\
\text{(2013)} \end{array}$ & $
\begin{array}{c} \text{upper bound}\\
    en \log(n) + O(n)\end{array}$ \\ \hline
\end{tabular}
\end{center}
\end{small}

In this paper we focus on the mutation rate\footnote{From an
algorithmic point of view, a mutation rate of order $\gg \frac1n$
leads to a complexity higher than polynomial, and is thus less
useful.} $p=\frac1n$ and prove that the expected number of steps
used by the ($1+1$)-EA to reach optimum for \textsc{OneMax} function
satisfies
\begin{align}\label{c1c2}
    \mathbb{E}(X_n) = en\log n +c_1 n+\tfrac12e\log n +c_2
    +O\left(n^{-1}\log n\right),
\end{align}
where $c_1$ and $c_2$ are explicitly computable constants. More
precisely,
\begin{align*}
    c_1 = -e\left(\log 2 -\gamma -
    \phi_1\left(\tfrac12\right)\right)
    \approx 1.89254 \,17883\,44686\,82302\,25714\dots,
\end{align*}
where $\gamma$ is Euler's constant,
\begin{align}\label{phi1-z}
    \phi_1(z) := \int_0^{z}
    \left(\frac1{S_1(t)}-\frac1t\right)\dd t,
\end{align}
with $S_1(z)$ an entire function defined by
\begin{align*}
    S_1(z) := \sum_{\ell\ge1}\frac{z^\ell}{\ell!}
    \sum_{0\le j<\ell} (\ell-j)\frac{(1-z)^j}{j!}.
\end{align*}
See \eqref{c2} for an analytic expression and numerical value for
$c_2$.

Note that these expressions, as well as the numerical value, are
consistent with those given in \cite{GKS99}. Finer properties such
as more precise expansions for $\mathbb{E}(X_n)$, the variance and
limiting distribution will also be established. The extension to
$p=\frac cn$ does not lead to additional new phenomena as already
discussed in \cite{GKS99}; it is thus omitted in this paper.

Our approach relies essentially on the asymptotic resolution of the
underlying recurrence relation for the optimization time and the
method of proof is different from all previous approaches (including
Markov chains, coupon collection, drift analysis, etc.). More
precisely, we consider $f(\mathbf{x}) = \sum_{1\le j\le n}x_j$ and
study the random variables $X_{n,m}$, which counts the number of
steps used by ($1+1$)-EA before reaching the optimum state
$f(\mathbf{x})=n$ when starting from $f(\mathbf{x})=n-m$. We will
derive very precise asymptotic approximations for each $X_{n,m}$,
$1\le m\le n$. In particular, the distribution of $X_{n,m}$ is for
large $n$ well approximated by a sum of $m$ exponential
distributions, and this in turn implies a Gumbel limit law when
$m\to\infty$; see Table~\ref{tab:max-lead} for a summary of our
major results.

In addition to its own methodological merit of obtaining stronger
asymptotic approximations and potential use in other problems in EA
of similar nature, our approach, to the best of our knowledge,
provides the first rigorous justification of Garnier et al.'s
far-reaching results \cite{GKS99} fifteen years ago.

Although the results for linear functions strongly support the
efficiency of $(1+1)$-EA, there exist several hard instances; for
example, functions with $\Theta (n^{n})$ expected time complexity
for $(1+1)$-EA were constructed in Droste et al. \cite{DJW02}, while
a naive complete search requires only $2^n$ to find out the global
optimum under an arbitrary function. Along another direction, long
path problems were introduced by Horn et al. \cite{HGD94}, and
examined in detail in Rudolph \cite{R97}; in particular, he studied
long $k$-paths problems (short-cuts all having distances at least
$k$) and proved that the expected time to reach optimum is
$O(k^{-1}n^{k+1})$. Droste et al.\ \cite{DJW02} then derived an
exponential time bound when $k=\sqrt{n-1}$.

The ($1+1$)-EA is basically a randomized hill-climbing heuristic and
cannot replace the crossover operator. Jansen and Wegener
\cite{JW05} showed a polynomial time for an EA using both mutation
and crossover, while $(1 + 1)$-EA necessitates exponential running
times. A more recent natural example \cite{DJKNT13} is the all-pairs
shortest path problem for which an EA using crossover reaches an
$O(n^3\log n)$ expected time bound, while $(1+1)$-EA needs a higher
cost $\Theta(n^{4})$.

This paper is organized as follows. We begin with deriving the
recurrence relation satisfied by the random variables $X_{n,m}$
(when the initial configuration is not random). From this
recurrence, it is straightforward to characterize inductively the
distribution of $X_{n,m}$ for small $1\le m=O(1)$. The hard case
when $m\to\infty, m\le n$ requires the development of more
asymptotic tools, which we elaborate in
Section~\ref{sec:Asymptotic_sum}. Asymptotics of the mean values of
$X_{n,m}$ and $X_n$ are presented in Section~\ref{sec:Exp} with a
complete error analysis and extension to a full asymptotic
expansion. Section~\ref{sec:Var} then addresses the asymptotics of
the variance. Limit laws are established in Section~\ref{sec:LL} by
an inductive argument and fine error analysis. Finally, we consider
in Section~\ref{sec:LO} the complexity of the $(1+1)$-EA using the
number of leading ones as the fitness function. Denote the
corresponding cost measure by $Y_n$ and $Y_{n,m}$, respectively. We
summarize our major results in the following table.
\begin{small}
\begin{table}[!ht]
\begin{center}
\begin{tabular}{|c||c|c|}\hline
$m$ & $\begin{array}{c} \textsc{OneMax}\\
m=O(1)\!: \text{Sum of Exp (Thm \ref{thm:YO1})}\\
m\to\infty\!: \text{Gumbel (Thm \ref{thm:Y-all})}
\end{array}$
& $\begin{array}{c} \textsc{LeadingOnes} \\
m=O(1)\!: \text{Mixture of Gamma (Thm \ref{thm:XO1})}\\
m\to\infty\!: \text{Normal (Thm \ref{thm:X-all})}
\end{array}$\\ \hline\hline
$O(1)$ & $\begin{array}{c}
\mathbb{P}\left(\frac{X_{n,m}}{en}\le x
\right) \to \left(1-e^{-x}\right)^m \end{array}$ &
$\begin{array}{c}
\mathbb{P}\left(\frac{Y_{n,m}}{en}\le x\right)
\to \!\! \sum\limits_{0\le j<m}\!\!\!\frac{\binom{m-1}{j}}{2^{m-1}}
\!\int_0^xe^{-t}\frac{t^j}{j!} \dd t
\end{array}$ \\ \hline
$\begin{array}{c}\to\infty \\ \le n \end{array}$ &
$\begin{array}{l}
\mathbb{P}\left(\frac{X_{n,m}}{en}-\log m-\phi_1(\tfrac mn)
\le x\right) \\
\qquad \to e^{-e^{-x}} \end{array}$
& $\mathbb{P}\left(\frac{Y_{n,m}-\nu_{n,m}}
{\varsigma_{n,m}}\le x\right)\to \frac1{\sqrt{2\pi}}
\int_{-\infty}^x e^{-\frac{t^2}2}\dd t$ \\ \hline
\end{tabular}
\end{center}
\vspace*{-.3cm} \caption{\emph{The limit laws for the number of
stages used by Algorithm $(1+1)$-EA under \textsc{OneMax}
($X_{n,m}$) and \textsc{LeadingOnes} ($Y_{n,m}$) fitness function,
respectively, when starting from the initial state with the
evaluation $n-m$. The function $\phi_1$ is defined in
\eqref{phi1-z}, and the two quantities $\nu_{n,m}$ and
$\varsigma_{n,m}$ are given in \eqref{mu-exact} and \eqref{V-exact},
respectively.}} \label{tab:max-lead}
\end{table}
\end{small}

\section{Recurrence and the limit laws when $m=O(1)$}

Recall that we start from the initial state $f(\mathbf{x})=n-m$ and
that $X_{n,m}$ denotes the number of steps used by ($1+1$)-EA before
reaching $f(\mathbf{x})=n$. We derive first a recurrence relation
satisfied by the probability generating function $P_{n,m}(t) :=
\mathbb{E}(t^{X_{n,m}})$ of $X_{n,m}$.

\begin{lmm} The probability generating function $P_{n,m}(t)$
satisfies the recurrence
\begin{align} \label{Qnmt}
    P_{n,m}(t) = \frac{t\sum_{1\le \ell \le m}
    \lambda_{n,m,\ell} P_{n,m-\ell}(t)}{\displaystyle
    1-\left(1-\sum_{1\le \ell \le m}
    \lambda_{n,m,\ell}\right)t}
    \qquad(1\le m\le n),
\end{align}
for $1\le m\le n$, with $P_{n,0}(t)=1$, where
\begin{align} \label{lambda}
    \lambda_{n,m,\ell} := \left(1-\frac1n\right)^n
    (n-1)^{-\ell} \sum_{0\le j\le \min\{n-m,m-\ell\}}
    \binom{n-m}{j}\binom{m}{j+\ell}(n-1)^{-2j}.
\end{align}
\end{lmm}
\begin{proof}
Start from the state $f(\mathbf{x}) = n-m$ and run the two steps
inside the loop of Algorithm ($1+1$)-EA. The new state becomes
$\mathbf{y}$ with $f(\mathbf{y}) = n-m+\ell$ if $j$ bits in the
group $\{x_i = 1\}$ and $j+\ell$ bits in the other group $\{x_i
=0\}$ toggled their values, where $0\le j\le \max\{n-m,m-\ell\}$ and
$\ell>0$. Thus, the probability from state $\mathbf{x}$ to
$\mathbf{y}$ is given by
\begin{align*}
    \lambda_{n,m,\ell}
    = \sum_{0\le j\le \min\{n-m,m-\ell\}}
    \binom{n-m}{j}\left(\frac1n\right)^j
    \left(1-\frac1n\right)^{n-m-j}
    \binom{m}{j+\ell}\left(\frac1n\right)^{j+\ell}
    \left(1-\frac1n\right)^{m-j-\ell},
\end{align*}
which is identical to \eqref{lambda}. We then obtain
\[
    P_{n,m}(t) = t\sum_{1\le \ell \le m} \lambda_{n,m,\ell}
    P_{n,m-\ell}(t)+\left(1-\sum_{1\le \ell \le m}
    \lambda_{n,m,\ell}\right)tP_{n,m}(t),
\]
and this proves the lemma.
\end{proof}
While this simple recurrence relation seems not new in the EA
literature, tools have been lacking for a direct asymptotic
resolution, which we will develop in detail in this paper.

For convenience, define
\begin{align*}
    \Lambda_{n,m} := \sum_{1\le \ell \le m} \lambda_{n,m,\ell}.
\end{align*}

In particular, when $m=1$,
\[
    \Lambda_{n,1} =
    \lambda_{n,1,1} = \frac1n\left(1-\frac1n\right)^{n-1},
\]
so that
\[
    P_{n,1}(t) = \frac{\frac1n\left(1-\frac1n\right)^{n-1}t}
    {1-\left(1-\frac1n\left(1-\frac1n\right)^{n-1}\right)t}.
\]
This is a standard geometric distribution $\text{Geo}(\rho)$ with
probability $\rho =\frac1n\left(1-\frac1n\right)^{n-1}$ (assuming
only positive integer values). Obviously, taking $t=
e^{\frac{s}{en}}$, we obtain
\[
    P_{n,1}\left(e^{\frac{s}{en}}\right)= \frac{1}{1-s}
    \left(1+O\left(\frac1{n|1-s|}\right)\right),
\]
as $n\to\infty$, uniformly for $|s|\le 1-\ve$, implying, by
Curtiss's convergence theorem (see \cite[\S 5.2.3]{HMC13}), the
\emph{convergence in distribution}
\[
    \frac{X_{n,1}}{en}
    \xrightarrow{(d)} \text{Exp}(1),
\]
where $\text{Exp}(c)$ denotes an exponential distribution with
parameter $c$. Equivalently, this can be rewritten as
\[
    \lim_{n\to\infty}\mathbb{P}
    \left(\frac{X_{n,1}}{en} \le x\right) = 1-e^{-x},
\]
for $x>0$. Such a limit law indeed extends to the case when
$m=O(1)$, which we formulate as follows.

\begin{figure}[!ht]
\begin{center}
\includegraphics[width=3.5cm]{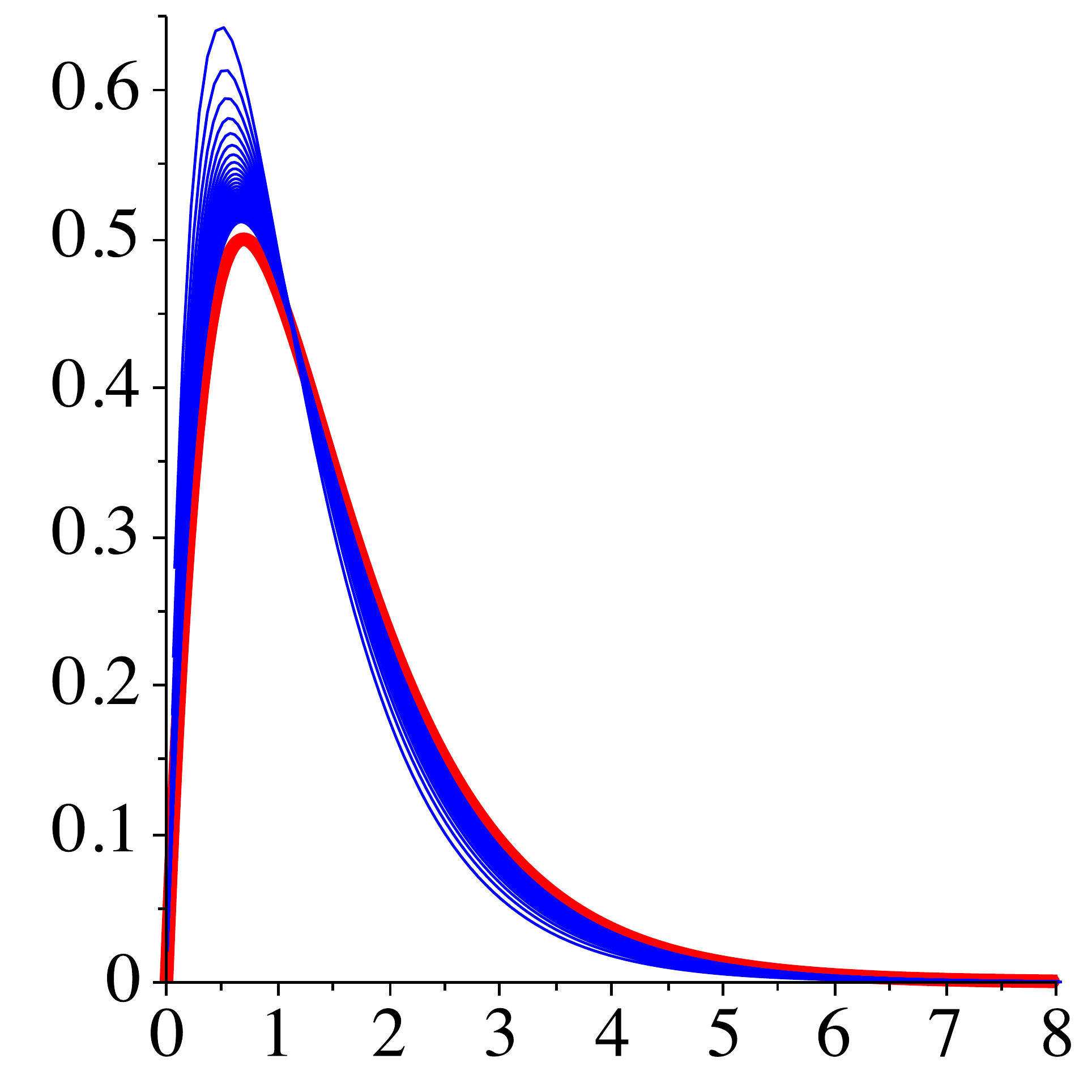}
\includegraphics[width=3.5cm]{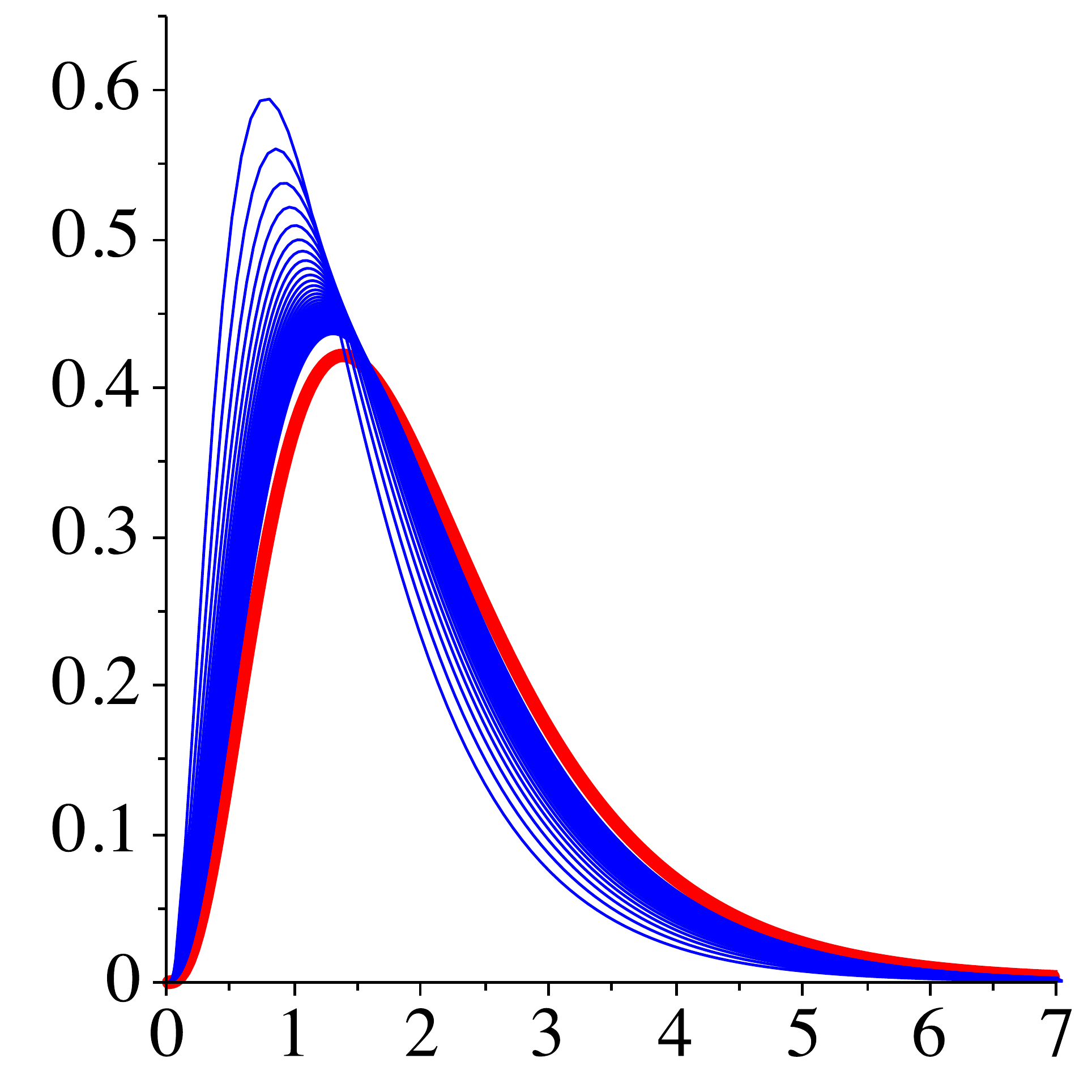}
\includegraphics[width=3.5cm]{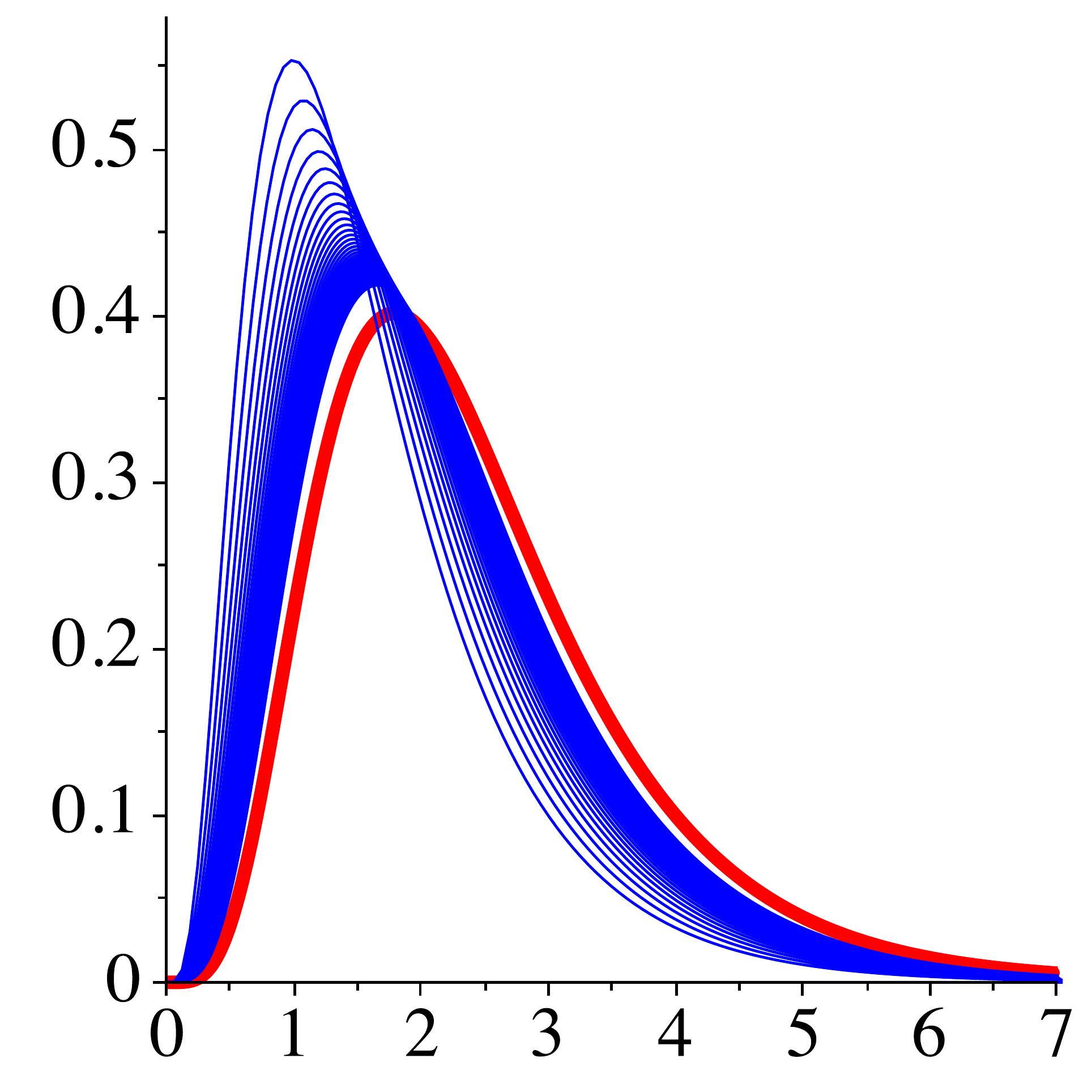}
\includegraphics[width=3.5cm]{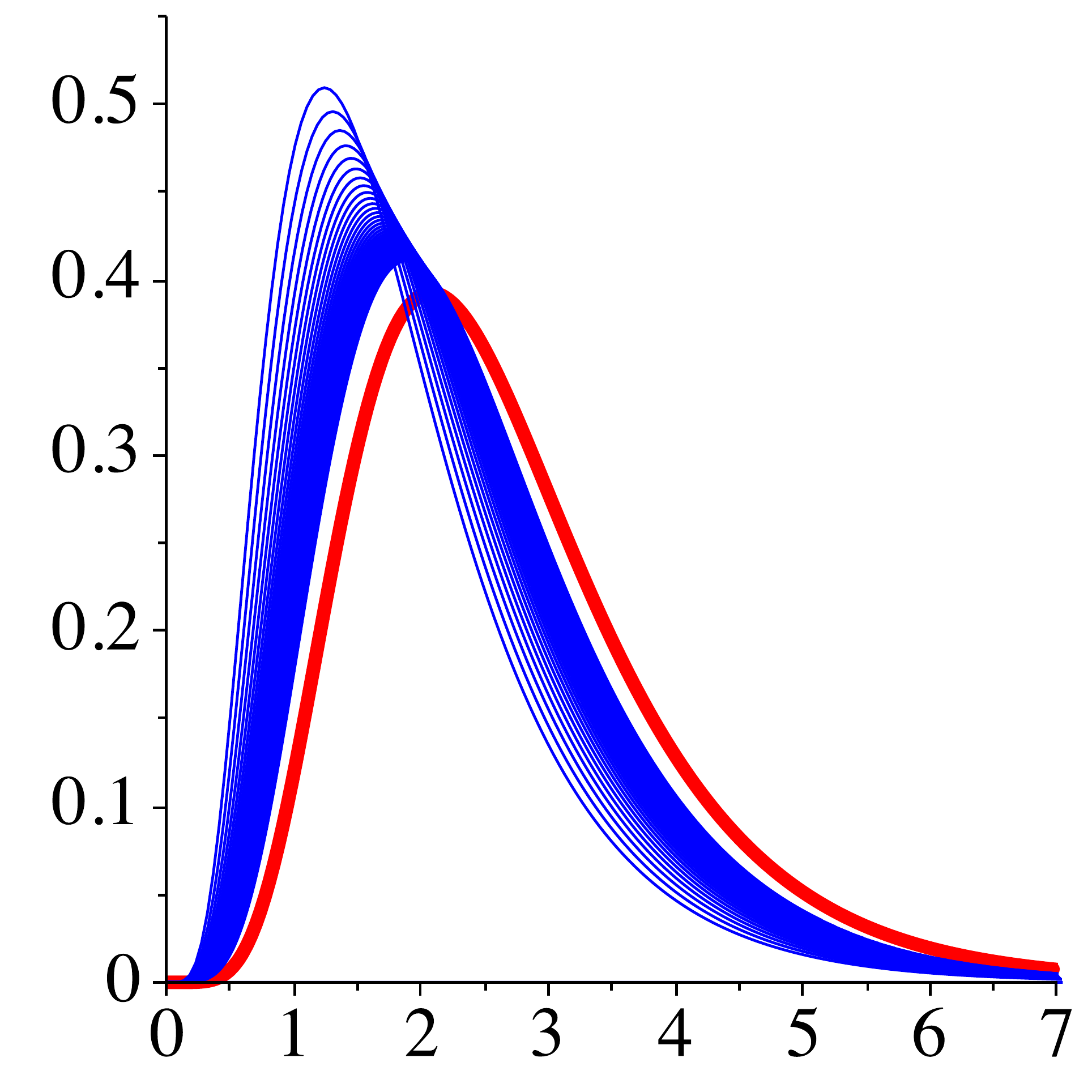}
\end{center}
\vspace*{-.4cm} \caption{\emph{Histograms of $X_{n,2j}/en$ for
$j=1,\dots,4$ (in left to right order) and $n=5,\dots,50$, and their
corresponding limit laws.}}
\end{figure}

Let $H_m^{(i)}=\sum_{1\le j\le m}j^{-i}$ denote the $i$-th order
harmonic numbers and $H_m=H_m^{(1)}$. For convenience, we define
$H_0^{(i)}=0$.
\begin{thm} \label{thm:YO1}
If $m=O(1)$, the time used by ($1+1$)-EA to reach the optimum state
$f(\mathbf{x})= n$, when starting from $f(\mathbf{x}) = n-m$,
converges, when normalized by $en$, to a sum of $m$ exponential
random variables
\begin{align}\label{Y-mO1}
    \frac{X_{n,m}}{en}
    \xrightarrow{(d)} \sum_{1\le r\le m}
    \mathrm{Exp}(r),
\end{align}
with mean asymptotic to $eH_m n$ and variance asymptotic to
$e^2H_m^{(2)}n^2$.
\end{thm}
The convergence in distribution \eqref{Y-mO1} can be expressed
alternatively as
\[
    \lim_{n\to\infty}\mathbb{P}\left(\frac{X_{n,m}}{en}\le x
    \right) = \left(1-e^{-x}\right)^m \qquad(x>0).
\]
\begin{proof}
By the sum definition of $\lambda_{n,m,\ell}$, we see that
\begin{align}
    \lambda_{n,m,\ell} &= \left(1-\frac1n\right)^n
    (n-1)^{-\ell} \sum_{0\le j\le \min\{n-m,m-\ell\}}
    \binom{m}{j+\ell}\binom{n-m}{j}(n-1)^{-2j}\nonumber \\
    &= \binom{m}{\ell} e^{-1}
    n^{-\ell} \left(1+O\left(\frac{(n-m)(m-\ell)}{n^2\ell}
    \right)\right), \label{mO1}
\end{align}
where the $O$-term holds uniformly for $1\le m=o(n)$ and $1\le
\ell\le m$. In particular, for each fixed $m=O(1)$, we then have
\[
    P_{n,m}(t) = \frac{\frac{m}{en}t}
    {1-\left(1-\frac{m}{en}\right)t}\, P_{n,m-1}(t)
    (1+o(1)),
\]
so that
\[
    P_{n,m}(t) = \left(\prod_{1\le r\le m}
    \frac{\frac{r}{en}t}{1-\left(1-\frac{r}{en}\right)t}\right)
    (1+o(1)),
\]
where both $o(1)$-terms are uniform for $|s|\le 1-\ve$. Now take
$t=e^{\frac{s}{en}}$. Then
\begin{align}\label{Qnm-m-bdd}
    P_{n,m}\left(e^{\frac{s}{en}}\right)
    = \left(\prod_{1\le r\le m}
    \frac1{1-\frac{s}{r}}\right)(1+o(1)),
\end{align}
uniformly for $|s|\le 1-\ve$. This and Curtiss's convergence theorem
(see \cite[\S 5.2.3]{HMC13}) imply \eqref{Y-mO1}. The asymptotic
mean and the asymptotic variance can be computed either by a similar
inductive argument or by following ideas used in the Quasi-Power
Framework (see \cite{FS09} or \cite{Hwang98}) that relies on the
\emph{uniformity} of the estimate \eqref{Qnm-m-bdd}
\[
    \frac{\mathbb{E}(X_{n,m})}{en}
    = [s] P_{n,m}\left(e^{\frac{s}{en}}\right)
    \sim [s]\prod_{1\le r\le m}\frac1{1-\frac{s}{r}}
    = H_m,
\]
where $[s^k]f(s)$ denotes the coefficient of $s^k$ in the Taylor
expansion of $f(s)$, and
\begin{align*}
    \frac{\mathbb{V}(X_{n,m})}{(en)^2}
    &= 2[s^2] \log P_{n,m}\left(e^{\frac{s}{en}}\right)\\
    &\sim 2 [s^2] \sum_{1\le r\le m} \log\frac1{1-\frac sr}\\
    &= H_m^{(2)}.
\end{align*}
We will derive more precise expansions below by a direct approach.
\end{proof}

The simple inductive argument fails when $m\to\infty$ and we need
more uniform estimates for the error terms.

\section{Asymptotics of sums of the form
$\sum_{1\le \ell \le m}  a_\ell
\lambda_{n,m,\ell}$\label{sec:Asymptotic_sum}}

Sums of the form
\[
    \sum_{1\le \ell \le m} a_\ell \lambda_{n,m,\ell}
\]
appear frequently in our analysis. We thus digress in this section
to develop tools for deriving the asymptotic behaviors of such sums.

For technical simplicity, we define the sequence $e_n :=
\left(1-\frac1{n+1}\right)^{n+1}$ and the normalized sum
\begin{align}\label{lmbd-star}
    \lambda_{n,m,\ell}^* := \frac{\lambda_{n+1,m,\ell}}{e_n}
    = \sum_{0\le j\le \min\{n+1-m,m-\ell\}}
    \binom{n+1-m}{j}\binom{m}{j+\ell}n^{-\ell-2j}.
\end{align}
Let also
\[
    A_{n,m}^* := \sum_{1\le \ell \le m}
    a_\ell \lambda_{n,m,\ell}^*.
\]
Throughout this paper, we use the abbreviation
\[
    \alpha := \frac{m}{n}.
\]

\paragraph{Asymptotics of $A_{n,m}^*$.}
Observe that most contribution to $A_{n,m}$ comes from small $\ell$,
say $\ell = o(m)$, provided that $a_\ell$ does not grow too fast;
see \eqref{mO1}. We formulate a more precise version as follows.

\begin{lmm} Assume that $\{a_\ell\}_{\ell\ge1}$ is a
given sequence such that $A(z) = \sum_{\ell\ge1}a_\ell z^{\ell-1}$
has a nonzero radius of convergence in the $z$-plane. Then
\begin{align}\label{b-ell}
    A_{n,m}^*
    = \tilde{A}_0(\alpha) +\frac{\tilde{A}_1(\alpha)}{n}
    + O\left(\alpha n^{-2}\right),
\end{align}
where $\tilde{A}_0(\alpha)$ and $\tilde{A}_1(\alpha)$ are entire
functions of $\alpha$ defined by
\begin{align}\label{A0}
    \tilde{A}_0(\alpha) :=
    \sum_{\ell\ge1}\frac{\alpha^\ell}{\ell!}
    \sum_{0\le j<\ell}a_{\ell-j}\frac{(1-\alpha)^j}{j!},
\end{align}
and ($a_0:=0$)
\begin{align}\label{A1}
    \tilde{A}_1(\alpha) := -\frac12
    \sum_{\ell\ge1} \frac{\alpha^\ell}{\ell!}
    \sum_{0\le j<\ell} \frac{(1-\alpha)^j}{j!}
    \left((\ell-j)a_{\ell+1-j}-(\ell+2-j)a_{\ell-1-j}
    +a_{\ell-j}\right).
\end{align}
\end{lmm}
\pf The first term on the right-hand side of \eqref{b-ell} can be
readily obtained as follows. If $1\le m\le n$, then
\begin{align*}
    A_{n,m}^*
    &= \sum_{j\ge0}\binom{n+1-m}{j} n^{-j}
    \sum_{j<\ell\le m}a_{\ell-j} \binom{m}{\ell} n^{-\ell}\\
    &\sim \sum_{j\ge0}\frac{(1-\alpha)^j}{j!}
    \sum_{\ell>j}a_{\ell-j}\frac{\alpha^\ell}{\ell!} \\
    &= \tilde{A}_0(\alpha).
\end{align*}
The more precise approximation in \eqref{b-ell} can be obtained by
refining all estimates, but the details are rather messy, notably
the error analysis. We resort instead to an analytic approach.
Observe that the sum on the left-hand side of \eqref{b-ell} is
itself a convolution. Our analytic proof then starts from the
relation
\begin{align}
    \lambda_{n,m,\ell}^*
    &= [z^{m-\ell}]\left(z+\frac1n\right)^m
    \left(1+\frac zn\right)^{n+1-m}\nonumber \\
    &= \frac1{2\pi i}\oint_{|z|=c} z^{\ell-1}
    \left(1+\frac1{nz}\right)^m
    \left(1+\frac zn\right)^{n+1-m} \dd z, \label{lnml}
\end{align}
where $c>0$. The relation \eqref{lnml} holds \emph{a priori} for
$1\le \ell \le m$, but the right-hand side becomes zero for
$\ell>m$. It follows that
\[
    A_{n,m}^*
    = \frac1{2\pi i}\oint_{|z|=c} A(z)
    \left(1+\frac1{nz}\right)^m
    \left(1+\frac zn\right)^{n+1-m} \dd z,
\]
where $0<c<\varrho$, $\varrho$ being the radius of convergence of
$A$. By the expansion
\begin{align*} 
\begin{split}
    &\left(1+\frac1{nz}\right)^m
    \left(1+\frac zn\right)^{n+1-m}\\&\qquad =
    e^{\frac{\alpha}z+(1-\alpha)z}\left(1-\frac1{2n}
    \left((1-\alpha)z^2 -2z+\frac{\alpha}{z^2}\right)
    + O\left(\frac{(1-\alpha)^2|z|^4
    +\alpha^2 |z|^{-4}}{n^2}\right)\right),
\end{split}
\end{align*}
uniformly for $z$ on the integration path, and the integral
representations
\begin{align*}
    \tilde{A}_0(\alpha) &=\frac1{2\pi i}\oint_{|z|=c} A(z)
    e^{\frac{\alpha}z+(1-\alpha)z} \dd z \\
    \tilde{A}_1(\alpha) &=-\frac1{4\pi i}\oint_{|z|=c} A(z)
    \left((1-\alpha)z^2 -2z+\frac{\alpha}{z^2}\right)
    e^{\frac{\alpha}z+(1-\alpha)z} \dd z,
\end{align*}
we deduce \eqref{b-ell}. The expression \eqref{A0} is then obtained
by straightforward term-by-term integration. For \eqref{A1}, we
apply the relation
\[
    \frac{\text{d}}{\text{d}z}\,e^{\frac{\alpha}z+(1-\alpha)z}
    = \left(-\frac{\alpha}{z^2}+1-\alpha\right)
    e^{\frac{\alpha}z+(1-\alpha)z},
\]
and integration by parts, and then obtain
\begin{align}\label{A1a}
    \tilde{A}_1(\alpha) =-\frac1{4\pi i}\oint_{|z|=c}
    \left((1-z^2)A'(z)+(1-4z)A(z)\right)
    e^{\frac{\alpha}z+(1-\alpha)z} \dd z.
\end{align}
Substituting the series expansion $A(z) = \sum_{\ell\ge1}a_\ell
z^{\ell-1}$ and integrating term by term, we get \eqref{A1}. \qed

When $\alpha$ tends to the two boundaries $0$ and $1$, we have
\[
    A_{n,m}^*
    \sim \tilde{A}_0(\alpha)\sim
    \begin{cases}
        \dfrac{a_k}{k!}\,\alpha^k ,& \text{as }\alpha\to0^+,\\
        \quad\\
        \displaystyle \sum_{\ell\ge1} \frac{a_\ell}{\ell!},
        &\text{as }\alpha\to1^{-},
    \end{cases}
\]
where $k$ is the smallest integer such that $a_k\ne0$.

\paragraph{Asymptotics of $\sum_{1\le \ell \le m}  \ell^r
\lambda_{n,m,\ell}^*$} We now discuss special sums of the form
\[
    \Lambda_{n,m}^{(r)}
    := \sum_{1\le \ell \le m} \ell^r \lambda_{n,m,\ell},
\]
which will be repeatedly encountered below. Define
\[
    \thickbar{\Lambda}_{n,m}^{(r)} := \sum_{1 \le \ell \le m}
    \ell^r \lambda_{n,m,\ell}^*,
\]
so that $\Lambda_{n,m}^{(r)} = e_n
\thickbar{\Lambda}_{n-1,m}^{(r)}$. For convenience, we also write
\begin{equation}\label{eqn:thickbarLambda}
    \thickbar{\Lambda}_{n,m}
    := \thickbar{\Lambda}_{n,m}^{(0)}
    = \sum_{1 \le \ell \le m} \lambda_{n,m,\ell}^*.
\end{equation}
Let $I_k$ denote the modified Bessel functions
\[
    I_k(2z) := \sum_{j\ge0}\frac{z^{2j+k}}
    {j!(j+k)!}\qquad(k\in\mathbb{Z}).
\]

\begin{cor} Uniformly for $1\le m\le n$ \label{cor:Ur}
\begin{align}\label{Lmbd}
    \thickbar{\Lambda}_{n,m}^{(r)} = S_r(\alpha)
    + \frac{U_r(\alpha)}{n}
    + O\left(\alpha n^{-2}\right),
\end{align}
for $r=0,1,\dots$, where both $S_r$ and $U_r$ are entire functions
given by
\begin{align*}
    S_r(z) = \sum_{\ell \ge 1} \frac{z^{\ell}}{\ell!}
    \sum_{0\le j<\ell} (\ell-j)^r\frac{(1-z)^{j}}{j!},
\end{align*}
and
\begin{align}\label{Ur}
    U_r(\alpha)=\begin{cases} \displaystyle
        \frac{S_0(\alpha)}2-
        \frac32\sqrt{\frac{\alpha}{1-\alpha}}\,
        I_1\left(2\sqrt{\alpha(1-\alpha)}\right),
        & \text{if }r=0\\ \displaystyle
        -\frac12\left((2r-1)S_r(\alpha) +
        \sum_{0\le j<r}\binom{r}{j}
        \frac{j-(-1)^{r-j}(2r+2-3j)}{r+1-j}
        S_j(\alpha)\right), &\text{if }r\ge1.
    \end{cases}
\end{align}
\end{cor}
In particular,
\begin{align}\label{Ura}
\begin{split}
    U_1(\alpha) &= -S_0(\alpha) -\tfrac12S_1(\alpha)\\
    U_2(\alpha) &= \phantom{-}S_0(\alpha)
    -2S_1(\alpha)-\tfrac32S_2(\alpha)\\
    U_3(\alpha) &= -S_0(\alpha)+2S_1(\alpha)-3S_2(\alpha)
    -\tfrac52S_3(\alpha).
\end{split}
\end{align}
These are sufficient for our uses.
\begin{proof}
We start with the integral representation (see \eqref{A1a})
\[
    U_r(\alpha)=-\frac1{4\pi i}\oint_{|z|=c}
    \left((1-z^2)E_r'(z)+(1-4z)E_r(z)\right)
    e^{\frac{\alpha}z+(1-\alpha)z} \dd z,
\]
where $E_r(z) :=\sum_{\ell\ge1}\ell^r z^{\ell-1}$. When $r=0$, we
have $E_0(z) = (1-z)^{-1}$. Thus
\[
    U_0(\alpha) =-\frac1{4\pi i}\oint_{|z|=c}
    \left(-\frac1{1-z}+3\right)
    e^{\frac{\alpha}z+(1-\alpha)z} \dd z .
\]
Note that
\begin{align}\label{Sr-int-rep}
    S_r(\alpha) = \frac1{2\pi i}
    \oint_{|z|=c} E_r(z) e^{\frac{\alpha}z+(1-\alpha)z} \dd z
    \qquad(r=0,1,\dots).
\end{align}
Thus
\begin{align*}
    U_0(\alpha) = \frac{S_0(\alpha)}2
    - \frac32\sum_{\ell\ge1}\frac{\alpha^\ell(1-\alpha)^{\ell-1}}
    {\ell!(\ell-1)!} ,
\end{align*}
which proves \eqref{Ur} for $r=0$. For $r\ge1$, we have
\begin{align*}
    &(1-z^2)E_r'(z) + (1-4z)E_r(z) \\
    &\qquad= (1-z^2)\sum_{\ell\ge2} \ell^r(\ell-1)z^{\ell-2}
    +(1-4z)\sum_{\ell\ge1}\ell^r z^{\ell-1}\\
    &\qquad= \sum_{\ell\ge1} \ell(\ell+1)^r z^{\ell-1}
    -\sum_{\ell\ge2}(\ell+2)(\ell-1)^r z^{\ell-1} + E_r(z)\\
    &\qquad= \sum_{0\le j\le r} \binom{r}j\,E_{j+1}(z)
    -\sum_{0\le j\le r}\binom{r}j (-1)^{r-j}
    \left(E_{j+1}(z) +2E_j(z)\right) + E_r(z).
\end{align*}
From this and the relation \eqref{Sr-int-rep}, we obtain \eqref{Ur}.
Note that the coefficient of $E_{r+1}$ is zero.
\end{proof}

The Corollary implies specially that
\begin{equation} \label{LS}
    \Lambda_{n,m}^{(r)} = e^{-1} S_r(\alpha)
    \left(1+O\left(n^{-1}\right)\right),
\end{equation}
uniformly for $1\le m\le n$ and $r\ge0$. Since $S_r(z)=z+O(|z|^2)$
as $|z|\to0$, we have the uniform bound
\begin{align}\label{L-ub}
    \thickbar{\Lambda}_{n,m}^{(r)} \asymp S_r(\alpha)\asymp
    \alpha \qquad(1\le m\le n),
\end{align}
meaning that the ratio of $\thickbar{\Lambda}_{n,m}^{(r)}/\alpha$
remains bounded away from zero and infinity for all $m$ in the
specified range.

We also have the limiting behaviors
\[
    \lim_{\alpha\to0}\frac{S_r(\alpha)}{\alpha}=1,
\]
and
\[
    \lim_{\alpha\to1}S_r(\alpha)
    = \sum_{\ell\ge1}\frac{\ell^r}{\ell!}
    = \{e-1,e,2e,5e,15e,\cdots\}.
\]
Without the first term, the right-hand side is, up to $e$, the Bell
numbers (all partitions of a set; Sequence A000110 in Sloane's
Encyclopedia of Integer Sequences).

The following expansions for $S_r(z)$ and $U_r(z)$ as $z \to 0$ will
be used later
\begin{equation}\label{eqn:SrTr_expansion}
\begin{split}
    S_r(z) & = z + \frac{2^{r}+1}{2} z^{2} + O(z^{3}),\\
    U_r(z) & = - \frac{2^{r}+1}{2} z + O(z^{2}),
\end{split}
\end{equation}
for $r=0,1,\dots$.

See also Appendix A for other properties of $S_r(\alpha)$.

\section{The expected values and their asymptotics\label{sec:Exp}}

Consider the mean $\mu_{n,m}:=\mathbb{E}(X_{n,m})=P_{n,m}'(1)$. It
satisfies the recurrence
\[
    \mu_{n,m} = \frac{1}{\Lambda_{n,m}}
    \left(1+ \sum_{1\le \ell \le m}
    \lambda_{n,m,\ell}\, \mu_{n,m-\ell}\right),
\]
for $1\le m\le n$ with $\mu_{n,0}=0$.

From Theorem~\ref{thm:YO1}, we already have $\mu_{n,m}\sim enH_m$
when $m=O(1)$, while for $m\to\infty$ and $m\le n$ we expect that
(recalling $\alpha=\frac{m}{n}$)
\[
    \mu_{n,m} \sim en(H_m+\phi_1(\alpha));
\]
see Section~\ref{sec:ae-small-m} for how such a form arises. We will
indeed derive in this section a more precise expansion. The uniform
appearance of the harmonic numbers $H_m$ may be traced to the
asymptotic estimate \eqref{mO1}; see also Lemma~\ref{lmm-at}.

\begin{thm} The expected value of $X_{n,m}$
satisfies the asymptotic approximation \label{thm:mu}
\begin{align}\label{mu-asymp0}
    \frac{\mathbb{E}(X_{n,m})}{en} = H_m + \phi_1(\alpha)
    + \frac{H_m-\phi_1(\alpha)+2\phi_2(\alpha)
    +2\alpha\phi_1'(\alpha)}{2n}
    +O\left(n^{-2}H_m\right),
\end{align}
uniformly for $1\le m\le n$, where $\phi_1$ is defined in
\eqref{phi1-z} and $\phi_2$ is an analytic function defined by
\begin{align}\label{phi-2}
\begin{split}
    \phi_2(\alpha) &= \frac12-
    \int_0^\alpha \left(\frac{S_2(x)S_1'(x)}{2S_1(x)^3}
    -\frac{S_0(x)}{S_1(x)^2}-\frac1{2S_1(x)}
    -\frac1{2x^2}+\frac1x\right)\dd x.
\end{split}
\end{align}
\end{thm}
For simplicity, we consider
\[
    \mu_{n,m}^* := \frac{e_n}{n}\,\mu_{n+1,m},
\]
where $e_n := \left(1-\frac1{n+1}\right)^{n+1}$, and we will prove
that
\begin{align}\label{mu-asymp}
    \mu_{n,m}^* = H_m + \phi_1(\alpha) +
    \frac{H_m + \phi_2(\alpha)}{n}
    +O\left(n^{-2}H_m\right),
\end{align}
for $1\le m\le n$, which is identical to \eqref{mu-asymp0}; see
Figure~\ref{fig:mu-diff} for a graphical rendering. More figures
are collected in Appendix B. 
\begin{figure}[!ht]
\begin{center}
\includegraphics[width=6cm]{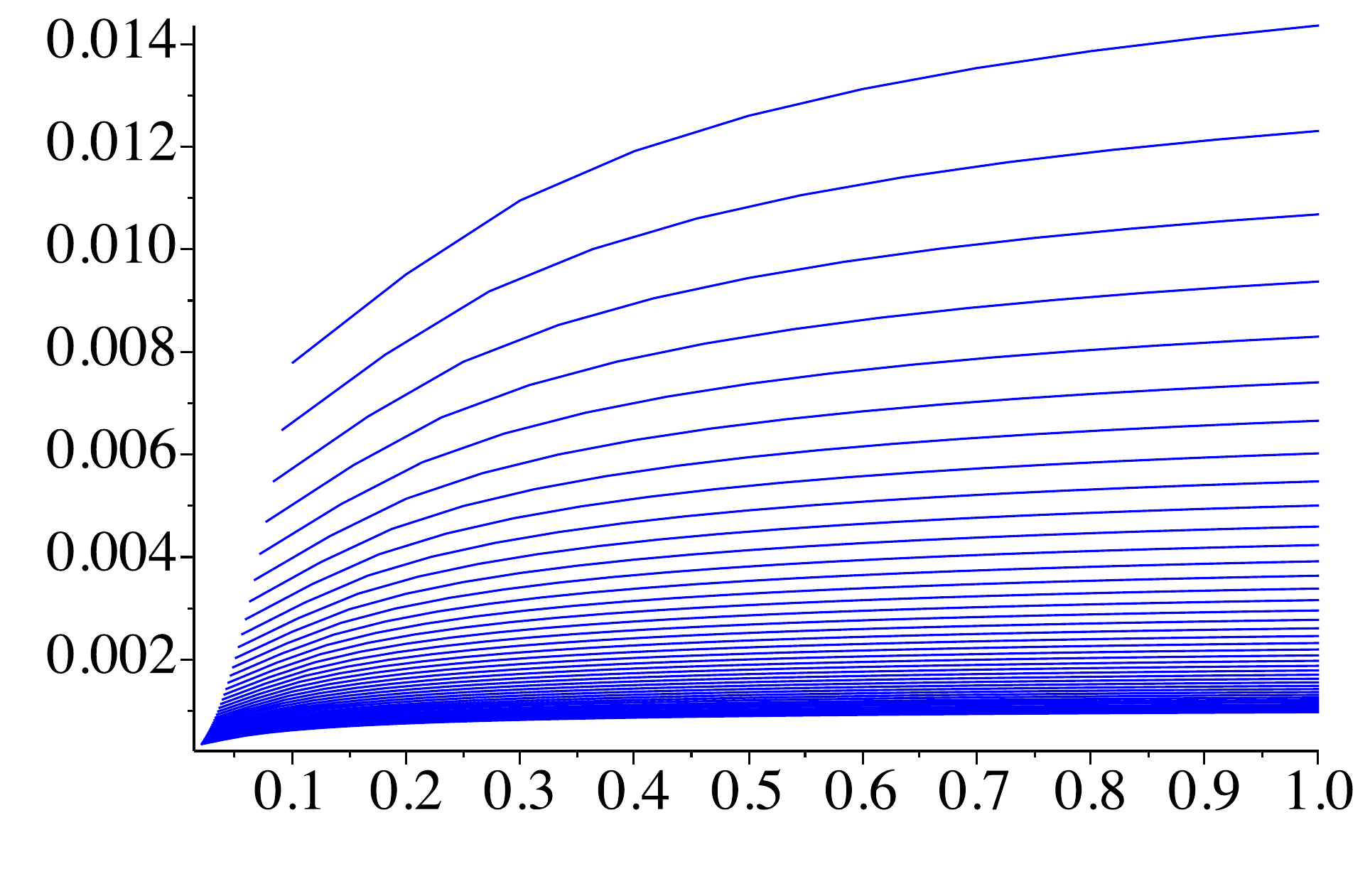}\;\;
\includegraphics[width=6cm]{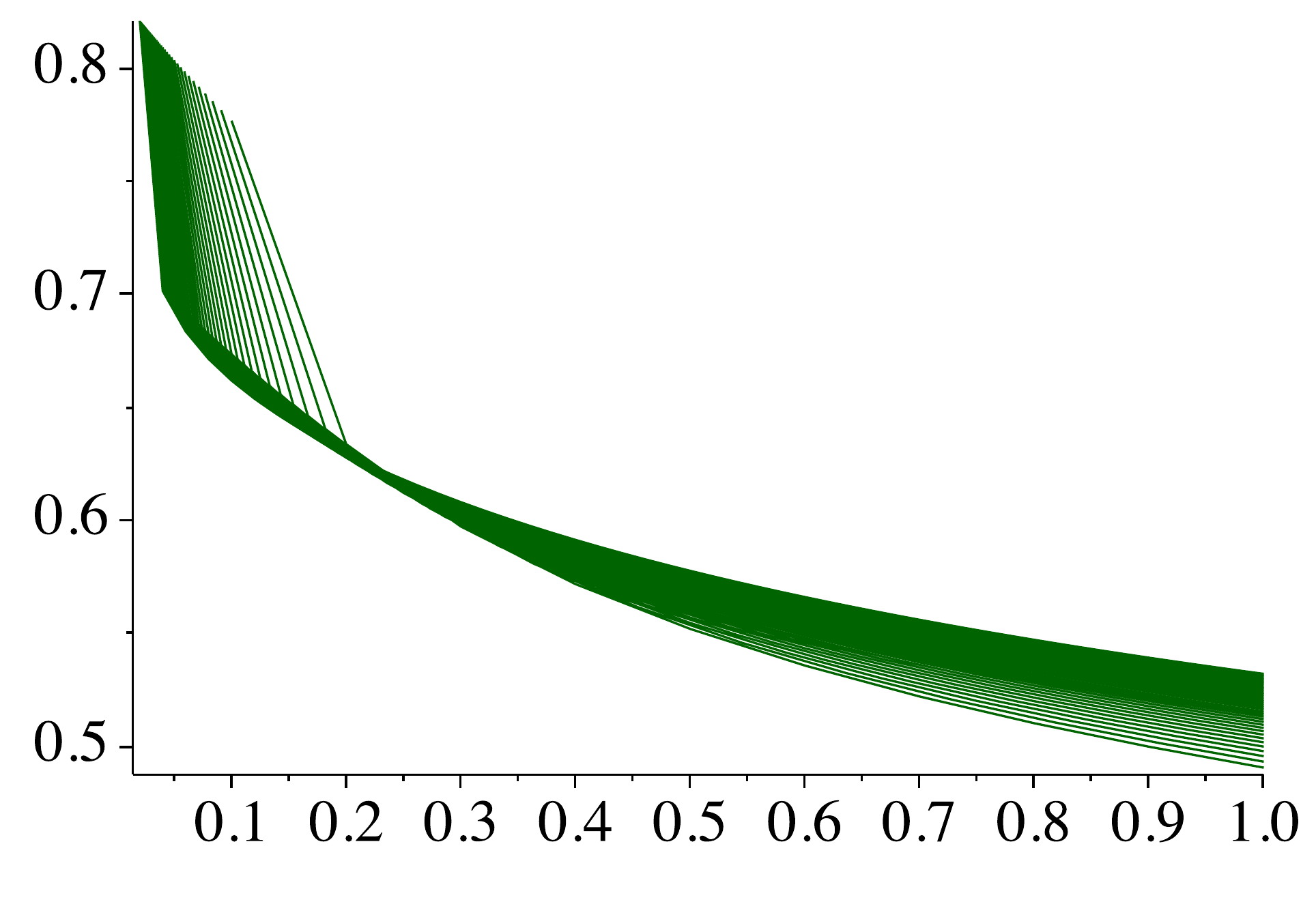}
\end{center}
\vspace*{-.5cm} \caption{\emph{The differences $\mu_{n,m}^*-(H_m +
\phi_1(\alpha) + \frac{H_m + \phi_2(\alpha)}{n})$ for $1\le m\le n$
(normalized to the unit interval) and $n=10,\dots,50$ (left in
top-down order), and the normalized differences $(\mu_{n,m}^*-(H_m +
\phi_1(\alpha) + \frac{H_m + \phi_2(\alpha)}{n}))n^2/H_m$ for
$n=10,\dots,50$ (right).}} \label{fig:mu-diff}
\end{figure}

Our analysis will be based on the recurrence
\begin{align}\label{mu-rr}
    \sum_{1\le \ell\le m}\lambda_{n,m,\ell}^*
    \left(\mu_{n,m}^*  - \mu_{n,m-\ell}^*\right)
    = \frac1n,
\end{align}
or alternatively
\[
    \mu_{n,m}^* = \frac{1}{\thickbar{\Lambda}_{n,m}}
    \left(\frac{1}{n} + \sum_{1 \le \ell \le m}
    \lambda_{n,m,\ell}^* \, \mu_{n,m-\ell}^*\right),
\]
with $\mu_{n,0}^*=0$, where $\lambda_{n,m,\ell}^*$ and
$\thickbar{\Lambda}_{n,m}$ are defined in \eqref{lmbd-star} and
\eqref{eqn:thickbarLambda}, respectively. In particular, this gives
$\mu_{n,1}^*=1$,
\begin{align}\label{mu-star-small}
\begin{split}
    \mu_{n,2}^* &= \frac{3n^2+n-1}{2n^2+2n-1},\\
    \mu_{n,3}^* &= \frac{22n^6+40n^5-19n^4
    -42n^3+14n^2+15n-6)}{(2n^2+2n-1)(6n^4+12n^3-7n^2-9n+6)}.
\end{split}
\end{align}
In general, the $\mu_{n,m}^*$ are all rational functions of $n$
but their expressions become long as $m$ increases. In
Section~\ref{sec:ae-small-m}, we give asymptotic expansions for
$\mu_{n,m}^*$ for small values of $m$ as $n \to \infty$, which are
also required as initial values for obtaining more refined
asymptotic expansions for other ranges of $m$ in
Section~\ref{sec:mu-ae}.

Starting from the asymptotic estimate $\mu_{n,m}^* \sim H_m$ when
$m=O(1)$, we first postulate an Ansatz approximation of the form
\begin{align}\label{mu-phi-1}
    \mu_{n,m}^* \sim H_m + \phi(\alpha) \qquad(1\le m\le n),
\end{align}
for some smooth function $\phi$. Then we will justify such an
expansion by an error analysis relying on Lemma~\ref{lmm-at} after a
proper choice of $\phi$. This same procedure can then be extended
and yields a more precise expansion; see Section~\ref{sec:mu-ae}.

Instead of starting from a state with a fixed number of ones, the
first step of the Algorithm ($1+1$)-EA described in Introduction
corresponds to the situation when the initial state $f(\mathbf{x})$
(the number of $1$s) is not fixed but random. Assume that this
input follows a binomial distribution of parameter $1-\rho\in(0,1)$
(each bit being $1$ with probability $1-\rho$ and $0$ with
probability $\rho$). Denote by $X_n$ the number of steps used by
($1+1$)-EA to reach the optimum state. Such a situation can also be
dealt with by applying Theorem~\ref{thm:Y-all} and we obtain the
same limit law. The following result describes precisely the
asymptotic behavior of the expected optimization time.

\begin{thm} The expected value of $X_n$ satisfies \label{thm:EYn}
\begin{align*}
\begin{split}
    \frac{\mathbb{E}(X_n)}{en} &= \log \rho n + \gamma
    + \phi_1(\rho) \\
    &\quad +\frac{\log \rho n +\gamma+1-\phi_1(\rho)
    +2\rho\phi_1'(\rho)+\rho(1-\rho)\phi_1''(\rho)
    +2\phi_2(\rho)}{2n} +O\left(\frac{\log n}{n^2}\right).
\end{split}
\end{align*}
\end{thm}

Note that $e(\log \rho+\gamma+\phi_1(\rho))$ is an increasing
function of $\rho$, which is consistent with the intuition that it
takes less steps to reach the final state if we start with more
$1$s (small $\rho$ means $1-\rho$ closer to $1$, or $1$ occurring
with higher probability). Also
\[
    1+2\rho\phi_1'(\rho)+\rho(1-\rho)\phi_1''(\rho)
    = -2+\frac1\rho +\frac{2\rho}{S_1(\rho)}
    -\rho(1-\rho)\frac{S_1'(\rho)}{S_1(\rho)^2}.
\]
The constant $c_2$ in \eqref{c1c2} can now be computed and has the
value
\begin{align}\label{c2}
    c_2 = \frac{e}2\left(-\log 2 + \gamma -\phi_1(\tfrac12)+
    2\phi_2(\tfrac12) +\frac{1}{S_1(\tfrac12)}
    -\frac{S_1'(\tfrac12)}{4S_1(\tfrac12)^2}\right)
    \approx 0.59789875\dots.
\end{align}

\medskip
\noindent
\begin{minipage}{0.5\textwidth}
\quad\, Numerically, to compute the value of $\phi_1(\alpha)$ for
$\alpha\in(0,1]$, the most natural way consists in using the Taylor
expansion
\[
    \frac1{S_1(x)}-\frac1x = \sum_{j\ge0}\sigma_j x^j,
\]
and after a term-by-term integration
\[
    \phi_1(\alpha) =
    \sum_{j\ge0}\frac{\sigma_j}{j+1} \,\alpha^{j+1}.
\]
\end{minipage}
\begin{minipage}{0.45\textwidth}
\bigskip
\begin{center}
\includegraphics[width=4.5cm]{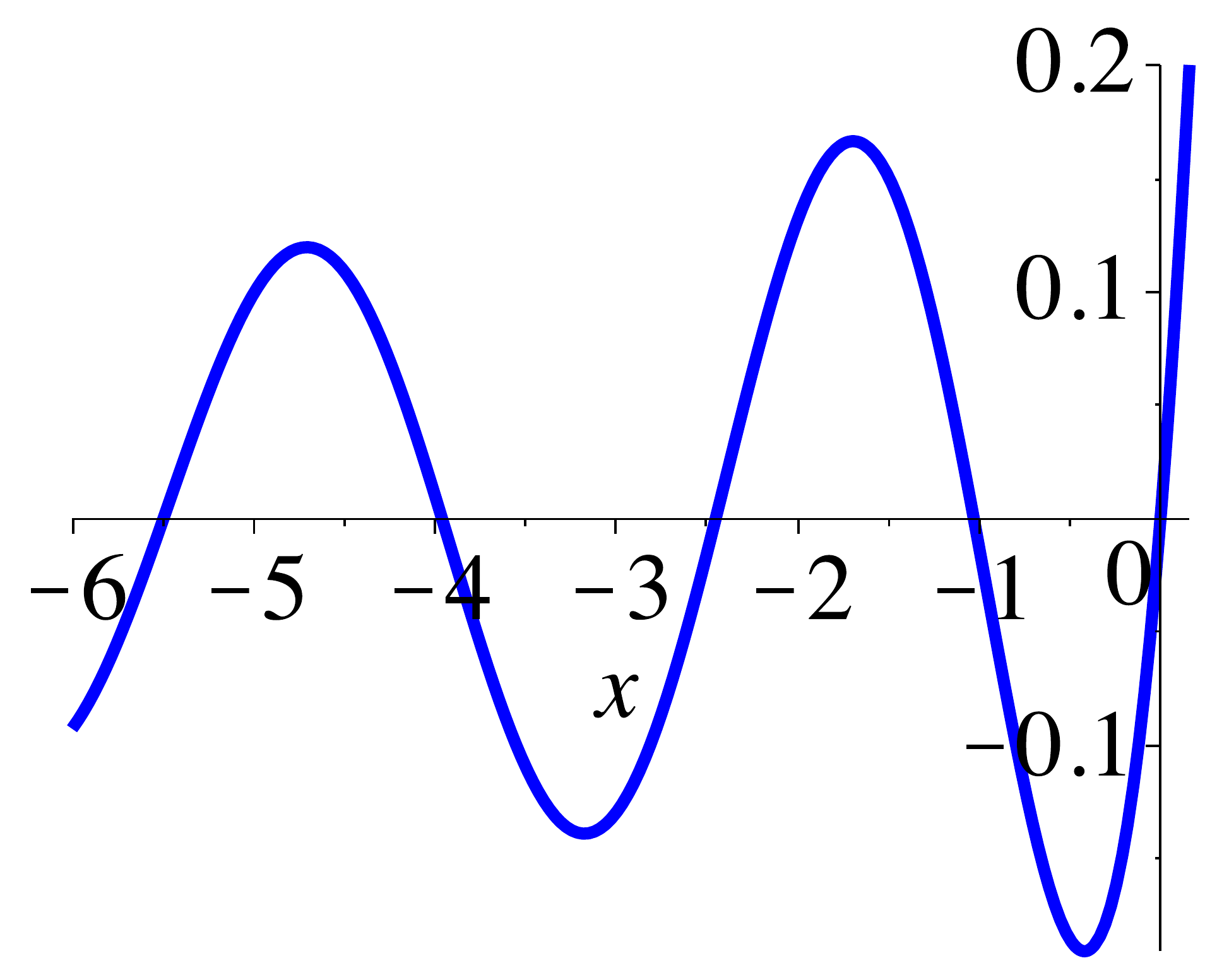}
\end{center}
\vspace*{-.5cm} \centerline{$S_1(x)$ \emph{has an infinity number of
zeros on $\mathbb{R}^-$.}}
\end{minipage}

\medskip

While $S_1(x)$ is an entire functions with rapidly decreasing
coefficients, such an expansion converges slowly when $\alpha\sim1$,
the main reason being that the smallest $|x|>0$ for which $S_1(x)=0$
occurs when $x\approx -1.0288$, implying that the radius of
convergence of this series is slightly larger than unity. Note that
$S_1(0)=0$ but the simple pole is removed by subtracting $\frac1x$.
A better idea is then expanding $\frac1{S_1(x)}-\frac1x$ at $x=1$
and integrating term-by-term
\[
    \phi_1(\alpha) =
    \sum_{j\ge0}\frac{\sigma_j'}{j+1}
    \left(1-(1-\alpha)^{j+1}\right)
    \quad \text{where}\quad
    \frac1{S_1(1-x)}-\frac1{1-x}
    = \sum_{j\ge0}\sigma_j' x^j.
\]
This expansion is numerically more efficient and stable because of
better convergence for $\alpha\in[0,1]$. The same technique also
applies to the calculation of $\phi_2$ and other functions in this
paper.

\subsection{Asymptotic expansions for small $m$}
\label{sec:ae-small-m} Our asymptotic approximation \eqref{mu-asymp}
to $\mu_{n,m}^*$ was largely motivated by intensive symbolic
computations for small $m$. We briefly summarize them here, which
will also be crucial in specifying the initial conditions for the
differential equations satisfied by functions ($\phi_1, \phi_2,
\dots$) involved in the full asymptotic expansion of $\mu_{n,m}^*$;
see \eqref{mu-n-star-ae}.

Starting from the closed-form expressions \eqref{mu-star-small}, we
readily obtain $\mu_{n,0}^* = 0$, $\mu_{n,1}^*= 1$, and
\begin{align*}
    \mu_{n,2}^* & = \tfrac{3}{2} - n^{-1}
    + \tfrac{5}{4}\,n^{-2} - \tfrac{7}{4}\,n^{-3}
    + \tfrac{19}{8}\,n^{-4} - \tfrac{13}{4}\, n^{-5}
    + O(n^{-6}),\\
    \mu_{n,3}^* & = \tfrac{11}{6} - \tfrac{13}{6}\, n^{-1}
    + \tfrac{155}{36} \, n^{-2} - \tfrac{323}{36} \, n^{-3}
    + \tfrac{4007}{216} \, n^{-4} - \tfrac{2783}{72} \, n^{-5}
    + O(n^{-6}).
\end{align*}
Similarly, we have
\begin{align*}
    \mu_{n,4}^* & = \tfrac{25}{12} - \tfrac{41}{12}\,n^{-1}
    + \tfrac{329}{36}\, n^{-2} - \tfrac{917}{36}\, n^{-3}
    + \tfrac{61841}{864}\, n^{-4} - \tfrac{19501}{96}\, n^{-5}
    + O(n^{-6}),\\
    \mu_{n,5}^* & = \tfrac{137}{60} - \tfrac{283}{60}\,n^{-1}
    + \tfrac{2839}{180}\,n^{-2} - \tfrac{19859}{360}\,n^{-3}
    + \tfrac{848761}{4320}\,n^{-4} - \tfrac{5107063}{7200}\,
    n^{-5} + O(n^{-6}).
\end{align*}
From these expansions, we first observe that the leading sequence is
exactly $H_m$ ($H_0:=0$)
\[
    \{H_m\}_{m\ge0} =\left\{0,1, \tfrac32,
    \tfrac{11}6, \tfrac{25}{12}, \tfrac{137}{60},
    \tfrac{49}{20}, \cdots\right\}.
\]
These also suggest the following Ansatz
\[
    \mu_{n,m}^* \approx \sum_{k\ge0}
    \frac{d_k(m)}{n^k},
\]
for some functions $d_k(m)$ of $m$. Using this form and the above
expansions to match the undetermined coefficients of the polynomials
(in $m$), we obtain successively
\begin{align*}
    d_{0}(m) & = H_{m} \quad (m \ge 0),\\
    d_{1}(m) & = H_{m} +\tfrac{1}{2}-\tfrac{3}{2}\,m
    \quad (m \ge 1),\\
    d_{2}(m) & = \tfrac{2}{3}\,H_{m} +\tfrac{1}{12}
    -\tfrac{7}{4}\,m+{\tfrac {11}{12}}\,{m}^{2}
    \quad (m \ge 2),\\
    d_{3}(m) & = \tfrac{1}{2}\,H_{m}
    +{\tfrac {7}{24}}-{\tfrac {575}{432}}\,m
    +{\tfrac {23}{18}}\,{m}^{2}-{\tfrac {283}{432}}\,{m}^{3},
    \quad (m \ge 2),\\
    d_{4}(m) & = {\tfrac {5}{18}}\,H_{m} -{\tfrac {59}{720}}
    -{\tfrac {3439}{3456}}\,m+{\tfrac {15101}{11520}}\,{m}^{2}
    -{\tfrac {19951}{17280}}\,{m}^{3}
    +{\tfrac {5759}{11520}}\,{m}^{4},
    \quad (m \ge 4).
\end{align*}
So we observe the general pattern
\[
    \mu_{n,m}^* \approx \sum_{k\ge0}
    \frac{1}{n^k}\left(b_k H_m + \sum_{0\le j\le k}
    \varpi_{k,j}m^j\right),
\]
for some explicitly computable sequence $b_k$ and coefficients
$\varpi_{k,j}$. A crucial complication arises here: the general form
for each $d_k(m)$ holds only for $m\ge 2\tr{\frac k2}$, and
correction terms are needed for smaller $m$. For example,
\begin{align*}
    d_{1}(m) & = H_{m} +\tfrac{1}{2}-\tfrac{3}{2}\,m
    - \tfrac{1}{2}\, \dbbracket{m=0},
    \quad (m\ge0)\\
    d_{2}(m) & = \tfrac{2}{3}\,H_{m} +\tfrac{1}{12}
    -\tfrac{7}{4}\,m+{\tfrac {11}{12}}\,{m}^{2}
    -\tfrac{1}{12}\, \dbbracket{m=0}
    +\tfrac{1}{12}\, \dbbracket{m=1},
    \quad (m\ge0),
\end{align*}
where we use the Iverson bracket notation $\dbbracket{A} = 1$ if $A$
holds, and $0$, otherwise. It is such a complication that makes the
determination of smaller-order terms more involved.

All the expansions here hold only for small $m$. When $m$ grows, we
see that
\[
    n^{-k}\sum_{0\le j\le k} \varpi_{k,j} m^j
    = \varpi_{k,k} \alpha^k + \varpi_{k,k-1}
    \frac{\alpha^{k-1}}{n} + \text{smaller order terms},
\]
and it is exactly this form that motivated naturally our choice of
the Ansatz \eqref{mu-asymp}.

\subsection{More asymptotic tools}
We develop here some other asymptotic tools that will be used in
proving Theorem~\ref{thm:mu}.

The following lemma is very helpful in obtaining error estimates to
be addressed below. It also sheds new light on the occurrence of the
harmonic numbers $H_m$ in \eqref{mu-asymp}.
\begin{lmm} Consider the recurrence \label{lmm-at}
\[
    \sum_{1\le \ell \le m} \lambda_{n,m,\ell}^*
    (a_{n,m}-a_{n,m-\ell}) = b_{n,m}\qquad (m\ge1),
\]
where $b_{n,m}$ is defined for $1\le m\le n$ and $n\ge1$. Assume
that $|a_{n,0}| \le d$ for $n \ge 1$, where $d \ge 0$. If
$|b_{n,m}|\le \frac cn$ holds uniformly for $1\le m\le n$ and
$n\ge1$, where $c>0$, then
\[
    |a_{n,m}| \le cH_m+d \qquad(0\le m\le n).
\]
\end{lmm}
\pf The result is true for $m=0$. For $m \ge 1$, we start from the
simple inequality
\[
    \thickbar{\Lambda}_{n,m}
    = \sum_{1\le \ell \le m} \lambda_{n,m,\ell}^*
    \ge \frac mn \qquad(1\le m\le n),
\]
because all terms in the sum expression \eqref{lmbd-star} are
positive and taking only one term ($j=0$ and $\ell=1$) gives the
lower bound. Then, by the induction hypothesis,
\begin{align*}
    |a_{n,m}|
    &\le \frac{|b_{n,m}|}{\thickbar{\Lambda}_{n,m}}
    + |a_{n,m-1}|\\
    &\le \frac{c}{n}\cdot \frac{n}{m} + c H_{m-1} +d\\
    &= c H_m +d,
\end{align*}
proving the lemma. \qed

Applying this lemma to the recurrence \eqref{mu-rr}, we then get a
simple upper bound to $\mu_{n,m}^*$.
\begin{cor} For $0\le m\le n$, the inequality
\[
    \mu_{n,m}^* \le H_m
\]
holds.
\end{cor}

\begin{lmm} If $\phi$ is a $C^2[0,1]$-function,
then  \label{lmm-phi}
\begin{align*}
    \sum_{1\le \ell \le m}
    \lambda_{n,m,\ell}^*\left(
    \phi\left(\frac mn\right)-
    \phi\left(\frac {m-\ell}n\right)\right)
    &= \frac{\phi'\left(\alpha\right)}{n}
    \sum_{1\le \ell \le m} \ell
    \lambda_{n,m,\ell}^* +O\left(n^{-2}\right),
\end{align*}
uniformly for $1\le m\le n$.
\end{lmm}
\pf A direct Taylor expansion with remainder gives
\[
    \phi(\alpha) -\phi\left(\alpha-\tfrac\ell n\right)
    = \phi'(\alpha)\tfrac\ell n + O\left(\ell^2 n^{-2}\right),
\]
uniformly for $1\le \ell\le m$, since $\phi''(t) = O(1)$ for
$t\in[0,1]$. The lemma follows from the estimates \eqref{Lmbd}. \qed

The approximation can be easily extended and refined if more
smoothness properties of $\phi$ are known, which is the case for all
functions appearing in our analysis (they are all $C^\infty[0,1]$).

Another standard technique we need is Stirling's formula for the
factorials
\begin{equation}\label{eqn:StirForm}
    \log n! = \log \Gamma(n+1)
    = \left(n+\tfrac{1}{2}\right) \log n -n +
    \tfrac{1}{2} \log(2 \pi) + \tfrac{1}{12} \, n^{-1}
    + O(n^{-3}),
\end{equation}
where $\Gamma$ denotes Euler's Gamma function.

\subsection{Proof of Theorem~\ref{thm:mu}}

\paragraph{Formal calculus.}

Applying formally \eqref{mu-phi-1} and Lemma~\ref{lmm-phi} using
$H_{m} - H_{m-\ell} \sim \frac{\ell}{m}$ and $\phi(\frac{m}{n}) -
\phi(\frac{m-\ell}{n}) \sim \phi'(\alpha) \frac{\ell}{n}$, we have
\begin{align*}
    \frac1n=\sum_{1\le \ell\le m}\lambda_{n,m,\ell}^*
    \left(\mu_{n,m}^* - \mu_{n,m-\ell}^*\right)
    &\sim \sum_{1\le \ell\le m}\lambda_{n,m,\ell}^*
    \left(\frac{\ell}{m} + \phi'(\alpha)
    \frac{\ell}{n}\right) \\
    &\sim \frac1n\left(\frac1\alpha+\phi'(\alpha)\right)
    S_1(\alpha),
\end{align*}
by \eqref{Lmbd}. Thus we see that $\phi$ satisfies
\[
    \phi'(z) = \frac1{S_1(z)}-\frac1z.
\]
We now specify the initial condition $\phi(0)$. Since the postulated
form \eqref{mu-phi-1} holds for $1\le m\le n$ (indeed also true for
$m=0$), we take $m=1$ and see that $\phi(0)=0$ because
$\mu_{n,1}^*=1$. This implies that $\phi=\phi_1$. The first few
terms in the Taylor expansion of $\phi_1(\alpha)$ read as follows.
\begin{align}\label{phi_1-z}
\begin{split}
    \phi_1(z) &= -\tfrac32 z + \tfrac{11}{12} z^2
    - \tfrac{283}{432} z^3 + \tfrac{5759}{11520} z^4 -
    \tfrac{57137}{144000}z^5
    +\tfrac{2353751}{7257600} z^6+\cdots,
\end{split}
\end{align}
which can then be checked with the explicit expressions of
$\mu_{n,m}^*$ for small $m$ (see Section~\ref{sec:ae-small-m}).

\paragraph{Error analysis.}
To justify the form \eqref{mu-phi-1} (with $\phi = \phi_1$), we
consider the difference
\[
    \Delta_{n,m}^* := \mu_{n,m}^* - H_m - \phi_1(\alpha),
\]
which satisfies the recurrence
\[
    \sum_{1\le \ell \le m} \lambda_{n,m,\ell}^*
    \left(\Delta_{n,m}^* -\Delta_{n,m-\ell}^*\right)
    = E_{1}(n,m),
\]
where
\[
    E_{1}(n,m) := \frac{1}{n} - \sum_{1\le \ell \le m}
    \lambda_{n,m,\ell}^*
    \left(H_m-H_{m-\ell} +\phi_1(\alpha)-
    \phi_1\left(\alpha-\tfrac\ell n\right)\right).
\]
By the asymptotic relation \eqref{Lmbd} with $r=1$ and the
definition of $\phi_1$, we have
\[
    \frac{1}{n} = \sum_{1\le \ell\le m} \lambda_{n,m,\ell}^*
    \left(\tfrac\ell m
    +\phi_1'(\alpha)\tfrac\ell n\right)
    +O(n^{-2}),
\]
and thus
\[
    E_{1}(n,m) = -\sum_{1\le \ell \le m} \lambda_{n,m,\ell}^*
    \left(H_m-H_{m-\ell} -\tfrac\ell m+\phi_1(\alpha)-
    \phi_1\left(\alpha-\tfrac\ell n\right)-
    \phi_1'(\alpha)\tfrac\ell n \right) +O(n^{-2})
\]

By Lemma~\ref{lmm-phi}, we see that
\[
    \sum_{1\le \ell \le m} \lambda_{n,m,\ell}^*
    \left(\phi_1(\alpha)-
    \phi_1\left(\alpha-\tfrac\ell n\right)
    -\phi'(\alpha)\tfrac\ell n \right)
    = O(n^{-2}),
\]
uniformly for $1\le m\le n$. On the other hand, we have the upper
bounds
\begin{align*}
    H_m - H_{m-\ell} - \tfrac\ell m
    = \begin{cases}
        O\left(\ell^2 m^{-2}\right),& \text{if } \ell=o(m),\\
        O(H_m), & \text{for }1\le \ell \le m.
    \end{cases}
\end{align*}
Note that the first estimate is only uniform for $1\le \ell=o(m)$.
When $\ell$ is close to $m$, say $m-\ell=O(m^{1-\ve})$, the
left-hand side blows up with $m$ but the right-hand side
$O\left(\ell^2 m^{-2}\right)$ remains bounded. Thus we split the sum
at $\cl{\sqrt{m}}$ and then obtain ($H_m-H_{m-\ell}-\frac\ell m=0$
when $\ell=1$)
\begin{align}\label{harmonic-split}
\begin{split}
    & \sum_{2\le \ell \le m} \lambda_{n,m,\ell}^*
    \left(H_m-H_{m-\ell} -\tfrac\ell m\right)\\
    &\qquad = O\left(m^{-2}\sum_{1\le \ell \le \cl{\sqrt{m}}}
    \ell^2\lambda_{n,m,\ell}^*+ H_m
    \sum_{\sqrt{m}+1\le \ell\le m}\lambda_{n,m,\ell}^* \right).
\end{split}
\end{align}
Now, by \eqref{lmbd-star},
\begin{align}\label{small-ell}
\begin{split}
    m^{-2} \sum\limits_{2 \le \ell \le m}
    \ell^{2} \lambda_{n,m,\ell}^*
    & = O\left(m^{-2}
    \sum_{j\ge 0}\frac{(1-\alpha)^j}{j!}
    \sum_{j+2\le \ell\le m}
    \frac{(j+\ell)^{2}\alpha^{\ell}}{\ell!}  \right) \\
    &= O\left(m^{-2} \alpha^2\right)
    =O\left(n^{-2}\right),
\end{split}
\end{align}
and
\begin{align}\label{large-ell}
\begin{split}
    H_m \sum\limits_{\sqrt{m}+1 \le \ell \le m}
    \lambda_{n,m,\ell}^* & = O\left(H_m
    \sum_{j\ge 0}\frac{(1-\alpha)^j}{j!}
    \sum_{j+\sqrt{m}+1\le \ell\le m}
    \frac{\alpha^{\ell}}{\ell!}  \right) \\
    &= O\left(H_m \sum_{\ell \ge \sqrt{m}+1}
    \frac{\alpha^\ell}{\ell!} \right)
    = O\left(\frac{H_m \alpha^{\sqrt{m}+1}}
    {\Gamma(\sqrt{m}+2)} \right).
\end{split}
\end{align}
By Stirling's formula~\eqref{eqn:StirForm}, the last $O$-term is of
order
\[
    \frac{m^{\frac14} \, H_m}{n} \,
    e^{-\sqrt{m}(\log n-\frac12\log m -1)}
    = O(n^{-2}),
\]
for $m\ge 1$. Combining these estimates, we then obtain
\[
    E_{1}(n,m) = O(n^{-2}),
\]
uniformly for $1\le m\le n$. Thus $\Delta_{n,m} := n \,
\Delta_{n,m}^*$ satisfies a recurrence of the form
\[
    \sum_{1\le \ell \le m} \lambda_{n,m,\ell}^*
    \left(\Delta_{n,m} -\Delta_{n,m-\ell}\right)
    = O(n^{-1}) \qquad(1\le m\le n),
\]
with $\Delta_{n,0}=0$. It follows, by applying Lemma~\ref{lmm-at},
that $\Delta_{n,m} = O(H_m)$, and we conclude that, uniformly for $0
\le m \le n$,
\[
    \mu_{n,m}^* = H_m + \phi_1(\alpha) +O\left(n^{-1}H_m\right).
\]
This proves the first two terms of the asymptotic approximation to
$\mu_{n,m}^*$ in \eqref{mu-asymp}. The more refined expansion is
obtained by refining the same calculations and justification, which
we carry out the main steps subsequently.

\paragraph{Refined computations.}
We consider now the difference
\begin{equation*}
    \Delta_{n,m}^*
    := \mu_{n,m}^* - \left(H_{m}+\phi_{1}(\alpha)\right)
    - \frac{1}{n} \left(b_{1} H_{m} + \phi_{2}(\alpha)\right),
\end{equation*}
and will determine the constant $b_{1}$ and the function
$\phi_{2}(z)$ such that
\begin{equation}\label{eqn:Delta_bound}
    \Delta_{n,m}^* = O(n^{-2} H_{m}),
\end{equation}
uniformly for $1 \le m \le n$, which then proves
Theorem~\ref{thm:mu}. By \eqref{mu-rr}, $\Delta_{n,m}^*$ satisfies,
for $1 \le m \le n$, the recurrence
\begin{equation}\label{eqn:recDelta}
    \sum_{1\le \ell \le m} \lambda_{n,m,\ell}^*
    \left(\Delta_{n,m}^* -\Delta_{n,m-\ell}^*\right)
    = E_{2}(n,m),
\end{equation}
where
\begin{align*}
    E_{2}(n,m) & := \frac{1}{n} -
    \sum_{1 \le \ell \le m} \lambda_{n,m,\ell}^* \,
    (H_{m}-H_{m-\ell})
    - \sum_{1 \le \ell \le m}
    \lambda_{n,m,\ell}^* \left(\phi_{1}\Big(\frac{m}{n}\Big) -
    \phi_{1}\Big(\frac{m-\ell}{n}\Big)\right)\\
    & \quad \mbox{} - \frac{b_{1}}{n}
    \sum_{1 \le \ell \le m}
    \lambda_{n,m,\ell}^* \, (H_{m}-H_{m-\ell})
    - \frac{1}{n} \sum_{1 \le \ell \le m}
    \lambda_{n,m,\ell}^*
    \left(\phi_{2}\Big(\frac{m}{n}\Big) -
    \phi_{2}\Big(\frac{m-\ell}{n}\Big)\right).
\end{align*}
In particular, $\Delta_{n,0}^* = - \frac{\phi_{2}(0)}{n}$.

The hard part here is to derive an asymptotic expansion for
$E_{2}(n,m)$ that holds uniformly for $1 \le m \le n$ as $n \to
\infty$. To that purpose, we first extend Lemma~\ref{lmm-phi} by
using a Taylor expansion of third order for a
$C^{\infty}[0,1]$-function $\phi(z)$, which then gives, uniformly
for $1 \le m \le n$,
\begin{align}
    &\sum_{1 \le \ell \le m} \lambda_{n,m,\ell}^*
    \left(\phi\Big(\frac{m}{n}\Big)
    - \phi\Big(\frac{m-\ell}{n}\Big)\right)\notag \\
    &\qquad= \frac{\phi'(\alpha)}{n} \sum_{1 \le \ell \le m}
    \ell \lambda_{n,m,\ell}^* - \frac{\phi''(\alpha)}{2 n^{2}}
    \sum_{1 \le \ell \le m} \ell^{2} \lambda_{n,m,\ell}^*
    + O(n^{-3})\notag \\
    &\qquad = \frac{\phi'(\alpha)}{n}\,S_1(\alpha)
    +\frac1{2n^2}\left(2\phi'(\alpha)U_1(\alpha)
    -\phi''(\alpha)S_2(\alpha)\right)+O(n^{-3}),
    \label{eqn:phi_expansion}
\end{align}
where we used Corollary~\ref{cor:Ur}.

We now examine weighted sums involving the difference of the
harmonic numbers. We start with the following identity whose proof
is straightforward. For a given function $f(x)$, let $\nabla$ denote
the backward difference operator $\nabla f(x) = f(x) - f(x-1)$. Then
for $0 \le \ell \le m$
\begin{equation*}
    f(m) + f(m-1) + \cdots + f(m-\ell+1)
    = \sum_{1\le k\le m}
    \binom{\ell}{k} (-1)^{k-1} \nabla^{k-1} f(m).
\end{equation*}
Note that the sum vanishes for $k>\ell$. Take $f(x) = \frac{1}{x}$.
Then we obtain
\begin{equation*}
    H_{m} - H_{m-\ell}
    = \sum_{k=1}^{m} \frac{\ell(\ell-1)\cdots(\ell-k+1)}
    {k m (m-1)\cdots(m-k+1)},
\end{equation*}
for $0 \le \ell \le m$. This relation implies that, for $0 \le \ell
\le \frac m2$ and $m \ge 1$,
\begin{equation*}
    H_{m} - H_{m-\ell} = \frac{\ell}{m} +
    \frac{\dbbracket{m \ge 2} \, \ell (\ell-1)}{2m(m-1)}
    +O\Big(\frac{\ell(\ell-1)(\ell-2)}{m^{3}}\Big).
\end{equation*}
By the same argument we used above for \eqref{harmonic-split}, we
get the expansion
\begin{align}
    & \sum_{1 \le \ell \le m}
    \lambda_{n,m,\ell}^* \, (H_{m}-H_{m-\ell})\notag\\
    & \qquad = \sum_{1 \le \ell \le m} \left(\frac{\ell}{m}
    + \frac{\dbbracket{m \ge 2} \, \ell(\ell-1)}{2m(m-1)}
    + O\Big(\frac{\ell(\ell-1)(\ell-2)}{m^{3}}\Big)\right)
    \lambda_{n,m,\ell}^* + O(n^{-3})\notag\\
    & \qquad = \frac{1}{m} \sum_{1 \le \ell \le m}
    \ell \lambda_{n,m,\ell}^*
    + \frac{\dbbracket{m \ge 2}}{2m(m-1)}
    \sum_{1 \le \ell \le m} \ell(\ell-1)
    \lambda_{n,m,\ell}^* + O(n^{-3})\notag\\
    &\qquad =\frac{S_{1}(\alpha)}{\alpha n}
    +\frac{1}{2n^{2}} \left(\frac{U_1(\alpha)}{\alpha} +
    \frac{S_{2}(\alpha)-S_{1}(\alpha)}{\alpha^{2}}\right)
    - \frac{\dbbracket{m=1}}{2n^{2}} + O(n^{-3}),
    \label{eqn:harmonic_expansion}
\end{align}
which holds uniformly for $1 \le m \le n$. Note that for $m=1$ a
correction term is needed; more correction terms have to be
introduced in more refined expansions (see Section~\ref{sec:mu-ae}).

Combining the estimates \eqref{eqn:phi_expansion} (with
$\phi=\phi_1, \phi_2$) and \eqref{eqn:harmonic_expansion}, we see
that
\begin{equation*}
    E_{2}(n,m) = \frac{J_{1}(\alpha)}{n}
    + \frac{J_{2}(\alpha)}{2n^{2}}
    + \frac{\dbbracket{m=1}}{2n^{2}} + O(n^{-3}),
\end{equation*}
uniformly for $1\le m\le n$, where
\begin{align*}
    J_{1}(z) & = 1 - \frac{S_{1}(z)}{z} - \phi_{1}'(z) S_{1}(z),\\
    J_{2}(z) & = - \frac{S_{2}(z)-S_{1}(z)}{z^{2}}
    -\frac{2b_{1}}{z}\,S_{1}(z)
    -\left(\frac{2}{z}+2\phi_1'(z)\right)U_1(z)\\
    & \quad \mbox{} + \phi_{1}''(z) S_{2}(z) - 2 \phi_{2}'(z) S_{1}(z).
\end{align*}
Obviously, $J_{1}(z) = 0$ because $\phi_{1}'(z) =
\frac{1}{S_{1}(z)}-\frac{1}{z}$. To determine $b_1$ and $\phi_2$, we
observe that
\[
    \lim_{z\to0} J_2(z) =
    \lim_{z\to0} \left(- \frac{S_{2}(z)-S_{1}(z)}{z^{2}}
    -\frac{2b_{1}}{z}\,S_{1}(z)
    -\frac{2}{z}\,U_1(z)\right) = 2b_1-2,
\]
where we used the relation $U_1(\alpha)= -S_0(\alpha)-\frac12
S_1(\alpha)$ (see \eqref{Ura}). In order that $E_2=o(n^{-2})$
uniformly for $1\le m\le n$, we need $2b_1-2=0$, so that $b_1=1$.

Now the equation $J_{2}(z)=0$ also implies, by \eqref{Ura}, that
\begin{equation}\label{phi-2-diff}
    \phi_{2}'(z) =
    - \frac{S_{1}'(z) S_{2}(z)}{2 S_{1}^{3}(z)}
    +\frac{S_0(z)}{S_1(z)^2}+\frac1{2S_1(z)}
    + \frac{1}{2z^{2}} - \frac{1}{z}.
\end{equation}

With these choices of $b_{1}$ and $\phi_{2}(z)$, we have
\begin{equation*}
    E_{2}(n,m) = \frac{\dbbracket{m=1}}{2n^{2}} + O(n^{-3}),
\end{equation*}
uniformly for $1 \le m \le n$.

The exact solution to the differential equation \eqref{phi-2-diff}
requires the constant term $\phi_{2}(0)$, which we have not yet 
specified. To specify this value, we take $m=1$ in
\eqref{eqn:Delta_bound} and then obtain, by the
recurrence~\eqref{eqn:recDelta},
\begin{equation*}
    \Delta_{n,1}^* = \Delta_{n,0}^* +
    \frac{E_{2}(n,1)}{\thickbar{\Lambda}_{n,1}}
    = -\frac{\phi_{2}(0)}{n} + n E_{2}(n,1)
    = -\frac{\phi_{2}(0)}{n} + \frac{1}{2n} + O(n^{-2}).
\end{equation*}
This entails the choice $\phi_{2}(0)=\frac{1}{2}$ in order that
$\Delta_{n,1}^* = O(n^{-2})$. Thus we obtain the integral solution
\eqref{phi-2} for $\phi_2(z)$. In particular, the first few terms of
$\phi_{2}(z)$ in the Taylor expansion are given as follows.
\begin{align*}
    \phi_2(z) = \tfrac12-\tfrac74 z
    +\tfrac{23}{18}z^{2}
    -\tfrac{19951}{17280}z^{3}
    +\tfrac{64903}{57600}z^{4}
    -\tfrac{13803863}{12096000}z^{5}
    +\cdots.
\end{align*}
As a function in the complex plane, the region where $\phi_2(z)$ is
analytic is dictated by the first zeros of $S_1(z)$, which exceeds
unity.

To complete the proof of \eqref{eqn:Delta_bound}, we require a
variation of Lemma~\ref{lmm-at}, since the assumption on $a_{n,0}$
given there is not satisfied here.
\begin{lmm} \label{lmm:rec_variation}
Consider the recurrence \label{lmm-at-variation}
\[
    \sum_{1\le \ell \le m} \lambda_{n,m,\ell}^* \,
    (a_{n,m}-a_{n,m-\ell}) = b_{n,m}\qquad (m\ge1),
\]
where $b_{n,m}$ is defined for $1\le m\le n$ and $n\ge1$. Assume
that $|a_{n,0}| \le c n$ for $n \ge 1$, and $|a_{n,1}| \le 2c$ for
$n \ge 1$. If there exists a $c>0$ such that $|b_{n,m}|\le \frac cn$
holds uniformly for $2\le m\le n$ and $n\ge1$, then
\begin{align}\label{anm-2cH}
    |a_{n,m}| \le 2c H_m \qquad(1 \le m\le n).
\end{align}
\end{lmm}
\pf The inequality \eqref{anm-2cH} holds when $m=1$ by assumption.
For $m \ge 2$, we write the recurrence as follows
\begin{align*}
    a_{n,m}
    = \frac{1}{\thickbar{\Lambda}_{n,m}}
    \sum_{1 \le \ell <m} \lambda_{n,m,\ell}^* \, a_{n,m-\ell} +
    \frac{\lambda_{n,m,m}^* a_{n,0}}{\thickbar{\Lambda}_{n,m}} +
    \frac{b_{n,m}}{\thickbar{\Lambda}_{n,m}}.
\end{align*}
By induction hypothesis and the two inequalities (see
Lemma~\ref{lmm-at})
\begin{equation*}
    \thickbar{\Lambda}_{n,m} \ge \frac{m}{n},\quad
    \text{and}\quad
    \lambda_{n,m,m}^* = n^{-m} \le n^{-2},
\end{equation*}
we obtain
\begin{equation*}
    |a_{n,m}| \le 2c H_{m-1} + \frac{n}{m}
    \cdot \frac{1}{n^{2}} \cdot cn + \frac{n}{m}
    \cdot \frac{c}{n} \le 2c H_{m},
\end{equation*}
and this proves the lemma. \qed

In view of the estimates $\Delta_{n,0}^*=O(n^{-1})$,
$\Delta_{n,1}^*=O(n^{-2})$ and $E_{2}(n,m)=O(n^{-3})$, for $2 \le m
\le n$, there exists a constant $c>0$ such that the quantity
$\Delta_{n,m} := n^{2} \Delta_{n,m}^*$ satisfies the assumptions of
Lemma~\ref{lmm:rec_variation}, which implies the bound $\Delta_{n,m}
= O(H_{m})$, or, equivalently $\Delta_{n,m}^* = O(n^{-2} H_{m})$,
uniformly for $1 \le m \le n$. This completes the proof of
Theorem~\ref{thm:mu}. \qed

\subsection{An asymptotic expansion for the mean}
\label{sec:mu-ae}

The above procedure can be extended to get more smaller-order terms,
but the expressions for the coefficients soon become very involved.
However, it follows from the discussions in \S~\ref{sec:ae-small-m}
that we expect the asymptotic expansion
\begin{equation}\label{mu-n-star-ae}
    \mu_{n,m}^* \sim \sum_{k \ge 0}\frac{b_k H_m +
    \phi_{k+1}(\alpha)}{n^k},
\end{equation}
in the sense that the truncated asymptotic expansion
\begin{align} \label{mu-truncated}
    \mu_{n,m}^* = \sum_{0\le k\le K}\frac{b_k H_m +
    \phi_{k+1}(\alpha)}{n^k} +O\left(n^{-K-1}H_m\right)
\end{align}
holds uniformly for $K\le m\le n$ and introduces an error of order
$n^{-K-1}H_m$. This asymptotic approximation may not hold when $1\le
m<K$ because additional correction terms are needed in that case.
Technically, the correction terms stem from asymptotic expansions
for sums of the form $\sum_{1 \le \ell \le m} \lambda_{n,m}^*
(H_{m}-H_{m-\ell})$; see \eqref{eqn:harmonic_expansion} and the
comments given there.

We propose here an easily codable procedure for the coefficients in
the expansion, whose justification follows the same error analysis
as above. We start with the formal expansion \eqref{mu-n-star-ae}
and expand in all terms for large $m=\alpha n$ in decreasing powers
of $n$, match the coefficients of $n^{-K-1}$ on both sides for each
$K \ge 0$, and then adjust the initial condition $\phi_{K+1}(0)$ by
taking into account the extremal case when $m=K$ (for $m<K$ the
expansion up to that order may not hold). With this algorithmic
approach it is possible to determine the coefficients $b_{K}$ and
the functions $\phi_{K+1}(z)$ successively one after another.

Observe first that
\begin{align*}
    H_m-H_{m-\ell} = \sum_{0\le j<\ell}\frac1{m-j}
    = \sum_{r\ge1}m^{-r} \beta_r(\ell)
    = \sum_{r\ge1}n^{-r} \alpha^{-r}\beta_r(\ell) ,
\end{align*}
where ($0^0=1$)
\[
    \beta_r(\ell) := \sum_{0\le j<\ell} j^{r-1}
    = \frac1r\sum_{0\le j<r} \binom{r}{j}
    B_j \ell^{r-j},
\]
the $B_j$ representing the Bernoulli numbers. On the other hand,
\begin{align*}
    \phi_{k+1}(\alpha)-
    \phi_{k+1}\left(\alpha-\frac {\ell}n\right)
    = -\sum_{r\ge1}\frac{\phi_{k+1}^{(r)}(\alpha)}{r!}\,
    \left(-\frac{\ell}n\right)^r.
\end{align*}
Thus
\begin{align*}
    \mu_{n,m}^*  - \mu_{n,m-\ell}^*
    &\sim \sum_{r\ge1}n^{-r}\sum_{1\le j\le r}
    \left(\frac{\beta_j(\ell)b_{r-j}}{\alpha^j (r-j)!}
    - \frac{(-\ell)^j}{j!}\,\phi_{r-j+1}^{(j)}(\alpha) \right).
\end{align*}
Then, by \eqref{lnml},
\[
    \sum_{1\le \ell \le m}\lambda_{n,m,\ell}^*
    \left(\mu_{n,m}^*  - \mu_{n,m-\ell}^*\right)
    \sim \sum_{r\ge1} n^{-r}
    [t^{-1}] \left(1+\frac1{nt}\right)^m
    \left(1+\frac tn\right)^{n+1-m} f_r(t),
\]
where
\[
    f_r(t) := \sum_{\ell\ge 1} t^{\ell-1}\sum_{1\le j\le r}
    \left(\frac{\beta_j(\ell)b_{r-j}}{\alpha^j (r-j)!}
    - \frac{(-\ell)^j}{j!}\,\phi_{r-j+1}^{(j)}(\alpha) \right).
\]
Now
\[
    \left(1+\frac1{nt}\right)^m
    \left(1+\frac tn\right)^{n+1-m}
    = \exp\left(\sum_{j\ge1} \frac{(-1)^{j-1}}{j}
    \left(\frac{\alpha t^{-j}+(1-\alpha)t^j}{n^{j-1}}
    +\frac{t^j}{n^j}\right)\right).
\]
A direct expansion using Bell polynomials $B_{k}^*(t_1,\dots,t_k)$
(see \cite{Comtet74}) then gives
\begin{align*}
    \left(1+\frac1{nt}\right)^m
    \left(1+\frac tn\right)^{n+1-m}
    &=e^{\frac\alpha t+(1-\alpha)t}
    \sum_{k\ge0} \frac{B_{k}^*(t_1,\dots,t_k)}
    {k!}\, n^{-k}\\
    &= e^{\frac\alpha t+(1-\alpha)t}
    \sum_{k\ge0} \frac{\tilde{B}_{k}(\mathbf{t})}
    {k!t^{2k}}\, n^{-k}
\end{align*}
where $\tilde{B}_{0}=1$,
\[
    t_j := \frac{(-1)^jj!}{j+1}\left(
    \frac{\alpha}{t^{j+1}} + (1-\alpha) t^{j+1}\right)
    +(-1)^{j-1}(j-1)!t^j\qquad(j=1,2,\dots),
\]
and $\tilde{B}_k(\mathbf{t})$ is a polynomial of degree $4k$.

Collecting these expansions, we get
\begin{align*}
    \sum_{1\le \ell \le m}\lambda_{n,m,\ell}^*
    \left(\mu_{n,m}^*  - \mu_{n,m-\ell}^*\right)
    \sim \sum_{K \ge 0}n^{-(K+1)} \sum_{0 \le r \le K}
    [t^{-1}]e^{\frac\alpha t+(1-\alpha)t}
    \frac{\tilde{B}_r(\mathbf{t})}{r!}\,f_{K+1-r}(t).
\end{align*}

All terms now have the form
\begin{align*}
    [t^{2r-1}]e^{\frac\alpha t+(1-\alpha)t} F(t)
    &= \sum_{\ell\ge0}\frac{\alpha^\ell}{\ell!}
    [t^{\ell+2r-1}]e^{(1-\alpha)t}F(t) \\
    &=\sum_{\ell\ge0}\frac{\alpha^\ell}{\ell!}
    \sum_{0\le j<2r+\ell} \frac{(1-\alpha)^j}{j!}
    \cdot \frac{F^{(\ell+2r-1-j)}(0)
    }{(\ell+2r-1-j)!}.
\end{align*}
Since we are solving the recurrence
\[
    \sum_{1\le \ell \le m}\lambda_{n,m,\ell}^*
    \left(\mu_{n,m}^*  - \mu_{n,m-\ell}^*\right)
    =\frac1n,
\]
we have the relations
\[
    \begin{cases} \displaystyle
        \left(\frac{b_0}{\alpha}+\phi_1'(\alpha)\right)
        [t^{-1}]\frac{e^{\frac\alpha t+(1-\alpha)t}}
        {(1-t)^2} = 1,\\
    \displaystyle
        \sum_{0 \le r \le K}
        [t^{2r-1}]e^{\frac\alpha t+(1-\alpha)t}
        \frac{\tilde{B}_r(\mathbf{t})}{r!}\,f_{K+1-r}(t)
        = 0, \quad (K\ge1).
    \end{cases}
\]
By induction, each $\phi_{K+1}$ satisfies a differential equation of
the form
\[
    \left(\frac{b_{K}}{\alpha}
    +\phi_{K+1}'(\alpha)\right)S_K(\alpha)
    = \Psi_{K+1}[\phi_1,\dots,\phi_{K}](\alpha),
\]
for some functional $\Psi_{K+1}$. Since $S_1(\alpha)\sim \alpha$ as
$\alpha \to0$, we also have the relation
\[
    b_K=\Psi_{K+1}[\phi_1,\dots,\phi_{K}](0).
\]
Once the value of $b_{K}$ is determined, we can then write
\begin{equation*}
    \phi_{K+1}(\alpha) = \phi_{K+1}(0)+
    \int_0^\alpha \left(\frac{\Psi_{K+1}[\phi_1,\dots,\phi_{K}](x)}
    {S_{1}(x)} - \frac{b_{K}}{x}\right)\dd x,
\end{equation*}
and it remains to determine the initial value $\phi_{K+1}(0)$, which
is far from being obvious. The crucial property we need is that the
truncated expansion \eqref{mu-truncated} holds when $K\le m\le n$,
and particularly when $m=K$. So we compute \eqref{mu-truncated} with
$m=K$ and drop all terms of order smaller than or equal to
$n^{-K-1}$. Then we match the coefficient of $n^{-K}$ with that in
the expansion of $\mu_{n,K}^*$ obtained by a direct calculation from
the recurrence \eqref{mu-rr}.

We illustrate this procedure by computing the first two terms in
\eqref{mu-n-star-ae}. First, we have
\[
    \left(\frac{b_0}{\alpha}+\phi_1'(\alpha)\right) S_1(\alpha)
    =1,
\]
which implies $b_0=1$ and $\phi_1'(\alpha)=\frac1{S_1(\alpha)}-
\frac 1\alpha$. Moreover, substituting the initial value $m=K=0$, we
get
\begin{equation*}
    0 = \mu_{n,0}^* = H_{0} + \phi_{1}(0) + O(n^{-1})
    = \phi_{1}(0) + O(n^{-1}),
\end{equation*}
entailing $\phi_{1}(0)=0$, which is consistent with what we obtained
above.

The next-order term when $K=1$ is (after substituting the relations
$b_0=1$, $\phi_1'(\alpha)=\frac1{S_1(\alpha)}-\frac 1\alpha$ and
$\phi_1''(\alpha)=\frac1{\alpha^2}-\frac{S_1'(\alpha)}
{S_1(\alpha)^2}$)
\begin{align*}
    &\left(\frac{b_1}{\alpha}+\phi_2'(\alpha)\right)
    [t^{-1}]\frac{e^{\frac \alpha t+(1-\alpha)t}}
    {(1-t)^2} \\
    &\qquad = [t^{-1}]e^{\frac \alpha t+(1-\alpha)t}
    \left(\frac1{2\alpha^2(1-t)^2}
    -\frac{(1+t)S_1'(\alpha)}{2(1-t)^3S_1(\alpha)}
    +\frac{\alpha-2t^3+(1-\alpha)t^4}{2
    t^2(1-t)^2S_1(\alpha)}\right),
\end{align*}
implying that
\begin{align}\label{phi2-de}
    \left(\frac{b_1}{\alpha}+\phi_2'(\alpha)\right)S_1(\alpha)
    = - \frac{S_{1}'(\alpha) S_{2}(\alpha)}{2 S_{1}(\alpha)^2}
    +\frac{S_0(\alpha)}{S_1(\alpha)}+\frac1{2}
    + \frac{S_1(\alpha)}{2\alpha^{2}} .
\end{align}
As $\alpha\to0$, the right-hand side of \eqref{phi2-de} has the
local expansion $1-\frac14\alpha+ \cdots$, forcing $b_1=1$, and,
accordingly, we obtain the same differential equation
\eqref{phi-2-diff}. Substituting the value $m=K=1$ in
\eqref{mu-truncated} yields
\begin{align*}
    1 = \mu_{n,1}^* &= H_{1} + \phi_{1}\Big(\frac{1}{n}\Big)
    + \frac{1}{n}\left(H_{1} + \phi_{2}\Big(\frac{1}{n}\Big)\right)
    + O(n^{-2})\\
    &= 1 + \frac{\phi_{1}'(0)}{n}  + \frac{1}{n}
    + \frac{\phi_{2}(0)}{n}  + O(n^{-2}),
\end{align*}
implying, by using \eqref{phi_1-z}, $\phi_{2}(0) = -\phi_{1}'(0) - 1
= \frac12$, which is consistent with Theorem~\ref{thm:mu}.

Although the expressions become rather involved for higher-order
terms, all calculations (symbolic or numerical) are easily coded.
For example, we have
\begin{equation*}
    \phi_{3}(z) = \tfrac{1}{12} -
    \tfrac{575}{432}z + \tfrac{15101}{11520} z^{2}
    - \tfrac{8827}{5400} z^{3}
    + \tfrac{2229089}{1036800} z^{4}
    - \tfrac{361022171}{127008000} z^{5} + \cdots.
\end{equation*}

\subsection{Proof of Theorem~\ref{thm:EYn}}\label{sec:Proof_thm_EYn}

We now give an outline of the proof of Theorem~\ref{thm:EYn}
concerning the asymptotics of $\mathbb{E}(X_n)$. The method of proof
relies on standard normal approximation to the binomial
distribution.

We begin with
\[
    \mathbb{E}\left(t^{X_n}\right) = \sum_{0\le m\le n}
    \pi_{n,m} \, P_{n,m}(t),
\]
where $\pi_{n,m} := \binom{n}{m} \bar{\rho}^{n-m} \rho^m$
($\bar{\rho} := 1-\rho$). From this expression, we see that
\begin{equation*}
    \mathbb{E}(X_{n})
    = \sum_{0 \le m \le n} \pi_{n,m} \, \mu_{n,m}
    = \frac{n-1}{\left(1-\frac{1}{n}\right)^{n}}
    \sum_{0 \le m \le n} \pi_{n,m}  \mu_{n-1,m}^*.
\end{equation*}
Write $m=\rho n + x \sqrt{\rho\bar{\rho}n}$. By Stirling's
formula~\eqref{eqn:StirForm}, we have
\[
    \pi_{n,m} =
    \frac{e^{-x^2/2}}{\sqrt{2\pi \rho\bar{\rho}n}}\left(1+
    \frac{p_1(x)}{\sqrt{\rho\bar{\rho}n}}
    +\frac{p_2(x)}{\rho\bar{\rho}n}
    + \frac{p_3(x)}{(\rho\bar{\rho}n)^{3/2}}+
    O\left(\frac{1+x^{12}}{n^2}\right)\right),
\]
uniformly for $x=o(n^{\frac16})$, where, here and throughout the
proof, the $p_j$ are polynomials of $x$ containing only powers of
the same parity as $j$. On the other hand, by Theorem~\ref{thm:mu},
we have in the same range of $m$
\begin{align*}
    \mu_{n-1,m}^* &= \log \rho n + \gamma+\phi_1(\rho) +
    \frac{p_5(x)}{\sqrt{\rho\bar{\rho}n}}
    +\frac{2 \rho\bar{\rho} (\log \rho n + \gamma)
    +p_4(x)}{2 \rho\bar{\rho}n} \\
    &\qquad + \frac{p_7(x)}{(\rho\bar{\rho}n)^{3/2}}+
    O\left(\frac{\log n+x^{4}}{n^2}\right).
\end{align*}
With these expansions, the asymptotic evaluation of
$\mathbb{E}(X_{n})$ is reduced to sums of the form
\begin{equation*}
    \frac{1}{\sqrt{2\pi \rho\bar{\rho}n}}
    \sum_{x=\frac{m-\rho n}{\sqrt{\rho\bar{\rho}n}}
    = o(n^{\frac16})} x^{r} \, e^{-\frac{x^{2}}{2}}
    = \frac{1}{\sqrt{2 \pi}} \int_{-\infty}^{\infty}
    x^{r} e^{-\frac{x^{2}}{2}} \dd x + O(n^{-L}),
\end{equation*}
for any $L>1$ by an application of the Euler-Maclaurin formula. Thus
polynomials of odd indices (containing only odd powers of $x$)  will
lead to asymptotically negligible terms after integration. Outside
the range where $x=o(n^{\frac16})$, the binomial distribution is
smaller than any negative power of $n$, so the contribution from
this range is also asymptotically negligible. Except for this part,
all other steps are easily coded. \qed

\section{Asymptotics of the variance}
\label{sec:Var}

We prove in this section that the variance $\sigma_{n,m}^2 :=
\mathbb{V}(X_{n,m}) =
\mathbb{E}(X_{n,m}^2)-(\mathbb{E}(X_{n,m}))^{2}$ of $X_{n,m}$ is
asymptotically quadratic.
\begin{thm} For $1\le m\le n$, the variance of $X_{n,m}$
satisfies \label{thm:var1}
\begin{align*}
    \frac{\mathbb{V}(X_{n,m})}{en} &=
    eH_m^{(2)}n -(2e+1)H_m
    +eH_m^{(2)}+e \psi_1(\alpha) - \phi_{1}(\alpha)
    -\frac{11e+1}{2n}\,H_m\\
    &\quad  +\frac{5eH_m^{(2)} + 2e \psi_{2}(\alpha)
    - 2\phi_{2}(\alpha) + 2e \alpha \psi_{1}'(\alpha)
    -2 \alpha \phi_{1}'(\alpha) +\phi_{1}(\alpha)}{2n}
    + O\left(n^{-2}H_m\right),
\end{align*}
where
\begin{align}\label{psi-1}
    \psi_1(\alpha) = \int_0^\alpha
    \left(\frac{S_2(x)}{S_1(x)^3}- \frac{1}{x^2} + \frac{2}{x}\right)
    \dd x,
\end{align}
and
\begin{align*}
\begin{split}
    \psi_2(\alpha) &= \frac{7}{12}
    - \int_0^\alpha\left( \frac{5S_1'(x)S_2(x)^2}{2S_1(x)^5}
    -\frac{2S_1'(x)S_3(x) +S_2(x) S_2'(x)+6S_0(x) S_2(x)}{2S_1(x)^4}
    \right.\\
    &\hspace*{2.5cm}
    \left.-\frac{S_0(x)}{S_1(x)^3} +\frac{2}{S_1(x)^2}
    -\frac{1}{x^3}+\frac{3}{x^2} -\frac{11}{2x}\right)\dd x.
\end{split}
\end{align*}
\end{thm}
Similar to the mean, we  work on the sequence $V_{n,m}^* := e_n^2
(\sigma_{n+1,m}^2 + \mu_{n+1,m})/n^2$ and prove that (see
Figure~\ref{fig:var-diff} and Appendix D)
\begin{align} \label{Vnm-star}
    V_{n,m}^* & = H_{m}^{(2)} + \frac{-2 H_{m}
    + \psi_{1}(\alpha) + 2H_{m}^{(2)}}{n}
    + \frac{-\frac{11}{2} H_{m} + \psi_{2}(\alpha)
    + \frac{7}{3} H_{m}^{(2)}}{n^{2}}\\
    & \qquad  + O(n^{-3} H_{m})
    \qquad (2 \le m \le n).\notag
\end{align}
\vspace*{-.5cm}
\begin{figure}[!ht]
\begin{center}
\includegraphics[width=6cm]{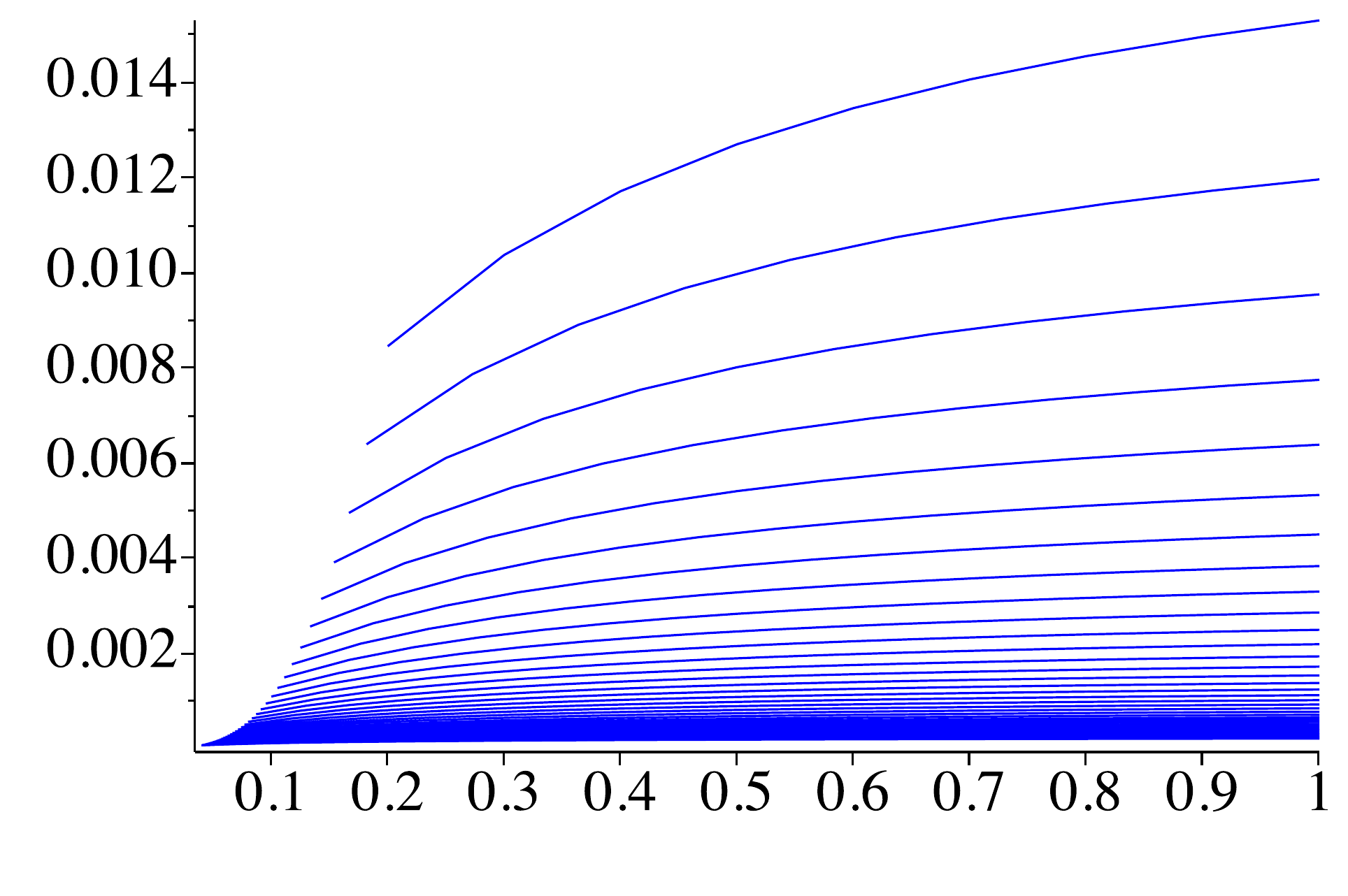}\;\;
\includegraphics[width=6cm]{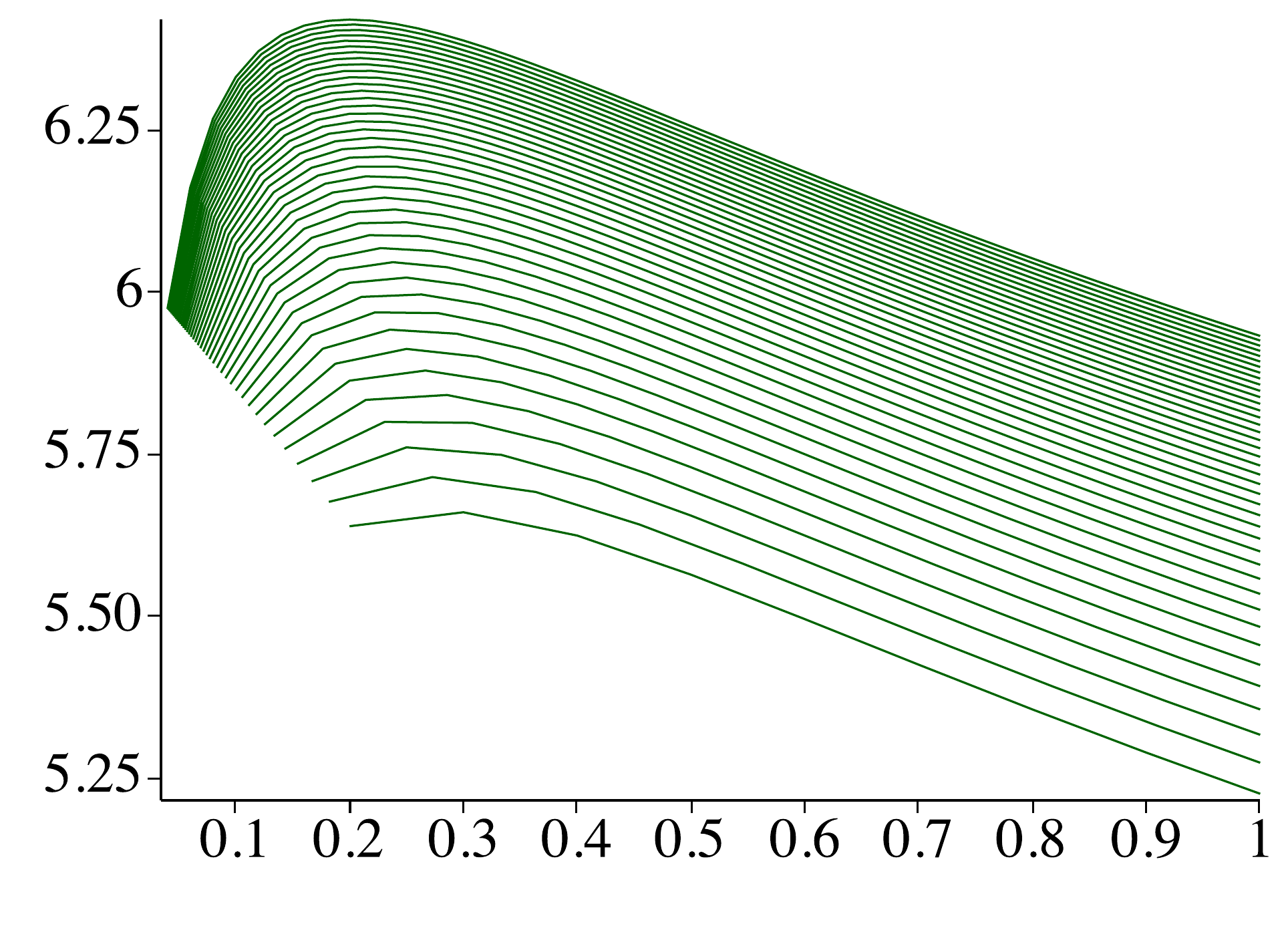}
\end{center}
\vspace*{-.5cm} \caption{\emph{The absolute differences
$|V_{n,m}^*-$ RHS of
\protect{\eqref{Vnm-star}}$|$ for $2\le m\le n$
(normalized to the unit interval) and $n=10,\dots,50$ (left in
top-down order), and the absolute normalized differences
$n^3H_m^{-1}|V_{n,m}^*-$ RHS of \protect{\eqref{Vnm-star}}$|$ for
$n=10,\dots,50$ (right).}} \label{fig:var-diff}
\end{figure}

The variance of $X_n$ is computed by the relation
\[
    \mathbb{V}(X_n) = \sum_{0\le m\le n}
    \pi_{n,m}\left(\sigma_{n,m}^2
    + \mu_{n,m}^2\right)-\left(
    \sum_{0\le m\le n}\pi_{n,m} \mu_{n,m} \right)^2,
\]
where $\pi_{n,m}=\binom{n}{m}\rho^m (1-\rho)^{n-m}$, $\mu_{n,m} :=
\mathbb{E}(X_{n,m})$ and $\sigma_{n,m}^2 := \mathbb{V}(X_{n,m})$.

\begin{thm} The variance of $X_n$ satisfies asymptotically
\label{thm:var2}
\begin{align*}
    \frac{\mathbb{V}(X_n)}{en}
    &= \frac{\pi^2}{6}\,en
    -(2e+1)(\log \rho n+\gamma)+v_1 \\
    &\qquad -\frac{(11e+1)(\log \rho n+\gamma)-v_2}
    {2n}+O\left(n^{-2}\log n\right),
\end{align*}
where ($\bar{\rho} := 1-\rho$)
\[
    v_1 := e\left(\frac{\pi^2}{6}-1\right)
    +e\psi_1(\rho)-\phi_{1}(\rho)
    +2e\bar{\rho}\phi_1'(\rho)+e\bar{\rho}\rho\phi_1'(\rho)^2,
\]
and
\begin{align*}
    v_2 &= e\bar{\rho}^2\rho^2\phi_1''(\rho)^2
    +2e\bar{\rho}^2\rho(1+\rho\phi_1'(\rho))\phi_1'''(\rho)
    +2e\rho\bar{\rho}(1+\phi_1'(\rho))\phi_1''(\rho)\\
    &\qquad +4e\bar{\rho}(1+\rho\phi_2'(\rho))\phi_1'(\rho)
    +2e\bar{\rho}\rho\phi_1'(\rho)^2+4e\bar{\rho}\phi_2'(\rho)
    +e\bar{\rho}\rho\psi_1''(\rho)-\bar{\rho}\rho\phi_1''(\rho)\\
    &\qquad +2e\psi_2(\rho)-2\phi_2(\rho)+2e\rho\psi_1'(\rho)
    -2\rho\phi_1'(\rho)+\phi_1(\rho)+\tfrac5{6}\,e\pi^2-3e-1.
\end{align*}
\end{thm}

\paragraph{Recurrences for the centered moment generating function.}
To compute the variance, one may start with the second moment and
then consider the difference with the square of the mean; however,
it is computationally more advantageous to study directly the
recurrence satisfied by the variances themselves.

From \eqref{Qnmt}, we have, by substituting $t=e^y$,
\[
    \left(1-\left(1-\Lambda_{n,m}\right)e^y \right)P_{n,m}(e^y)
    = e^y \sum_{1\le \ell \le m}
    \lambda_{n,m,\ell}P_{n,m-\ell}(e^y),
\]
which can be rewritten as
\[
    \left(e^{-y}-1-\Lambda_{n,m} \right)P_{n,m}(e^y)
    = \sum_{1\le \ell \le m}
    \lambda_{n,m,\ell}P_{n,m-\ell}(e^y),
\]
or
\[
    \sum_{1\le \ell \le m}
    \lambda_{n,m,\ell} \bigl(P_{n,m}(e^y)
    - P_{n,m-\ell}(e^y)\bigr)
    = (1-e^{-y}) P_{n,m}(e^y).
\]
This is a simpler recurrence to start as fewer terms are involved
for the moments.

We now consider the moment generating function for the centered
random variables $X_{n,m}-\mu_{n,m}$
\[
    R_{n,m}(y) := P_n(e^y) e^{-\mu_{n,m}y},
\]
which then satisfies the recurrence
\[
    \sum_{1\le \ell \le m}
    \lambda_{n,m,\ell} \bigl(R_{n,m}(y)
    - R_{n,m-\ell}(y) e^{-(\mu_{n,m}-\mu_{n,m-\ell})y}\bigr)
    = (1-e^{-y}) R_{n,m}(y),
\]
for $1\le m\le n$ with $R_{n,0}(y)=1$.

\paragraph{Variance.}
Let $\sigma_{n,m}^2 = \mathbb{V}(X_{n,m})=R_{n,m}''(0)$ be the
variance of $X_{n,m}$. Then $\sigma_{n,m}^2$ satisfies the
recurrence
\begin{equation*}
    \sum_{1\le \ell \le m} \lambda_{n,m,\ell}
    \left(\sigma_{n,m}^2-\sigma_{n,m-\ell}^2\right)
    = -1+ \sum_{1\le \ell \le m}
    \lambda_{n,m,\ell} \,
    \left(\mu_{n,m}-\mu_{n,m-\ell}\right)^2.
\end{equation*}
In terms of $V_{n,m}^*:= e_n^2(\sigma_{n+1,m}^2 + \mu_{n+1,m})/n^2$,
we have $V_{n,0}^*=0$, and for $1 \le m \le n$
\begin{equation}\label{eqn:Vnm_rec}
    \sum_{1\le \ell \le m} \lambda_{n,m,\ell}^*
    \left(V_{n,m}^*-V_{n,m-\ell}^*\right) = T_{n,m}^*,
\end{equation}
where
\begin{equation}\label{eqn:Tnm}
    T_{n,m}^* := \sum_{1\le \ell \le m} \lambda^*_{n,m,\ell}
    \left(\mu_{n,m}^*-\mu_{n,m-\ell}^*\right)^2.
\end{equation}
In particular, this gives
\begin{align*}
    V_{n,1}^* &= 1,\\
    V_{n,2}^* &= \frac{5n^4+8n^3-n^2-4n+1}{(2n^2+2n-1)^2}.
\end{align*}
The expressions become very lengthy as $m$ increases. In Appendix~C
we give asymptotic expansions for $V_{n,m}^*$ for a few small $m$ as
$n \to \infty$. Based on these expansions, a suitable Ansatz for the
asymptotic behavior of $V_{n,m}^*$ can be deduced (assisted again by
computer algebra system), which then can be proven analogous to the
method of proof presented in Section~\ref{sec:Exp}.

\paragraph{Proof of asymptotics with error analysis.} By the same
procedure used for $\mu_{n,m}^*$, we start from computing the
asymptotic expansions for $V_{n,m}^*$ for small $m$. These
expansions suggest the more uniform (for $1\le m\le n$) asymptotic
expansion
\[
    V_{n,m}^* \sim c_{0} H_{m}^{(2)} + \frac{a_{1} H_{m}
    + \psi_{1}(\alpha) + c_{1} H_{m}^{(2)}}{n},
\]
for some constants $c_{0}$, $c_{1}$ and $a_{1}$, and some function
$\psi_{1}(z)$. Such an asymptotic form can be justified by the same
approach we used above for $\mu_{n,m}^*$. More precisely, we now
prove that
\begin{equation}\label{eqn:Vnm_first}
    V_{n,m}^* = H_{m}^{(2)} + \frac{-2 H_{m}
    + \psi_{1}(\alpha) + 2 H_{m}^{(2)}}{n} + O(n^{-2} H_{m}),
\end{equation}
uniformly for $0 \le m \le n$ and $n \ge 1$, where $\psi_{1}(z)$ is
given in \eqref{psi-1}. Our proof start from considering the
difference
\begin{equation*}
    \Delta_{n,m}^* := V_{n,m}^* - c_{0} H_{m}^{(2)}
    - \frac{a_{1} H_{m} + \psi_{1}(\alpha) + c_{1} H_{m}^{(2)}}{n},
\end{equation*}
and specify the involved coefficients and $\psi_{1}(z)$ such that
$\Delta_{n,m}^* = O(n^{-2} H_{m})$. By \eqref{eqn:Vnm_rec},
$\Delta_{n,m}^*$ satisfies, for $1 \le m \le n$, the recurrence
\begin{equation}\label{V-Delta}
    \sum_{1 \le \ell \le m} \lambda_{n,m,\ell}^*
    \left(\Delta_{n,m}^* - \Delta_{n,m-\ell}^*\right)
    = \tilde{E}_{1}(n,m),
\end{equation}
with the initial value is $\Delta_{n,0}^* = -
\frac{\psi_{1}(0)}{n}$, where ($T_{n,m}^*$ being defined in
\eqref{eqn:Tnm})
\begin{align*}
\begin{split}
    \tilde{E}_{1}(n,m) & := T_{n,m}^*
    -  \sum_{1 \le \ell \le m} \lambda_{n,m,\ell}^*
    \biggl\{(c_{0}+\frac{c_1}{n})
    \left( H_{m}^{(2)} - H_{m-\ell}^{(2)}\right)
    +\frac{a_{1}}{n}\left( H_{m} - H_{m-\ell}\right)  \\
    & \hspace*{4cm}  +\frac{\psi_{1}\big(\frac{m}{n}\big) -
    \psi_{1}\big(\frac{m-\ell}{n}\big)}{n}\biggr\}.
\end{split}
\end{align*}

We will derive an asymptotic expansion for $\tilde{E}_{1}(n,m)$. For
that purpose, we use the expansions \eqref{eqn:phi_expansion},
\eqref{eqn:harmonic_expansion} as well as Theorem~\ref{thm:mu} in
Section~\ref{sec:Exp}, and apply the same error analysis used for
$\mu_{n,m}^*$. A careful analysis then leads to
\begin{align}\label{eqn:munm2_expansion}
    \sum_{1 \le \ell \le m} \lambda_{n,m,\ell}^*
    \left(\mu_{n,m}^* - \mu_{n,m-\ell}^*\right)^{2}
    &= \frac{1}{mn} + \frac{1}{n^{2}}
    \left(-\frac1\alpha +\frac{(1+\alpha \phi_1'(\alpha))^2}
    {\alpha^2}S_2(\alpha)\right)\\
    & \qquad + \frac{3}{2mn^{2}}
    + \frac{\dbbracket{m \ge 2}}{2m(m-1)n^{2}}
    - \frac{\dbbracket{m=1}}{n^{2}}+ O(n^{-3}),\notag
\end{align}
and
\begin{align}\label{eqn:harmonic2_expansion}
\begin{split}
    & \sum_{1 \le \ell \le m} \lambda_{n,m,\ell}^* \,
    \left(H_{m}^{(2)}-H_{m-\ell}^{(2)}\right)\\
    &\qquad = \frac{1}{m n} + \frac{1}{n^{2}}
    \left(\frac{S_{1}(\alpha)-\alpha}{\alpha^{2}}\right)
    - \frac{1}{2 m n^{2}} + \frac{\dbbracket{m \ge 2}}
    {2m(m-1)n^{2}} - \frac{\dbbracket{m=1}}{n^{2}}
    + O(n^{-3}),
\end{split}
\end{align}
both holding uniformly for $1 \le m \le n$ as $n \to \infty$.

Collecting the expansions \eqref{eqn:phi_expansion},
\eqref{eqn:harmonic_expansion}, \eqref{eqn:munm2_expansion} and
\eqref{eqn:harmonic2_expansion}, we obtain
\begin{align*}
\begin{split}
    \tilde{E}_{1}(n,m) & = \frac{1-c_{0}}{m n}
    + \frac{1}{n^{2}} \bigg\{-\frac1\alpha
    +\frac{(1+\alpha \phi_1'(\alpha))^2}
    {\alpha^2}S_2(\alpha) - \frac{c_{0}}{\alpha^2}
    \left(S_{1}(\alpha) - \alpha\right)\\
    & \qquad \quad - \left(\frac{a_{1}}{\alpha}+\psi_1'(\alpha)
    \right) S_{1}(\alpha)\bigg\}
    + \frac{1}{mn^2} \left(\frac{3}{2} + \frac{c_{0}}{2}
    - c_{1}\right)\\
    & \qquad - \frac{\dbbracket{m=1} (1-c_{0})}{n^{2}}
    + \frac{\dbbracket{m \ge 2} (1-c_{0})}{2m(m-1)n^{2}}
    +O(n^{-3}),
\end{split}
\end{align*}
uniformly for $1 \le m \le n$.

We can now specify all the undetermined constants and $\psi_1(z)$
such that all terms except the last will vanish and
$\tilde{E}_1(n,m)=O(n^{-3})$. This entails first the choices
$c_{0}=1$ and $c_1=2$.

It remains only the $\frac{1}{n^{2}}$-term. We consider the limit
when $\alpha$ tends to zero using the Taylor expansions
\eqref{eqn:SrTr_expansion}, and deduce that $a_1=-2$. These values
give the equation satisfied by $\psi_1'(z)$
\begin{equation*}
    \psi_1'(z)S_1(z) = -\frac{S_{1}(z)}{z^2}
    +\frac{(1+z \phi_1'(z))^2}{z^2}S_2(z) +\frac{2 S_{1}(z)}{z},
\end{equation*}
which in view of \eqref{phi1-z} leads to the differential equation
\begin{align}
    \psi_{1}'(z) & = \frac{S_{2}(z)}{S_{1}^{3}(z)}
    - \frac{1}{z^{2}} + \frac{2}{z}.
    \label{eqn:psi1diff}
\end{align}

Thus with the choices $c_{0}=1$, $a_{1}=-2$, $c_{1}=2$, and the
function $\psi_{1}'(z)$ by \eqref{eqn:psi1diff}, we get the bound
$\tilde{E}_{1}(n,m) = O(n^{-3})$ uniformly for $1 \le m \le n$.
Accordingly, by \eqref{V-Delta}, the sequences $\Delta_{n,m} :=
n^{2} \Delta_{n,m}^*$ satisfy the recurrence
\begin{equation*}
    \sum_{1 \le \ell \le m} \lambda_{n,m,\ell}^*
    \left(\Delta_{n,m} - \Delta_{n,m-\ell}\right)
    = O(n^{-1}) \qquad (1 \le m \le n),
\end{equation*}
with $\Delta_{n,0} = - \psi_{1}(0) n$. Choose now the initial value
$\psi_{1}(0) = 0$, so that $\Delta_{n,0}=0$ and Lemma~\ref{lmm-at}
can be applied. This implies that $\Delta_{n,m} = O(H_{m})$, and
consequently $\Delta_{n,m}^* = O(n^{-2} H_{m})$.

Also $\psi_{1}(z)$ is indeed given by \eqref{psi-1}. In particular,
the first few terms in the Taylor expansion of $\psi_{1}(z)$ are
given as follows.
\begin{small}
\begin{align*}
    \psi_{1}(z) & = \tfrac{11}{4} z - \tfrac{49}{36} z^{2}
    + \tfrac{2473}{4320} z^{3} + \tfrac{1307}{14400} z^{4}
    - \tfrac{12743687}{18144000} z^{5}
    + \tfrac{194960323}{152409600} z^{6} + \cdots.
\end{align*}
\end{small}
This completes the proof of the asymptotic expansion
\eqref{eqn:Vnm_first} for $V_{n,m}^*$. The more refined
approximation \eqref{Vnm-star} follows the same line of proof but
with more detailed expansions.

\section{Limit laws when $m\to\infty$}\label{sec:LL}

We show in this section that the distribution of $X_{n,m}$, when
properly normalized, tends to a Gumbel (or extreme-value or double
exponential) distribution, as $m\to\infty$, $m\le n$. The proof
consists in showing that the result \eqref{Y-mO1} when $m=O(1)$
extends to all $m\le n$ but requires an additional correction term
$\phi_1$ coming from the linear part of the random variables, which
complicates significantly the proof.

The standard Gumbel distribution $\mathscr{G}(1)$ (with mode zero,
mean $\gamma$) is characterized by the distribution function
$e^{-e^{-x}}$ and the moment generating function $\Gamma(1-s)$,
respectively. Note that if $X\sim \text{Exp}(1)$, then $-\log X
\sim\mathscr{G}(1)$, which was the description used in \cite{GKS99}.

The genesis of the Gumbel distribution is easily seen as follows.
\begin{lmm} Let $\eta_m := \sum_{1\le r\le m} \Exp(r)$, where the
$m$ exponential random variables are independent. Then $\eta_m-\log
m$ converges in distribution to the Gumbel distribution
\label{lmm-gumbel}
\[
    \mathbb{P}\left(\eta_m - \log m \le x\right)
    \to e^{-e^{-x}}\qquad(x>0; m\to\infty).
\]
\end{lmm}
\pf We have
\begin{align*}
    \mathbb{E}\left(e^{(\eta_m-H_m)s}\right)
    =\prod_{1\le r\le m} \frac{e^{-\frac sr}}{1-\frac sr}
    \to \prod_{r\ge 1} \frac{e^{-\frac sr}}{1-\frac sr}
    = e^{-\gamma s}\Gamma(1-s),
\end{align*}
uniformly for $|s|\le 1-\ve$. Here we used the infinite-product
representation of the Gamma function
\[
    \Gamma(1+s) = e^{-\gamma s} \prod_{r\ge1}
    \frac{e^{\frac sr}}{1+\frac sr} \qquad(s\in\mathbb{C}
    \setminus \mathbb{Z}^-).
\]
The lemma then follows from the asymptotic estimate
\[
    H_m = \log m + \gamma + O(m^{-1})\qquad
    (m\to\infty),
\]
and Curtiss's theorem (see \cite[\S 5.2.3]{HMC13}). \qed

Unlike the case when $m=O(1)$, we need to subtract more terms to
have the limit distribution.
\begin{prop} For $1\le m\le n$, we have the uniform asymptotic
approximation \label{prop-Y-all}
\[
    \mathbb{E}\left(e^{\frac{X_{n,m}}{en}\,s
    -(H_m+\phi_1(\frac mn))s}\right)
    = \left(1+O\left(\frac{H_m}{n}\right)\right)
    \prod_{1\le r\le m} \frac{e^{-\frac sr}}{1-\frac sr},
\]
for $|s|\le 1-\ve$, where $\phi_1$ is defined in \eqref{phi1-z}.
\end{prop}
Note that $\phi_1(x) = O(x)$ as $x\to0$, and thus
$\phi_1(\frac{m}{n})=o(1)$ when $m=O(1)$. In this case, the
proposition re-proves Theorem~\ref{thm:YO1} (with an explicit error
term).

A combination of Lemma~\ref{lmm-gumbel} and
Proposition~\ref{prop-Y-all} leads to the limit law for $X_{n,m}$ in
the remaining range.
\begin{thm} If $m \to \infty$ with $n$ and $m\le n$, then
\label{thm:Y-all}
\[
    \mathbb{P}\left(
    \frac{X_{n,m}}{en}-\log m-\phi_1(\tfrac mn)
    \le x\right) \to e^{-e^{-x}} \qquad(x>0),
\]
where $\phi_1$ is defined in \eqref{phi1-z}.
\end{thm}

\begin{thm} The number $X_n$ of steps used by the ($1+1$)-EA to reach
the final state $f(\mathbf{x})=n$, when starting from the initial
state $f(\mathbf{x})\sim \mathrm{Binom}(n;1-\rho)$, satisfies
\label{thm:Y}
\[
    \mathbb{P}\left(
    \frac{X_n}{en}-\log \rho n-\phi_1(\rho)
    \le x\right) \to e^{-e^{-x}} \qquad(x>0).
\]
\end{thm}
From Figure~\ref{fig:Yn}, we see the fast convergence of the
distribution to the limit law.
\begin{figure}[!ht]
\begin{center}
\includegraphics[width=8cm]{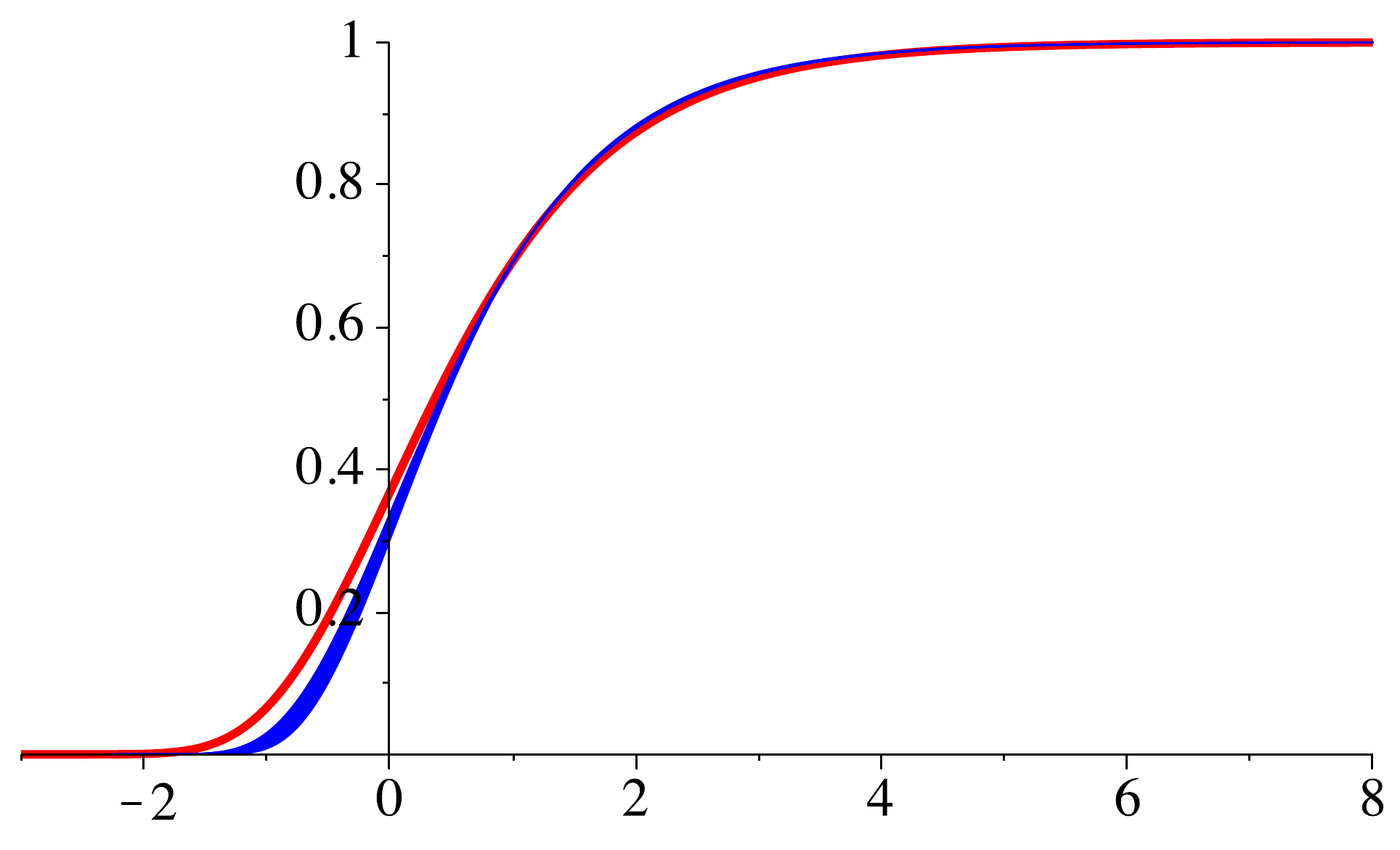}
\end{center}
\vspace*{-.5cm} \caption{\emph{Distributions of $\frac{X_n}{en}-\log
n -\log 2-\phi_1(\frac12) $ for $n=15,\dots,35$, and the limiting
Gumbel curve.}} \label{fig:Yn}
\end{figure}

\paragraph{Outline of proofs.}
We focus on the proof of Proposition~\ref{prop-Y-all} for which we
introduce the following normalized function
\begin{equation*}
    F_{n,m}(s) := \frac{
    \mathbb{E}\left(e^{\frac{X_{n,m}}{en}\,s}\right)
    e^{-\phi\left(\frac{m}{n}\right)s}}
    {\prod\limits_{1 \le r \le m}\frac{1}{1-\frac{s}{r}}}
    = \frac{ P_{n,m}\bigl(e^{\frac{s}{e n}}\bigr)
    e^{-H_{m} s - \phi\left(\frac{m}{n}\right)s}}
    {\prod\limits_{1 \le r \le m}
    \frac{e^{-\frac sr}}{1-\frac{s}{r}}}.
\end{equation*}
Here the probability generating function $P_{n,m}(t) :=
\mathbb{E}\left(t^{X_{n,m}}\right)$ of $X_{n,m}$ satisfies the
recurrence~\eqref{Qnmt} and the function $\phi(x)$ is any
$C^{2}[0,1]$-function with $\phi(0)=0$ (because in the proof we will
require a Taylor expansion of order two). It turns out that if we
choose $\phi(x)=\phi_1(x)$, where $\phi_1$ (see \eqref{phi1-z})
appears as the second-order term in the asymptotic expansion of the
mean (see \eqref{mu-asymp}), then
\begin{equation*}
    F_{n,m}(s) \sim 1,
\end{equation*}
uniformly for all $1 \le m \le n$, $n \to\infty$, and $|s| \le
1-\ve$, where $\ve > 0$ is independent of $m, n$. Indeed, our
induction proof here does not rely on any information of the mean
asymptotics and entails particularly the right choice of $\phi(x)$.
This is why we specify $\phi$ only at a later stage.

\paragraph{The recurrence satisfied by $F_{n,m}$.}
By \eqref{Qnmt}, $F_{n,m}(s)$ satisfies the following recurrence
\begin{align*}
    F_{n,m}(s) = \frac{e^{\frac{s}{e n}}
    \sum\limits_{1 \le \ell \le m} \lambda_{n,m,\ell}
    F_{n,m-\ell}(s)
    e^{-\big(\phi\left(\frac{m}{n}\right)
    -\phi\left(\frac{m-\ell}{n}\right)\big)s}
    \prod\limits_{m-\ell+1 \le r \le m}\left(1-\tfrac{s}{r}\right) }
    {1-\big(1-\Lambda_{n,m}\big) e^{\frac{s}{e n}}} ,
\end{align*}
for $1\le m \le n$, with $F_{n,0}(s)=1$, where $\Lambda_{n,m}
:=\sum_{1 \le \ell \le m} \lambda_{n,m,\ell}$.

\paragraph{An auxiliary sum.} Since we expect $F_{n,m}(s)$
to be close to $1$, we replace all occurrences of $F$ on the
right-hand side by $1$ and consider the following function
\begin{align*}
    G_{n,m}(s) &:= \frac{e^{\frac{s}{e n}}
    \sum\limits_{1 \le \ell \le m}
    \lambda_{n,m,\ell} e^{-\big(\phi\left(\frac{m}{n}\right)
    -\phi\left(\frac{m-\ell}{n}\right)\big)s}
    \prod\limits_{m-\ell+1 \le r \le m}\left(1-\frac{s}{r}\right)}
    {1-\big(1-\Lambda_{n,m}\big) e^{\frac{s}{e n}}} .
\end{align*}
The following lemma is the crucial step in our proof.
\begin{lmm} \label{lem:Gmns_general}
Let $\phi(x)$ be a $C^{2}$-function on the unit interval satisfying
$\phi(0)=0$. Then
\begin{equation}\label{Gnms}
    G_{n,m}(s) = \frac{1-\frac{s}{m}
    \left(1+\alpha\phi'(\alpha)\right)
    \frac{S_{1}(\alpha)}
    {S(\alpha)}
    + O\left(\frac{1}{m n}\right)}{1-\frac{s}{m}\cdot
    \frac{\alpha}{S(\alpha)}
    + O\left(\frac{1}{m n}\right)},
\end{equation}
where the $O$-terms hold uniformly for $1 \le m \le n$, and $|s|\le
1-\ve$.
\end{lmm}
\pf The proof consists in a detailed inspection of all factors,
using estimates \eqref{LS} and \eqref{L-ub} we derived earlier for
$\Lambda_{n,m}^{(r)}$. We consider first the case when $m=O(1)$. In
this case, $S(\alpha), S_1(\alpha)=\alpha+O(\alpha^2)$ and the
numerator and the denominator of \eqref{Gnms} both have the form
\[
    1-\frac sm +O(n^{-1}),
\]
which can be readily checked by using the estimates \eqref{mO1} and
\[
    \Lambda_{n,m} = e^{-1}\alpha +O(\alpha^2).
\]
From now on, we assume $m\ge m_0$, where $m_0$ is sufficiently
large, say $m_0\ge 10$. Throughout the proof, all $O$-terms hold
uniformly for $|s|\le 1-\ve$ and $m_0 \le m \le n$ and $n$ large
enough.

We begin with the denominator of $G_{n,m}(s)$, which satisfies
\begin{align*}
    1-(1-\Lambda_{n,m})e^{\frac{s}{e n}}
    & = \Lambda_{n,m} - \frac{s}{en}
    + \Lambda_{n,m} \frac{s}{en} + O\left(n^{-2}\right)\\
    &= \Lambda_{n,m} \left(1+\frac{s}{en}\right)
    \left(1-\frac{s}{en \Lambda_{n,m}} +
    O\left((mn)^{-1}\right)\right),
\end{align*}
where we used the estimate $\Lambda_{n,m} = \Omega(\alpha)$; see
\eqref{L-ub}. By \eqref{LS} and \eqref{L-ub}, the second-order term
on the right-hand side satisfies
\begin{align*}
    \frac{s}{en \Lambda_{n,m}}
    = \frac{s}{nS(\alpha) (1+O(n^{-1}))}
    = \frac{s}{m} \cdot \frac{\frac{m}{n}}{S(\frac{m}{n})}
    + O\left((mn)^{-1}\right).
\end{align*}
Thus we obtain
\begin{equation}\label{denom}
    1-(1-\Lambda_{n,m}) e^{\frac{s}{e n}} =
    \Lambda_{n,m} \left(1+\frac{s}{en}\right)
    \left(1 - \frac{s}{m} \cdot \frac{\alpha}{S(\alpha)}
    + O\left((mn)^{-1}\right)\right).
\end{equation}

Now we turn to the numerator of $G_{n,m}(s)$ and look first at the
exponential term
\begin{align*}
    e^{-\big(\phi\left(\frac{m}{n}\right)
    -\phi\left(\frac{m-\ell}{n}\right)\big)s}
    & = e^{-\frac{\ell}{n} \phi'(\alpha) s
    + O\left(\frac{\ell^{2}}{n^{2}}\right)}\\
    & = \left(1-\frac{\ell}{n} \,\phi'(\alpha) s\right)
    \left(1+O\left(\frac{\ell^{2}}{n^{2}}\right)\right),
\end{align*}
uniformly for $1\le \ell\le m$, where we used the twice continuous
differentiability of $\phi$.

Consider now the finite product $\prod_{m-\ell+1 \le r \le
m}\big(1-\frac{s}{r}\big)$. Obviously, for $|s|\le1$, we have the
uniform bound
\begin{align*}
    \prod_{m-\ell+1 \le r \le m}
    \left|1-\frac{s}{r}\right|
    \le \prod_{m-\ell+1 \le r \le m}\left(1+\frac1r\right)
    \le e^{H_m} =O(m).
\end{align*}
On the other hand, we also have the finer estimates
\begin{align*}
    \prod_{m-\ell+1 \le r \le m}
    \left(1-\frac{s}{r}\right)
    &= e^{-(H_{m} - H_{m-\ell})s}
    \left(1+O\left(\frac{\ell^{2}}{m^{2}}\right)\right)\\
    &= e^{-\frac \ell m \,s +O\left(\frac{\ell^{2}}{m^{2}}\right)}
    \left(1+O\left(\frac{\ell^{2}}{m^{2}}\right)\right)\\
    &= \left(1-\frac{\ell}{m}\,s\right)
    \left(1+O\left(\frac{\ell^{2}}{m^{2}}\right)\right),
\end{align*}
uniformly for $1\le \ell = o(m)$.

Combining these two estimates, we obtain the following approximation
\begin{equation*}
    \prod_{m-\ell+1 \le r \le m}\left(1-\frac{s}{r}\right)
    = \left(1 - \frac{\ell}{m}\, s\right)
    \left(1 + \dbbracket{\ell \ge 2}
    O\left(\frac{\ell^{2}}{m^{2}}\right)
    + \dbbracket{\ell > \textstyle{\lceil\sqrt{m} \rceil}}
    O(m)\right),
\end{equation*}
which holds uniformly for $1 \le \ell \le m$. Thus the numerator, up
to the factor $e^{\frac{s}{e n}}$, satisfies
\begin{align*}
    & \sum\limits_{1 \le \ell \le m}
    \lambda_{n,m,\ell} e^{-\big(\phi\left(\frac{m}{n}\right)
    -\phi\left(\frac{m-\ell}{n}\right)\big)s}
    \prod\limits_{m-\ell+1 \le r \le m}\left(1-\frac{s}{r}\right) \\
    & \qquad = \sum\limits_{1 \le \ell \le m}
    \lambda_{n,m,\ell} - \frac{s}{m}\left(1
    +\alpha\phi'(\alpha) \right) \sum\limits_{1 \le \ell \le m}
    \ell \lambda_{n,m,\ell}  \\
    & \quad\qquad \mbox{} + O\left(\frac{1}{m^{2}}
    \sum\limits_{2 \le \ell \le m} \ell^{2}
    \lambda_{n,m,\ell}+ \frac{1}{m n}
    \sum\limits_{1 \le \ell \le m} \ell^{2}
    \lambda_{n,m,\ell}
    + m \sum\limits_{\lceil\sqrt{m}\rceil+1 \le \ell \le m}
    \lambda_{n,m,\ell}\right).
\end{align*}
Each of the sums can be readily estimated as in \eqref{small-ell}
and \eqref{large-ell}, and we have
\begin{align*}
    m^{-2} \sum\limits_{2 \le \ell \le m}
    \ell^{2} \lambda_{n,m,\ell}
    =O\left(n^{-2}\right).
\end{align*}
Similarly,
\begin{align*}
    (m n)^{-1} \sum\limits_{1 \le \ell \le m} \ell^{2}
    \lambda_{n,m,\ell}
    =O\left((mn)^{-1} \alpha\right) = O\left(n^{-2}\right).
\end{align*}
Finally, for $m\ge 1$,
\begin{align*}
    m \sum\limits_{\lceil\sqrt{m}\rceil+1 \le \ell \le m}
    \lambda_{n,m,\ell} &= O\left(\frac{m \alpha^{\sqrt{m}+1}}
    {\Gamma(\sqrt{m}+2)} \right)\\
    &= O\left(m^{7/4}n^{-1} e^{-\sqrt{m}(\log n-\frac12\log m -1)}
    \right)=O(n^{-2})
\end{align*}

Collecting these estimates, we get
\begin{align}
    & \sum\limits_{1 \le \ell \le m}
    \lambda_{n,m,\ell} e^{-\big(\phi\left(\frac{m}{n}\right)
    -\phi\left(\frac{m-\ell}{n}\right)\big)s}
    \prod\limits_{m-\ell+1 \le r \le m}
    \left(1-\frac{s}{r}\right)\nonumber \\
    & \qquad = \Lambda_{n,m} - \frac{s}{m}\left(1+\alpha
    \phi'(\alpha) \right) \Lambda_{n,m}^{(1)}
    + O\left(n^{-2}\right)\nonumber  \\
    & \qquad = \Lambda_{n,m} \left( 1 - \frac{s}{m}
    \left(1+\alpha\phi'(\alpha) \right)
    \frac{\Lambda_{n,m}^{(1)}}{\Lambda_{n,m}}
    + O\left((m n)^{-1}\right) \right)\nonumber \\
    &\qquad = \Lambda_{n,m} \left( 1 - \frac{s}{m}
    \left(1+\alpha\phi'(\alpha) \right)
    \frac{S_1(\alpha)}{S(\alpha)}
    + O\left((m n)^{-1}\right) \right),\label{numer}
\end{align}
by applying \eqref{LS}.

By \eqref{denom}, \eqref{numer} and the simple estimate
\[
    e^{\frac{s}{en}}= \left(1+\tfrac s{en}\right)
    \left(1+O(n^{-2})\right),
\]
we conclude \eqref{Gnms}. \qed

\begin{cor} \label{lem:Gmns_special}
Let $\phi(x) = \phi_{1}(x) = \int_{0}^{x} \big(\frac{1}{S_{1}(t)} -
\frac{1}{t}\big) \dd t$. Then
\begin{equation*}
    G_{n,m}(s) = 1 + O\big((m n)^{-1}\big),
\end{equation*}
where the $O$-term holds uniformly for $1 \le m \le n$, $n$ large
enough and $|s|\le 1-\ve$.
\end{cor}
\pf To obtain the error term $O((mn)^{-1})$, we choose $\phi$ in a
way that the two middle terms in the fraction of \eqref{Gnms} are
identical, which means
\[
    \frac{x}{S(x)} = (1+x\phi'(x))\frac{S_{1}(x)}{S(x)}.
\]
Observe that $S(x)>0$ for $x>0$. This, together with
$\phi_{1}(0)=0$, implies $\phi=\phi_1$, which is not only a
$C^2$-function but also analytic in the unit circle. \qed

\paragraph{Proof of Proposition~\ref{prop-Y-all}.}
We now prove Proposition~\ref{prop-Y-all} by induction.
\begin{lmm} \label{lemma:Fnms}
Let $\phi=\phi_{1}$. Then
\begin{equation*}
    F_{n,m}(s) = 1 + O\left(n^{-1}H_{m}\right),
\end{equation*}
uniformly for $0 \le m \le n$, $n$ large enough and $|s|\le \kappa$,
$\kappa\in(0,1)$.
\end{lmm}
\pf We use induction on $m$ and show that there exists a constant $C
> 0$, such that
\[
    |F_{n,m}(s)-1| \le C n^{-1} H_m,
\]
for all $1 \le m \le n$, $n \ge n_0$ large enough and $|s| \le
\kappa$.

When $m=0$, the lemma holds, since $F_{n,0}(s) \equiv 1$.

Assume that the lemma holds for all functions $F_{n,k}(s)$ for $0
\le k \le m$ and $n \ge n_0$. By Corollary~\ref{lem:Gmns_special},
there exists a constant $C_{1} > 0$ such that for all $1 \le m \le
n$, $n \ge n_1$ large enough and $|s| \le \kappa_1$, $\kappa_{1} >
0$,
\begin{equation*}
    |G_{n,m}(s) - 1| \le C_{1}(m n)^{-1}.
\end{equation*}
Now
\begin{align*}
    |F_{n,m}(s)-1| &= |F_{n,m}(s)-G_{n,m}(s)
    + G_{n,m}(s)-1| \\
    &\le |F_{n,m}(s)-G_{n,m}(s)| + C_{1}(m n)^{-1} .
\end{align*}
The first term on the right-hand side can be re-written as
\begin{align*}
    &|F_{n,m}(s)-G_{n,m}(s)|\\
    &\qquad =
    \frac{\left|e^{\frac{s}{en}}
    \sum\limits_{1 \le \ell \le m} \lambda_{n,m,\ell}
    (F_{n,m-\ell}(s)-1)
    e^{-\big(\phi\left(\frac{m}{n}\right)
    -\phi\left(\frac{m-\ell}{n}\right)\big)s}
    \prod\limits_{m-\ell+1 \le r \le m}
    \left(1-\tfrac{s}{r}\right)\right|}
    {\left|1-\big(1-\Lambda_{n,m}\big) e^{\frac{s}{en}}\right|}.
\end{align*}
Since we assume $|s| \le 1$, the product involved in the sum on the
right-hand side is nonnegative and we have, by the induction
hypothesis,
\begin{align*}
    |F_{n,m}(s)-G_{n,m}(s)|
    &\le \frac{e^{\frac{s}{en}} \sum\limits_{1 \le \ell \le m}
    \lambda_{n,m,\ell} e^{-\big(\phi\left(\frac{m}{n}\right)
    -\phi\left(\frac{m-\ell}{n}\right)\big)s}\frac{C H_{m-\ell}}{n}
    \prod\limits_{m-\ell+1 \le r \le m}\left(1-\frac{s}{r}\right)}
    {1-\big(1-\Lambda_{n,m}\big) e^{\frac{s}{en}}}\\
    & \le  \frac{C H_{m-1}}{n}\, G_{n,m}(s)
    \le  \frac{C H_{m-1}}{n} + \frac{C H_{m-1}}{n}
    \cdot |G_{n,m}(s)-1| \\
    & \le  \frac{C H_{m-1}}{n} + \frac{C H_{m-1}}{n}
    \cdot \frac{C_{1}}{m n}.
\end{align*}
It follows that
\begin{align*}
    |F_{n,m}(s)-1|\le
    \frac{C H_{m}}{n} + \frac{1}{m n}
    \left(C_{1} - C + C_{1} C\frac{H_{m-1}}{n}\right).
\end{align*}
Choose first $n_2 \ge n_1$ such that $\frac{H_{n_2-1}}{n_2} \le
\frac{1}{2 C_1}$, which implies that $C_{1} C \frac{H_{m-1}}{n} \le
\frac{C}{2}$ for $1 \le m \le n$ and $n \ge n_2$. Then choose $C = 2
C_1$. We then have
\[
    C_{1} - C + \frac{C_{1} C H_{m-1}}{n}
    \le C_{1} - \frac{C}{2} \le 0,
\]
and thus
\[
    |F_{n,m}(s)-1| \le \frac{C H_{m}}{n}.
\]
Note that apart from requiring $|s| \le 1$ the only restriction on
$s$ comes from $G_{n,m}(s)$, thus we may choose $\kappa =
\min(1,\kappa_{1})$. This completes the proof. \qed

\paragraph{The Gumbel limit laws for $X_{n,m}$ ($m\to\infty$).}
We prove Theorem~\ref{thm:Y-all} by Proposition~\ref{prop-Y-all}.
Since $m \to \infty$, we have
\begin{align}
    \mathbb{E}\left(e^{\frac{X_{n,m}}{en}\,s - (\log m +
    \phi_{1}(\alpha))s}\right)
    &= P_{n,m}\left(e^{\frac s{e n}}\right)
    e^{-H_{m} s +\gamma s- \phi_{1}(\frac{m}{n}) s}
    \left(1+O\left(m^{-1}\right)\right)\nonumber \\
    &= e^{\gamma s} F_{n,m}(s) \prod_{1 \le r \le m}
    \frac{e^{-\frac sr}}{1-\frac{s}{r}}
    \left(1+O\left(m^{-1}\right)\right)\nonumber \\
    &= \Gamma(1-s)\left(1+O\left(\frac{\log m}{n}
    + \frac1m\right)\right) .\label{G-m}
\end{align}
Thus Theorem~\ref{thm:Y-all} follows from another application of
Curtiss's theorem (see \cite[\S 5.2.3]{HMC13}).

\paragraph{The Gumbel limit law for $X_n$.}
We now prove Theorem~\ref{thm:Y}, starting from the moment
generating function ($\bar{\rho}:=1-\rho$)
\[
    \mathbb{E}\left(e^{X_ns}\right)
    = \sum_{0\le m\le n} \binom{n}{m}\rho^m\bar{\rho}^{n-m}
    P_{n,m}\left(e^s\right) .
\]
Then
\begin{align*}
    \mathbb{E}\left(e^{\frac{X_n}{en}s
    - (\log \rho n +\phi_1(\rho))s}\right)
    &= \sum_{0\le m\le n} \binom{n}{m}
    \rho^m \bar{\rho}^{n-m}
    P_{n,m}\left(e^{\frac{s}{en}}\right)
    e^{-(H_m-\gamma+\phi_1(\alpha))s+ \delta_{n,m} s},
\end{align*}
where
\[
    \delta_{n,m}
    := H_m- \log \rho n-\gamma+\phi_1(\alpha) -\phi_1(\rho).
\]
Since the binomial distribution is highly concentrated around the
range $m=\rho n+x\sqrt{\rho\bar{\rho}n}$ where $x=o(n^{\frac16})$,
we see that
\[
    \delta_{n,m} = O\left(n^{-\frac12}|x|\right),
\]
for $m$ in this central range. By a standard argument (Gaussian
approximation of the binomial and exponential tail estimates) using
the expansion \eqref{G-m}, we then deduce that
\[
    \mathbb{E}\left(e^{\frac{X_n}{en}s
    - (\log \rho n +\phi_1(\rho))s}\right)
    = \Gamma(1-s)\left(1+O\left(n^{-\frac12}
    \right)\right).
\]
This proves Theorem~\ref{thm:Y}.

\section{Analysis of the ($1+1$)-EA for \textsc{LeadingOnes}}
\label{sec:LO}

We consider the complexity of the ($1+1$)-EA when the underlying
fitness function is the number of leading ones. This problem has
been examined repeatedly in the literature due to the simple
structures it exhibits; see \cite{BDN10,DJW02,Ladret05} and the
references therein. The strongest results obtained were those by
Ladret \cite{Ladret05} (almost unknown in the EA literature) where
she proved that the optimization time under \textsc{LeadingOnes} is
asymptotically normally distributed with mean asymptotic to
$\frac{e^c-1}{2c^2}n^2$ and variance to $\frac{3(e^{2c}-1)}{8c^3}
n^3$, where $p=\frac cn$, $c>0$.

We re-visit this problem and obtain similar type of results by a
completely different approach, which can be readily amended for
obtaining the convergence rate.

Throughout this section, the probability $p$ still carries the same
meaning from Algorithm $(1+1)$-EA and $q=1-p$.

\begin{lmm} Let $Y_{n,m}$ denote the conditional optimization
time when beginning with a random input (each bit being 1 with
probability $\frac12$) that has $n-m$ leading ones. Then the moment
generating function $Q_{n,m}(s)$ of $Y_{n,m}$ satisfies the
recurrence relation
\begin{align} \label{Qnms}
    Q_{n,m}(s)
    = \frac{pq^{n-m}e^s}{1-(1-pq^{n-m})e^s}
    \left(2^{1-m}+ \sum_{1\le \ell<m}
    \frac{Q_{n,\ell}(s)}{2^{m-\ell}}\right),
\end{align}
for $1\le m\le n$, where $q=1-p$.
\end{lmm}
\pf The probability of jumping from a state with $n-m$ leading ones
to another state with $n-m+\ell$ leading ones is given by
\[
    (1-p)^{n-m} \cdot p\cdot  2^{-\ell}\qquad(1\le \ell <m),
\]
which corresponds to the situation when the first $n-m$ bits do not
toggle their values, the $(n-m+1)$st bit toggles (from $0$ to $1$),
together with the following $\ell-1$ bits also being 1. When $\ell =
m$, the probability becomes
\[
    (1-p)^{n-m} \cdot p\cdot  2^{-\ell+1}.
\]
We thus obtain the recurrence relation
\[
    Q_{n,m}(s)
    = pq^{n-m}e^s \left(2^{1-m}+ \sum_{1\le \ell<m}
    \frac{Q_{n,\ell}(s)}{2^{m-\ell}}\right)+
    (1-pq^{n-m})e^sQ_{n,m}(s),
\]
which implies \eqref{Qnms}. \qed

The most interesting case is when $p\asymp n^{-1}$ (roughly, $pq^n$
is linear, giving rise to polynomial bounds for the cost), all other
cases when $pn\to\infty$ lead to higher-order complexity.

\paragraph{Small $m$.}
We start with the simplest case when $m=1$ and obtain, by
\eqref{Qnms},
\begin{align}\label{Pn1}
    Q_{n,1}(s) = \frac{pq^{n-1}e^s}{1-(1-pq^{n-1})e^s}.
\end{align}
Then the mean of $Y_{n,1}$ is simply given by
\[
    \mathbb{E}(Y_{n,1}) = \frac1{pq^{n-1}},
\]
which, by substituting $p=\frac cn$, yields
\[
    \mathbb{E}(Y_{n,1}) = \frac{e^c}{c}\, n
    -\left(1-\frac c2\right)e^c + O\left(\frac{e^c}
    {n}(1+c^3)\right).
\]
Note that this estimate holds as long as $c=o(\sqrt{n})$. Similarly,
the variance is given by
\[
    \mathbb{V}(Y_{n,1})
    = \frac1{(pq^{n-1})^2}-\frac1{pq^{n-1}},
\]
which satisfies, when $p=\frac cn$,
\begin{align*}
    \mathbb{V}(Y_{n,1}) &= \frac{e^{2c}}{c^2}\, n^2
    -\frac{e^c+(2-c)e^{2c}}{c}\,n
    +\left(1-\frac c2\right)e^c +
    \left(1-\frac43c+\frac{c^2}2\right)e^{2c}\\
    &\qquad + O\left(\frac{e^{2c}}
    {n}(c^2+c^4)\right),
\end{align*}
uniformly when $c=o(\sqrt{n})$.

We then consider the normalized random variables $cY_{n,1}/(e^cn)$.
By the expansion
\[
    \frac{\frac cn\left(1-\frac cn\right)^{n-m}
    \exp\left(\frac{ce^{-c}}{n}\,s\right)}
    {1-\left(1-\frac cn\left(1-\frac cn\right)^{n-m}\right)
    \exp\left(\frac{ce^{-c}}{n}\,s\right)}
    = \frac1{1-s}+ O\left(\frac{c|s|(m+c)}{n|1-s|^2}\right),
\]
uniformly when $s$ is away from $1$ and $m=o(n)$, we obtain
\[
    \mathbb{E}\left(e^{cY_{n,1}s/(e^cn)}\right)
    \to \frac1{1-s},
\]
implying that the limit law is an exponential distribution with the
density $e^{-x}$. While \eqref{Pn1} shows that $Y_{n,1}\equiv
X_{n,1}$ when $p=\frac1n$, they behave differently when $m\ge2$.

\begin{thm} For each $1\le m=O(1)$, the limit distribution of
$cY_{n,m}/(e^c n)$ is a binomial mixture of Gamma distributions;
more precisely, \label{thm:XO1}
\begin{align}\label{Xnm-mO1}
    \mathbb{P}\left(\frac{c Y_{n,m}}{e^c n}\le x\right)
    \to  \frac{1}{2^{m-1}}\sum_{0\le j<m}
    \binom{m-1}{j}\int_0^xe^{-t}\frac{t^j}{j!} \dd t
     \qquad(x>0),
\end{align}
as $n\to\infty$. The mean and the variance satisfy
\begin{align}\label{Xnm-mv-s}
    \mathbb{E}(Y_{n,m}) \sim \frac{m+1}{2ce^{-c}}\,n,
    \qquad \mathbb{V}(Y_{n,m}) \sim \frac{3m+1}
    {4c^2e^{-2c}}\,n^2.
\end{align}
\end{thm}
Note that when $m=1$, \eqref{Xnm-mO1} degenerates to the exponential
distribution. On the other hand, the normalizing factor $ce^{-c}n$
is not asymptotically equivalent to the mean.
\begin{proof}
By induction and \eqref{Qnms}, we see that
\[
    \mathbb{E}\left(e^{cY_{n,m}s/(e^cn)}\right)
    \to \frac{(1-\frac s2)^{m-1}}{(1-s)^m},
\]
when $m=O(1)$. Since
\[
    \int_0^\infty e^{-x(1-s)}
    \sum_{0\le j<m}\frac1{2^{m-1}}\binom{m-1}{j}
    \frac{x^j}{j!} \dd x
    = \frac{(1-\frac s2)^{m-1}}{(1-s)^m},
\]
we then deduce \eqref{Xnm-mO1}. The mean and the variance then
follows from straightforward calculations.
\end{proof}

\paragraph{Mean and the variance of $Y_{n,m}$: $1\le m\le n$.}
The recurrence \eqref{Qnms} is much simpler than \eqref{Qnmt} and we
can indeed obtain very precise expressions and approximations for
the mean and the variance.
\begin{thm} The mean $\nu_{n,m}$ and the variance $\varsigma_{n,m}$
of $Y_{n,m}$ are given explicitly as follows. For $1\le m\le n$
\label{thm:Xmv}
\begin{align}\label{mu-exact}
    \nu_{n,m} := \mathbb{E}(Y_{n,m}) = \frac1{pq^{n-1}}
    \left(\frac{1-q^{m-1}}{2p} + q^{m-1}\right),
\end{align}
and
\begin{align}\label{V-exact}
    \varsigma_{n,m}^2 := \mathbb{V}(Y_{n,m}) = -\nu_{n,m}
    + \frac{3q^2-(4q^2-1)q^{2m}}{4p^3(1+q)q^{2n}}.
\end{align}
\end{thm}

With these closed-form expressions, we easily obtain, assuming
$p=\frac cn$, where $c>0$,
\[
    \nu_{n,m} = \frac{e^c(m+1)}{2c}\,n
    -\frac{e^c}{4}\left(m^2+(3-c)m-c\right)
    +O\left(\frac{ce^c}{n}\left(m^3+mc^2\right)
    \right),
\]
and
\begin{align*}
    \varsigma_{n,m}^2 &= \frac{e^{2c}(3m+1)}{4c^2}\, n^2
    -\frac{e^{2c}}{8c}\left(3m^2+(5-3c)m-c+
    2(m+1)e^{-c}\right)n \\
    &\qquad +O\left(e^{2c}\left(m^3+c^2\right)
    \right),
\end{align*}
uniformly for $cm=o(n)$. We see that the asymptotic equivalents
\eqref{Xnm-mv-s} indeed hold in the wider range $cm=o(n)$.

More uniform approximations have the following forms.
\begin{cor} Assume that $p=\frac cn$, where $c=o(\sqrt{n})$.
Then, uniformly for $0\le \alpha:=\frac mn\le 1$,
\begin{align}\label{LO-mu}
    \nu_{n,m} = \frac{e^c}{2c^2}\left(1-e^{-c\alpha}\right) n^2
    +\frac{e^c}{4c}\left(c-2+e^{-c\alpha}\left(
    4-c+c\alpha\right)\right)n+O\left( c(c+1)e^c
    \right),
\end{align}
and
\begin{align}\label{LO-var}
    \varsigma_{n,m}^2 =
    \frac{3e^{2c}}{8c^3}\left(1-e^{-2c\alpha}\right)n^3
    +O\left(c^{-2}e^{2c}(1+c)n^2\right).
\end{align}
\end{cor}

\paragraph{Proof of Theorem~\ref{thm:Xmv}.}
Our approach is based on \eqref{Qnms} and it turns out that all
moments satisfy the same simple recurrence of the following type.
\begin{lmm} The solution to the recurrence relation
\[
    a_m = b_m + \sum_{1\le\ell<m} \frac{a_\ell}{2^{m-\ell}}
    \qquad(m\ge1),
\]
is given by the closed-form expression
\begin{align} \label{am-bm}
    a_m = b_m + \frac12\sum_{1\le j<m}b_j.
\end{align}
\end{lmm}
\pf The corresponding generating functions $f(z) := \sum_{m\ge1}a_m
z^m$ and $g(z) := \sum_{m\ge1}b_m z^m$ satisfy the equation
\[
    f(z) = g(z)+ \frac{z}{2-z} \, f(z) ,
\]
or
\[
    f(z) = \frac{1-\frac z2}{1-z}\, g(z).
\]
This proves \eqref{am-bm}. \qed

From \eqref{Qnms} (by taking derivative with respect to $s$ and then
substituting $s=1$), we see that the mean $\nu_{n,m}$ satisfies the
recurrence
\[
    \nu_{n,m} = \frac1{pq^{n-m}}+
    \sum_{1\le\ell<m} \frac{\nu_{n,\ell}}{2^{m-\ell}}
    \qquad(m\ge1).
\]
Substituting $b_m= 1/(pq^{n-m})$ into \eqref{am-bm}, we obtain
\eqref{mu-exact}.

Similarly, for the second moment $s_{n,m} := \mathbb{E}(Y_{n,m}^2)$,
we have the recurrence
\[
    s_{n,m} = \frac{2\nu_{n,m}-1}{pq^{n-m}} +
    \sum_{1\le \ell<m}\frac{s_{n,\ell}}{2^{m-\ell}}.
\]
By the same procedure, we obtain
\begin{align}\label{snm}
    s_{n,m} = -\nu_{n,m} + \frac{q^2(2-q)-(q+1)q^{m+1}
    (2q-1)+(2q-1)(2q^2-1)q^{2m}}{2p^4(1+q)q^{2n}},
\end{align}
implying \eqref{V-exact}. This proves Theorem~\ref{thm:Xmv}. \qed

The proofs of the two Corollaries are straightforward and omitted.

\paragraph{A finite-product representation for $Q_{n,m}(s)$.}
The recurrence relation \eqref{Qnms} can indeed be solved explicitly
as follows.
\begin{prop} The moment generating function $Q_{n,m}(s)$ of $Y_{n,m}$
has the closed-form
\begin{align}\label{Pnms-fp}
    Q_{n,m}(s) = \frac1{1-\frac{1-e^{-s}}{pq^{n-m}}}
    \prod_{1\le j<m}
    \frac{1-\frac{1-e^{-s}}{2pq^{n-j}}}
    {1- \frac{1-e^{-s}}{pq^{n-j}}},
\end{align}
for $m\ge1$.
\end{prop}
\begin{proof}
Let $\omega := (1-e^{-s})/(pq^n)$. We start with the recurrence
(from \eqref{Qnms})
\[
    Q_{n,m}(s) = \omega q^m Q_{n,m}(s)
    + 2^{1-m} + \sum_{1\le \ell <m} \frac{Q_{n,\ell}(s)}
    {2^{m-\ell}},
\]
which, by \eqref{am-bm}, has the alternative form
\begin{align}\label{Qnms1}
    Q_{n,m}(s) = 1+ \omega q^m Q_{n,m}(s)
    + \frac{\omega}{2}\sum_{1\le h <m}
    q^h Q_{n,h}(s).
\end{align}
From \eqref{Qnms1}, we see that the bivariate generating function
$$Q_n(z,s) := \sum_{m\ge1}Q_{n,m}(s) z^m$$ of $Q_{n,m}(s)$ satisfies
\[
    Q_n(z,s) = \frac{z}{1-z} + \omega Q_n(qz,s)
    + \frac{\omega}{2}\cdot \frac{z}{1-z}\,Q_n(qz,s),
\]
which implies the simpler functional equation
\[
    Q_n(z,s) = \frac{z}{1-z} + \omega
    \frac{1-\frac z2}{1-z}\, Q_n(qz,s).
\]
Multiplying both sides by $1-z$ gives
\[
    (1-z)Q_n(z,s) = z + \omega
    \left(1-\frac z2\right) Q_n(qz,s),
\]
implying the relation
\[
    \frac{Q_{n,m}(s)}{Q_{n,m-1}(s)}
    = \frac{1-\frac12 \omega q^{m-1}}
    {1- \omega q^m} \qquad(m\ge2).
\]
Accordingly, we obtain the closed-form expression \eqref{Pnms-fp}.
\end{proof}

Let
\[
    G_m(t) := \frac{pq^{n-m} t}{1-(1-pq^{n-m})t}
\]
denote the probability generating function of a geometric
distribution $\text{Geo}(pq^{n-m})$ with parameter $pq^{n-m}$ and
support $\{1,2,\dots\}$.
\begin{cor} The random variables $Y_{n,m}$ can be decomposed as
the sum of $m$ independent random variables
\begin{align}\label{X-Z}
    Y_{n,m} \stackrel{d}{=} Z_{n,m}^{[0]} +\cdots
    + Z_{n,m}^{[m-1]},
\end{align}
where $Z_{n,m}^{[0]} \sim \text{Geo}(pq^{n-m})$ and the
$Z_{n,m}^{[j]}$ are mixture of $\text{Geo}(pq^{n-j})$
\[
    \mathbb{E}\left(t^{Z_{n,m}^{[j]}}\right)
    = \frac12\cdot\frac{1-(1-2pq^{n-j})t}{1-(1-pq^{n-j})t}
    = \frac12 + \frac{R_j(t)}2
    \qquad(j=1,\dots,m-1).
\]
\end{cor}

Thus the mean of $Y_{n,m}$ is given by
\[
    \mathbb{E}(Y_{n,m}) = \sum_{0\le j<m}
    \mathbb{E}\left(Z_{n,m}^{[j]}\right)
    = \frac1{pq^{n-m}}+ \frac12\sum_{1\le j<m} \frac1{pq^{n-j}},
\]
which is identical to \eqref{mu-exact}. Similarly, the variance of
$Y_{n,m}$ satisfies
\[
    \varsigma_{n,m}^2 = \sum_{0\le j<m}
    \mathbb{V}\left(Z_{n,m}^{[j]}\right)
    = \frac{3-2pq^{n-m}}{8(pq^{n-m})^2}+
    \frac12\sum_{1\le j<m} \frac{1-pq^{n-j}}{(pq^{n-j})^2},
\]
which is also identical to \eqref{V-exact}.

\begin{thm} The distributions of $\frac{Y_{n,m}-\nu_{n,m}}
{\varsigma_{n,m}}$ are asymptotically normal \label{thm:X-all}
\[
    \mathbb{P}\left(\frac{Y_{n,m}-\nu_{n,m}}
    {\varsigma_{n,m}}\le x\right)
    \to \Phi(x),
\]
uniformly as $m\to\infty$ (with $n$) and $m\le n$, where $\Phi(x) :=
\frac1{\sqrt{2\pi}} \int_{-\infty}^x e^{-\frac{t^2}2}\dd t$ denotes
the standard normal distribution function.
\end{thm}
\begin{proof}
Again from the decomposition \eqref{X-Z}, we derive the following
expression for the third central moment
\[
    \kappa_{n,m}
    :=\mathbb{E}\left(Y_{n,m}-\nu_{n,m}\right)^3
    = \frac{7q^3-(8q^3-1)q^{3m}}{4(1-q^3)(pq^n)^3}
    -3\varsigma_{n,m}^2-2\nu_{n,m}.
\]
Similarly, the fourth central moment satisfies
\begin{align*}
    \mathbb{E}\left(Y_{n,m}-\nu_{n,m}\right)^4
    -3\varsigma_{n,m}^4
    &= \frac{3(15q^4-(16q^4-1)q^{4m})}{8(1-q^4)(pq^n)^4}\\
    &\quad -6\mathbb{E}\left(Y_{n,m}-\nu_{n,m}\right)^3
    -11\varsigma_{n,m}^2-6\nu_{n,m},
\end{align*}
which implies that
\[
    \mathbb{E}\left(Y_{n,m}-\nu_{n,m}\right)^4
    =3\varsigma_{n,m}^4(1+o(1)),
\]
uniformly for $1\le m\le n$. We then deduce a central limit theorem
by, say Lyapounov's condition, or by Levy's continuity theorem; see,
for example, \cite{Petrov75}. We can indeed derive an optimal
Berry-Esseen bound by more refined Fourier argument, details being
omitted here.
\end{proof}

In particular, we have
\begin{align}\label{4th-mm}
    \mathbb{E}\left(Y_{n,m}-\nu_{n,m}\right)^4
    -3\varsigma_{n,m}^4 \sim \frac{45(1-e^{-4c\alpha})}
    {32 c^5 e^{-4c}}\, n^5,
\end{align}
when $m\to\infty$ and $m\le n$. This will be needed later.

\paragraph{Random input.} Now consider the cost $Y_n$ used by
Algorithm $(1+1)$-EA when starting from a random input (each bit
being $1$ with probability $\frac12$). Then its moment generating
function satisfies
\[
    \mathbb{E}\left(e^{Y_ns}\right)
    := 2^{-n} + \sum_{1\le m\le n}2^{m-n-1} Q_{n,m}(s).
\]
\begin{thm} The random variables $Y_n$ are asymptotically
normally distributed
\[
    \mathbb{P}\left(\frac{Y_n-\nu_n}
    {\varsigma_n}\le x\right)
    \to \Phi(x),
\]
with mean $\nu_n$ and variance $\varsigma_n$ asymptotic to
\begin{align}\label{nu-sg}
\begin{split}
    \nu_n &= \frac{e^c-1}{2c^2}\,n^2
    + \frac{(c-2)e^c+2}{4c}\,n+O(1) \\
    \varsigma_n &= \frac{e^{2c}-1}{8c^3}\,n^3 + \frac{3e^{2c}(2c-3)
    -8e^c+17}{16c^2}\,n^2 +O(n),
\end{split}
\end{align}
respectively.
\end{thm}
In particular, we also have, by replacing the exact mean and
variance by the corresponding asymptotic approximations
\[
    \mathbb{P}\left(\frac{Y_n-\frac{e^c-1}{2c^2}\,n^2}
    {\sqrt{\frac{e^{2c}-1}{8c^3}\,n^3}}\le x\right)
    \to \Phi(x).
\]
\begin{proof}
By \eqref{mu-exact}, we have
\begin{align*}
    \nu_n = \sum_{1\le m\le n} 2^{-n+m-1}\nu_{n,m}
    = \frac{q}{2p^2}\left(q^{-n}-1\right),
\end{align*}
and then the first estimate in \eqref{nu-sg} follows. Similarly, by
\eqref{snm},
\begin{align*}
    \varsigma_n^2 = \sum_{0\le m\le n}2^{-n+m-1}
    \mathbb{E}(Y_{n,m}^2) - \nu_n^2
    = \frac{3q^2}{4p^3(1+q)}\left(q^{-2n}-1\right)-\nu_n,
\end{align*}
and the second estimate in \eqref{nu-sg} also follows.

For the asymptotic normality, we consider the characteristic
function
\[
    \mathbb{E}\left(e^{\frac{Y_n-\nu_n}{\varsigma_n}it}\right)
    = 2^{-n}+ \sum_{0\le m< n}2^{-n+m-1} Q_{n,m}
    \left(\frac{it}{\varsigma_n}\right)
    e^{-\frac{\nu_n}{\varsigma_n}it}.
\]
We split the sum into two parts: $0\le n-m\le n^{\frac13}$ and $1\le
m<n-n^{\frac13}$. Observe that when $n-m\le n^{\frac13}$, we have
the uniform estimate
\[
    \nu_n - \nu_{n,m} = O(n|n-m+1|)
    \quad \text{and}\quad
    \varsigma_n^2-\varsigma_{n,m}^2 = O\left(n^2|n-m+1|\right),
\]
by \eqref{LO-mu} and \eqref{LO-var}. We then have the local
expansion (see \eqref{4th-mm})
\[
    Q_{n,m}\left(\tfrac{it}{\varsigma_n}\right)
    e^{-\frac{\nu_n}{\varsigma_n}it} = \exp\left(
    \frac{\nu_{n,m}-\nu_n}{\varsigma_n}\,it -
    \frac{\varsigma_{n,m}^2}{2\varsigma_n^2}\,t^2 +O\left(
    \frac{|t|^3}{n^{\frac32}} \right)\right).
\]
Thus
\begin{align*}
    Q_{n,m}\left(\tfrac{it}{\varsigma_n}\right)
    e^{-\frac{\nu_n}{\varsigma_n} it}
    &= \exp\left(-\frac{t^2}2 +
    O\left(\frac{|n-m+1|}{\sqrt{n}}\,|t|
    +\frac{|n-m+1|}{n}\,t^2\right)\right)\\
    &= \exp\left(-\frac{t^2}{2}+
    O\left(n^{-\frac16}|t| + n^{-\frac23}|t|^2\right)\right)\\
    &= e^{-\frac{t^2}2}(1+o(1)),
\end{align*}
uniformly in $m$. Consequently,
\[
    \sum_{n-n^{\frac13}\le m\le n} 2^{-n+m-1}
    Q_{n,m}\left(\tfrac{it}{\varsigma_n}\right)
    e^{-\frac{\nu_n}{\varsigma_n}it}
    = e^{-\frac{t^2}2}(1+o(1)).
\]
The remaining part is negligible since $|Q_{n,m}(e^{it/\sigma})|\le
1$ and
\[
    \sum_{1\le m\le n-n^{\frac13}} 2^{-n+m-1}
    Q_{n,m}\left(\tfrac{it}{\varsigma_n}\right)
    e^{-\frac{\nu_n}{\varsigma_n}it} = O\left(\sum_{m>n^{\frac13}}
    2^{-m}\right)=O\left(2^{-n^{\frac13}}\right).
\]
We conclude that
\begin{align*}
    \mathbb{E}\left(e^{\frac{Y_n-\nu_n}{\varsigma_n}it}\right)
    \to e^{-\frac{t^2}2},
\end{align*}
which implies the convergence in distribution of $\frac{Y_n-\nu_n}
{\varsigma_n}$ to the standard normal distribution.
\end{proof}

\section*{Appendix. A. Some properties of $S_r(z)$.}
We collected here some interesting expressions for $S_r(z)$.

We begin with proving that all $S_r$ can be expressed in terms of
$S_0$ and the two modified Bessel functions
\[
\begin{split}
    \thickbar{I}_0(\alpha)
    &:= I_0\left(2\sqrt{\alpha(1-\alpha)}\right)
    = \sum_{\ell\ge0}\frac{\alpha^\ell(1-\alpha)^\ell}
    {\ell!\ell!},\\
    \thickbar{I}_1(\alpha) &:= \sqrt{\frac{\alpha}{1-\alpha}}
    I_1\left(2\sqrt{\alpha(1-\alpha)}\right)
    = \sum_{\ell\ge1}\frac{\alpha^\ell(1-\alpha)^{\ell-1}}
    {\ell!(\ell-1)!}.
\end{split}
\]
The starting point is the obvious relation ($E_r(z) :=
\sum_{\ell\ge1}\ell^r z^{\ell-1}$)
\[
    E_r(z) = zE_{r-1}'(z) + E_{r-1}(z)
    \qquad(r\ge1).
\]
Applying the integral representation \eqref{Sr-int-rep} and
integration by parts, we have
\[
    S_r(\alpha) = \frac1{2\pi i}\oint_{|z|=c}
    \left(\frac{\alpha}{z}-(1-\alpha)z\right)E_{r-1}(z)
    e^{\frac{\alpha}z+(1-\alpha)z} \dd z.
\]
By the same argument used for Corollary~\ref{cor:Ur}, we deduce the
recurrence
\begin{align} \label{Sr-alpha}
    S_r(\alpha) = \alpha\thickbar{I}_0(\alpha)
    + \sum_{0\le j<r}\binom{r-1}{j}S_j(\alpha)
    \left(\alpha + (-1)^{r-j}(1-\alpha)\right),
\end{align}
for $r\ge2$ with
\[
    S_1(\alpha)
    =(2\alpha-1)S_0(\alpha) + \alpha\thickbar{I}_0(\alpha)
    + (1-\alpha)\thickbar{I}_1(\alpha) .
\]
A closed-form expression can be obtained for the recurrence
\eqref{Sr-alpha} but it is very messy. More precisely, let $f(z) :=
\sum_{r\ge0} S_r(\alpha)z^r/r!$. Then $f$ satisfies the first-order
differential equation
\[
    f'(z) = \left(\alpha e^z-(1-\alpha)e^{-z}\right) f(z)
    + \alpha\thickbar{I}_0(\alpha)
    + (1-\alpha)\thickbar{I}_1(\alpha).
\]
The solution to the differential equation with the initial condition
$f(0)=S_0(\alpha)$ is given by
\[
\begin{split}
    f(z) &= S_0(\alpha)e^{\alpha(e^z-e^{-z})+e^{-1}-1}\\
    &\qquad +e^{\alpha e^z+(1-\alpha)e^{-z}}
    \int_0^z \left(\alpha\thickbar{I}_0(\alpha)
    e^u+(1-\alpha)\thickbar{I}_1(\alpha)\right)
    e^{-\alpha e^u-(1-\alpha)e^{-u}} \dd u.
\end{split}
\]
This implies that $S_r(\alpha)$ has the general form
\[
    S_r(\alpha) = p_r^{[0]}(\alpha) \thickbar{I}_0(\alpha)
    + p_r^{[1]}(\alpha)
    \thickbar{I}_1(\alpha)+p_r^{[2]}(\alpha) S_0(\alpha)
    \qquad(r\ge1),
\]
where the $p_r^{[i]}$ are polynomials of $\alpha$ of degree $r$.
Closed-form expressions can be derived but are less simpler than the
recurrence \eqref{Sr-alpha} for small values of $r$.

On the other hand, the same argument also leads to
\[
    S_r'(\alpha) = \thickbar{I}_0(\alpha)
    + \sum_{0\le j<r} \binom{r}{j}S_j(\alpha)
    \left(1-(-1)^{r-j}\right)\qquad(r\ge1).
\]
In particular, $S_1'(\alpha) = \thickbar{I}_0(\alpha)
+2S_0(\alpha)$. Note that
\[
    S_0'(\alpha) = \thickbar{I}_0(\alpha)
    + \thickbar{I}_1(\alpha),
\]
implying that
\[
    S_0(\alpha) = \int_0^\alpha
    \left(\thickbar{I}_0(u)
    + \thickbar{I}_1(u)\right) \dd u.
\]
This in turn gives
\[
    S_1(\alpha) = \int_0^\alpha
    \left((1+2(\alpha-u))\thickbar{I}_0(u)
    + 2(\alpha-u)
    \thickbar{I}_1(u)\right)\dd u.
\]
This expression can be further simplified by taking second
derivative with respect to $\alpha$ of the integral representation
\[
    S_1(\alpha) = \frac1{2\pi i}\oint_{|z|=c}
    \frac{e^{\frac \alpha z+(1-\alpha)z}}{(1-z)^2} \dd z,
\]
giving
\[
    S_1''(\alpha) = 2\thickbar{I}_0(\alpha)
    + \alpha^{-1}\thickbar{I}_1(\alpha),
\]
which implies that (with $S_1(0)=0, S_1'(0)=1$)
\[
    S_1(\alpha) = \int_0^\alpha (\alpha-u)
    \left(2\thickbar{I}_0(u)
    +u^{-1}\thickbar{I}_1(u) \right) \dd u
\]
Similarly, since $S_2' = I_0 + 4S_1$, we have
\begin{align*}
    S_2(\alpha) &= \int_0^\alpha
    (1+4(\alpha-u)(\alpha-u+1))\thickbar{I}_0(u)\dd u
    + 4\int_0^\alpha (\alpha-u)^2
    \thickbar{I}_1(u)\dd u.
\end{align*}
These expressions show not only the intimate connections of $S_r$ to
Bessel functions but also their rich algebraic aspects.

We now consider $S_r(1-\alpha)$. By the same integral representation
and a change of variables, we see that, for $r\ge1$,
\[
    (-1)^rS_r(\alpha) +S_r(1-\alpha)
    = [z^0] E_r(1-z) e^{\frac{\alpha}{1-z}
    +(1-\alpha)(1-z)}.
\]
Now
\[
    E_r(1-z) = r![w^r] \frac{e^w}{1-(1-z)e^w}
    = \sum_{0\le j\le r} (-1)^{r+j}j!\,
    \text{Stirling}_2(r,j) z^{-j-1}.
\]
Thus we deduce the identity (for $r\ge1$)
\begin{align*}
    &(-1)^rS_r(\alpha) +S_r(1-\alpha)\\
    &\qquad = e\sum_{0\le \ell \le r}(-1)^{r+\ell}\ell!\,
    \text{Stirling}_2(r,\ell) 
    \sum_{\substack{0\le h\le \ell\\
    0\le j<h/2}}\binom{h-j-1}{j-1}
    \frac{(2\alpha-1)^{\ell-h}\alpha^j}
    {(\ell-h)!j!}
\end{align*}
or
\begin{align*}
    &(-1)^rS_r(\alpha) +S_r(1-\alpha)\\
    &\qquad = e\sum_{0\le \ell \le r}(-1)^{r+\ell}\ell!\,
    \text{Stirling}_2(r,\ell) 
    \left(\frac{(\alpha-1)^{\ell}}
    {\ell!}+ \sum_{\substack{0\le h\le \ell\\
    0\le j<h}}\binom{h-1}{j}\frac{\alpha^{h-j}(\alpha-1)^{\ell-h}}
    {(\ell-h)!(h-j)!}\right).
\end{align*}
Note that for $r=0$
\[
    S_0(\alpha)+S_0(1-\alpha)
    = e-\thickbar{I}_0(\alpha).
\]
In particular, this gives $S(\frac12)=\frac12(e-I_0(1)) \approx
0.726107$. For $r\ge1$
\begin{align*}
    S_1(\alpha)-S_1(1-\alpha) &= e(2\alpha-1)\\
    S_2(\alpha)+S_2(1-\alpha) &= e(4\alpha^2-4\alpha+2)\\
    S_3(\alpha)-S_3(1-\alpha) &= e(8\alpha^3-12\alpha^2+14\alpha-5)\\
    S_4(\alpha)-S_4(1-\alpha) &= e(16\alpha^4-32\alpha^3
    +64\alpha^2-48\alpha+15).
\end{align*}

\section*{Appendix. B. Closeness of the approximation 
\eqref{mu-asymp} for $\mu_{n,m}^*$: graphical representations}

The successive improvements attained by adding more terms 
on the right-hand side of \eqref{mu-asymp} can be viewed  
in Figures~\ref{App:fig-mu} and \ref{App:fig-mu2}. 

\begin{figure}[!ht]
\begin{center}
\includegraphics[width=6cm]{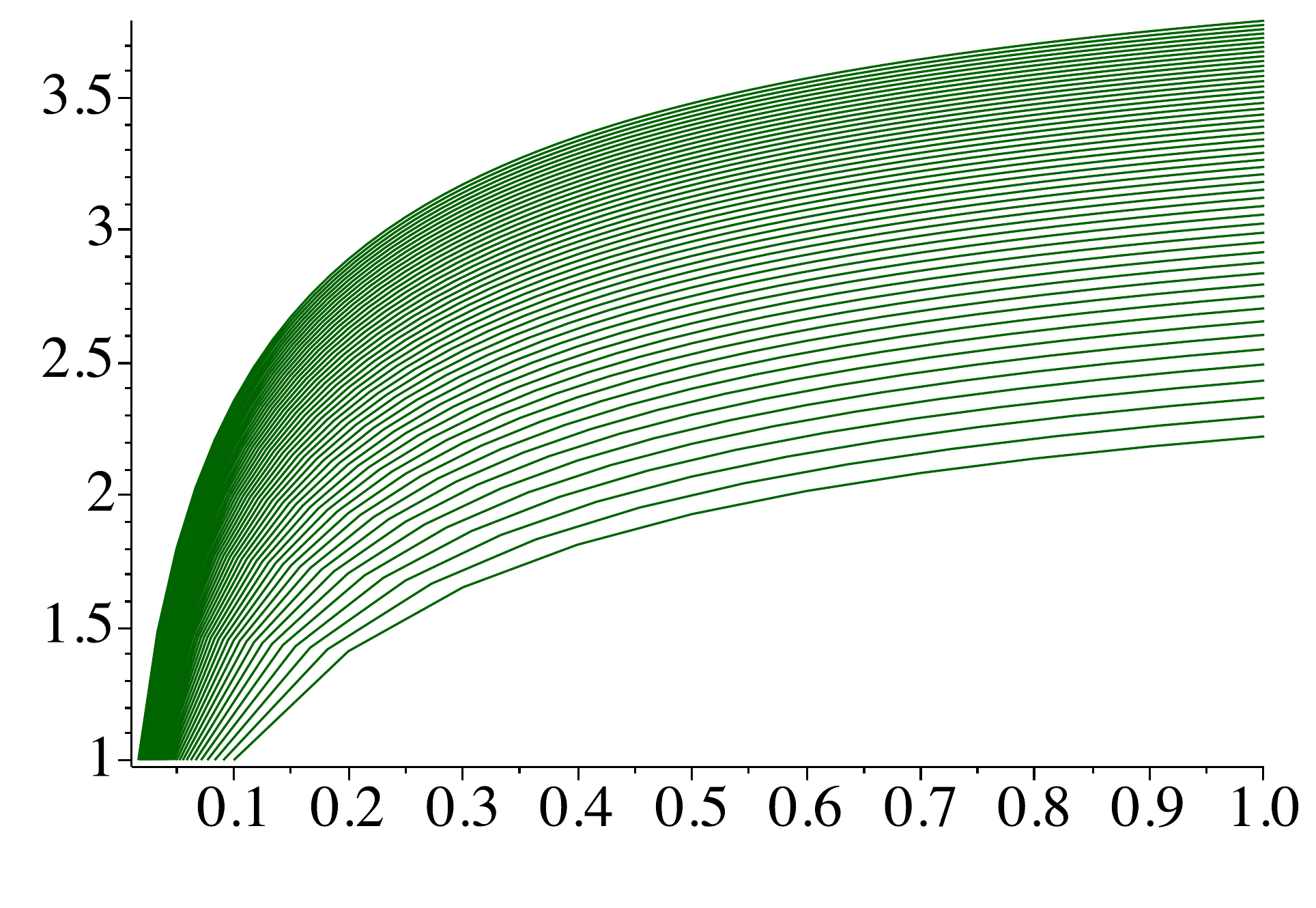}\quad
\includegraphics[width=6cm]{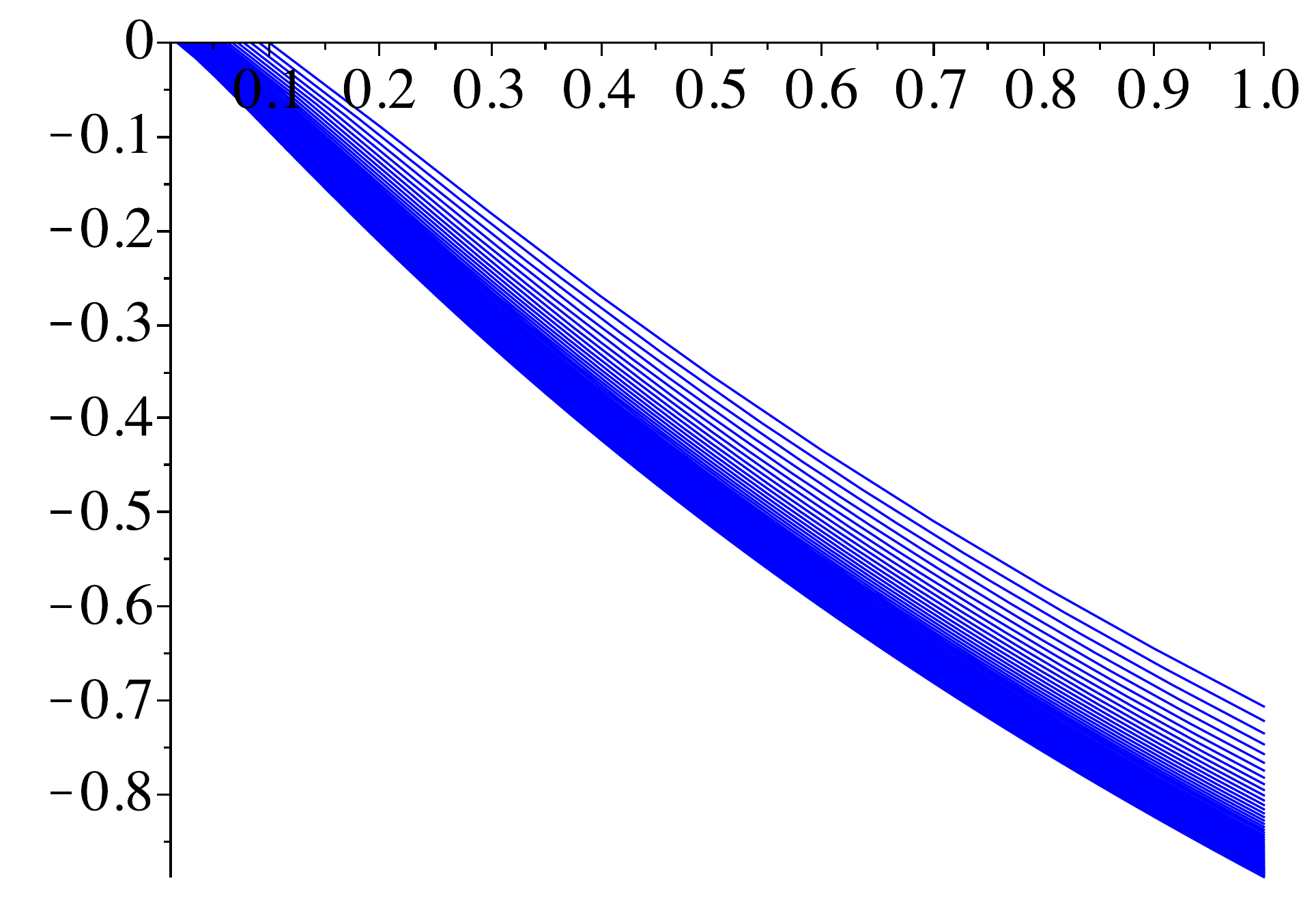}
\end{center}
\vspace*{-.6cm} \caption{\emph{Left: the sequence 
$\mu_{n,m}^*$ for $1\le m\le n$ and $n=10,\dots,60$; 
Right: the difference between $\mu_{n,m}^*-H_m$ 
for $n,m$ in the same ranges.}}\label{App:fig-mu}
\end{figure}

\begin{figure}[!ht]
\begin{center}
\includegraphics[width=6cm]{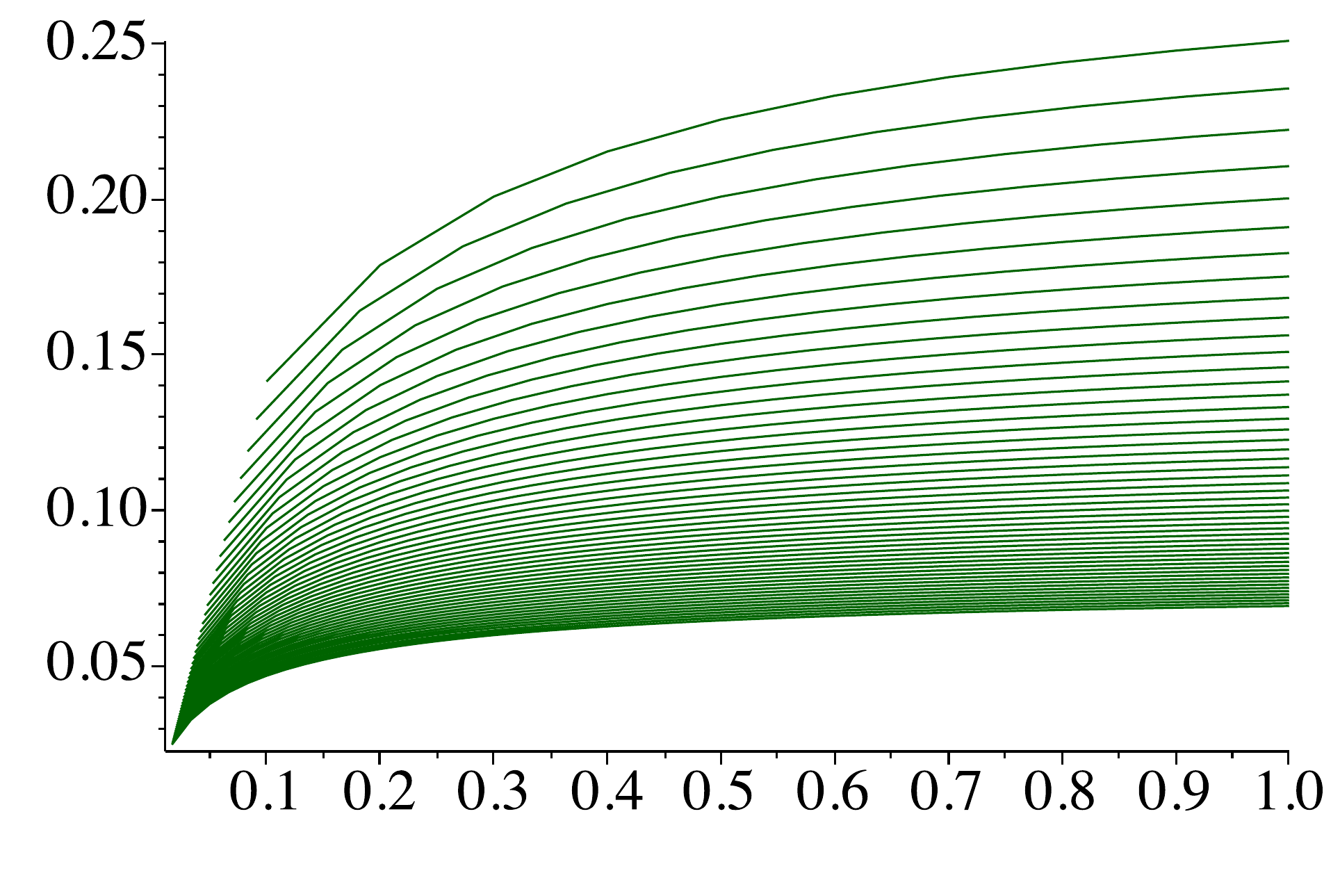}\quad
\includegraphics[width=6cm]{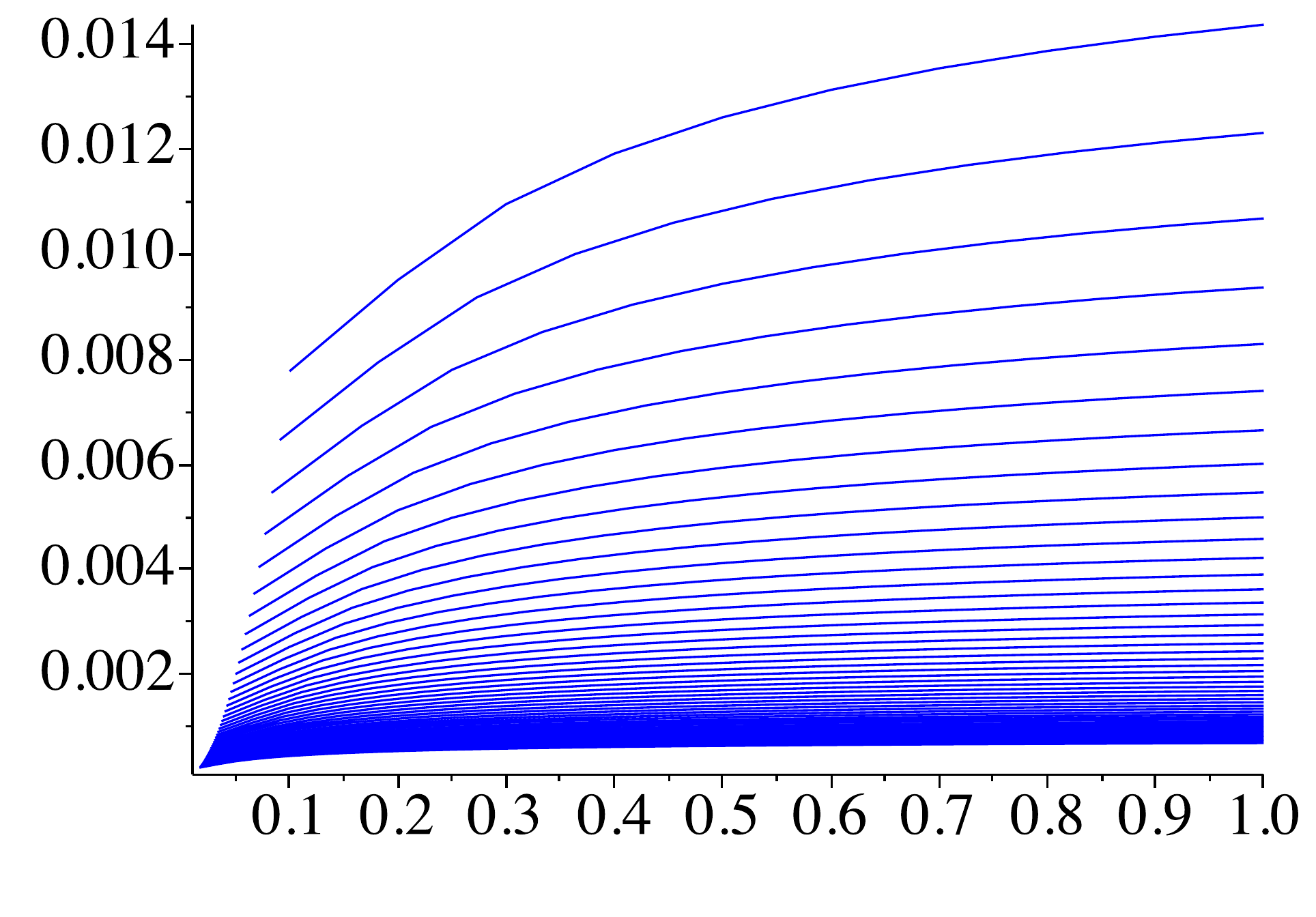}
\end{center}
\vspace*{-.6cm} \caption{\emph{The difference 
$\mu_{n,m}^*-(H_m+\phi_1(\frac mn))$ (left) and
$\mu_{n,m}^* -\bigl(H_m + \phi_1(\frac mn) +
\frac{H_m + \phi_2(\frac mn)}{n}\bigr)$ (right)
for $1\le m\le n$ and $n=10,\dots,60$.}}\label{App:fig-mu2}
\end{figure}

\section*{Appendix. C. Asymptotic expansions for $V_{n,m}^*$
for small $m$ and the refined approximation \eqref{Vnm-star} to
$V_{n,m}^*$}

Recall that
\[
    V_{n,m}^* = \frac{e_{n}^2}{n^2}
    \big(\mathbb{V}(X_{n+1,m})+\mathbb{E}(X_{n+1,m})\big).
\]
This sequence satisfies $V_{n,0}^*=0$, and for $1 \le m \le n$,
\begin{equation}\label{App:Vnm_rec}
    \sum_{1\le \ell \le m} \lambda_{n,m,\ell}^*
    \left(V_{n,m}^*-V_{n,m-\ell}^*\right) = T_{n,m}^*,
\end{equation}
where
\begin{equation*}
    T_{n,m}^* := \sum_{1\le \ell \le m} \lambda^*_{n,m,\ell}
    \left(\mu_{n,m}^*-\mu_{n,m-\ell}^*\right)^2.
\end{equation*}
From this recurrence, we obtain the following expansions.
\begin{align*}
    V_{n,1}^* &= 1,\\
    V_{n,2}^* & = \tfrac{5}{4} - \tfrac{1}{2} n^{-1}
    + \tfrac{3}{4} n^{-2} - \tfrac{5}{4} n^{-3}
    + \tfrac{31}{16} n^{-4} - 3 n^{-5} + O(n^{-6}),\\
    V_{n,3}^* & = \tfrac{49}{36} - \tfrac{17}{18} n^{-1}
    + \tfrac{52}{27} n^{-2} - \tfrac{139}{36} n^{-3}
    + \tfrac{3157}{432} n^{-4} - \tfrac{361}{27} n^{-5}
    + O(n^{-6}),\\
    V_{n,4}^* & = \tfrac{205}{144} - \tfrac{95}{72} n^{-1}
    + \tfrac{1489}{432} n^{-2} - \tfrac{1243}{144} n^{-3}
    + \tfrac{33091}{1728} n^{-4} - \tfrac{28979}{864} n^{-5}
    + O(n^{-6}).
\end{align*}
Observe that the leading constant terms are exactly given by
\[
    \left\{H_m^{(2)}\right\}
    = \left\{1,\tfrac54,\tfrac{49}{36},\tfrac{205}{144},
    \tfrac {5269}{3600},\tfrac{5369}{3600},\dots\right\}.
\]
These expansions suggest the general form
\begin{equation*}
    V_{n,m}^* \approx H_{m}^{(2)} + \sum_{k \ge 1}
    \frac{\tilde{d}_{k}(m)}{n^{k}}.
\end{equation*}
With this form using the technique of matched asymptotics, we are
then led to the following explicit expressions.
\begin{small}
\begin{align*}
    \tilde{d}_{1}(m) & = -2H_{m} + 2 H_{m}^{(2)},
    \quad \text{for $m \ge 0$},\\
    \tilde{d}_{2}(m) & = -\tfrac{11}{2} H_{m}
    + \tfrac{7}{3} H_{m}^{(2)} +\tfrac{7}{12} +\tfrac{11}{4}m,
    \quad \text{for $m \ge 2$},\\
    \tilde{d}_{3}(m) & = -\tfrac{73}{9} H_{m}
    +\tfrac{7}{3} H_{m}^{(2)}
    +\tfrac{1}{6} + \tfrac{239}{36} m
    - \tfrac{49}{36} m^{2}, \quad \text{for $m \ge 2$},\\
    \tilde{d}_{4}(m) & = -\tfrac{1349}{144} H_{m}
    + 2 H_{m}^{(2)} +\tfrac{197}{144}
    +\tfrac{14135}{1728}m -\tfrac{6283}{2880}m^{2}
    +\tfrac{2473}{4320}m^{3}, \quad \text{for $m \ge 4$}.
\end{align*}
\end{small}
The above expansions for small $m$ suggest the more uniform
asymptotic expansion for $V_{n,m}^*$ for $1\le m\le n$
\begin{equation}\label{App:Vnm_expansion}
    V_{n,m}^* \sim H_{m}^{(2)}
    + \sum_{k \ge 1} \frac{a_{k} H_{m}
    + \psi_{k}(\alpha) + c_{k}H_{m}^{(2)}}{n^{k}},
\end{equation}
in the sense that when omitting all terms with indices $k > K$
introduces an error of order $n^{-(K+1)} H_{m}$; furthermore, the
expansion holds uniformly for $K \le m \le n$. We elaborate this
approach by carrying out the required calculations up to $k=2$,
which then characterizes particularly the constant $a_{2}$ and the
function $\psi_{2}(z)$.

We start with the formal expansion \eqref{App:Vnm_expansion} and
expand in recurrence \eqref{App:Vnm_rec} all terms for large $m =
\alpha n$ in decreasing powers of $n$; we then match the
coefficients of $n^{-(K+1)}$ on both sides for each $K \ge 1$. To
specify the initial condition $\psi_{K}(0)$ we incorporate the
information from the asymptotic expansion for $V_{n,K}^*$ (obtained
by exact solution). With this algorithmic approach it is possible to
determine the coefficients $a_{k}$ and $c_{k}$ and the functions
$\psi_{k}(z)$ successively one after another. We remark that a
formalization of this procedure at the generating function level as
carried out for the expectation in Section~\ref{sec:mu-ae} could be
given also, but here we do not pursue this any further.

We use the expansions
\begin{align*}
    & \phi\Big(\frac{m}{n}\Big) - \phi\Big(\frac{m-\ell}{n}\Big)
    = \phi'(\alpha) \frac{\ell}{n} - \phi''(\alpha)
    \frac{\ell^{2}}{2n^{2}} + \phi'''(\alpha)
    \frac{\ell^{3}}{6n^{3}} + \cdots,\\
    & H_{m} - H_{m-\ell} = \frac{\ell}{\alpha \, n}
    + \frac{\ell(\ell-1)}{2 \alpha^{2} \, n^{2}}
    + \frac{\ell (\ell-1) (2\ell-1)}{6 \alpha^{3} \, n^{3}} 
    + \cdots,\\
    & H_{m}^{(2)} - H_{m-\ell}^{(2)}
    = \frac{\ell}{\alpha^{2} \, n^{2}}
    + \frac{\ell(\ell-1)}{\alpha^{3} \, n^{3}}
    + \frac{\ell (\ell-1) (2\ell-1)}{2 \alpha^{4} \, n^{4}} + \cdots
\end{align*}
as well as those for $\mu_{n,m}^*$ and
$\thickbar{\Lambda}_{n,m}^{(r)}$ in \eqref{mu-n-star-ae} and
\eqref{Lmbd}, respectively. The expansion of the right-hand side of
\eqref{App:Vnm_rec} then starts as follows.
\begin{equation*}
    T_{n,m}^* = \frac{T_{1}(\alpha)}{n^{2}}
    + \frac{T_{2}(\alpha)}{n^{3}} + \cdots,
\end{equation*}
where
\begin{align*}
    T_{1}(z) & =\frac{S_{2}(z)}{S_{1}^{2}(z)},\\
    T_{2}(z) & = -\frac{S_{2}^{2}(z) S_{1}'(z)}{S_{1}^{4}(z)}
    + \frac{S_{3}(z) S_{1}'(z)}{S_{1}^{3}(z)}
    + \frac{2 S_{0}(z) S_{2}(z)}{S_{1}^{3}(z)}
    + \frac{S_{0}(z)}{S_{1}^{2}(z)}
    - \frac{S_{2}(z)}{2 S_{1}^{2}(z)} - \frac{2}{S_{1}(z)}.
\end{align*}
For the left-hand side of \eqref{App:Vnm_rec}, the asymptotic form
\eqref{App:Vnm_expansion} leads to
\begin{equation*}
    \sum_{1 \le \ell \le m} \lambda_{n,m,\ell}^*
    \left(V_{n,m}^* - V_{n,m-\ell}^*\right)
    = \frac{V_{1}(\alpha)}{n^{2}} + \frac{V_{2}(\alpha)}{n^{3}}
    + \cdots,
\end{equation*}
where
\begin{align*}
    V_{1}(z) & = \left(\frac{1}{z^{2}}+\frac{a_{1}}{z}
    + \psi_{1}'(z)\right) S_{1}(z),\\
    V_{2}(z) & = \left(-\frac{1}{z^{2}} -\frac{a_{1}}{z}
    - \psi_{1}'(z)\right) S_{0}(z)\\
    & \quad \mbox{} + \left(-\frac{1}{z^{3}}-\frac{1}{2z^{2}}
    +\frac{c_{1}}{z^{2}}-\frac{a_{1}}{2z^{2}}-\frac{a_{1}}{2z}
    +\frac{a_{2}}{z}-\frac{\psi_{1}'(z)}{2}
    +\psi_{2}'(z)\right) S_{1}(z)\\
    & \quad \mbox{} + \left(\frac{1}{z^{3}}
    +\frac{a_{1}}{2z^{2}}-\frac{\psi_{1}''(z)}{2}\right) S_{2}(z).
\end{align*}
Observe that all functions $V_{k}(z)$, $T_{k}(z)$ have a simple pole
at $z=0$.

We match the terms in the expansion and consider $V_{1}(z) =
T_{1}(z)$. First we compare the first two terms of the Laurent
expansions of both functions. Using \eqref{eqn:SrTr_expansion}, we
get
\begin{align*}
    V_{1}(z) & = \frac{1}{z} + \left(\frac{3}{2}+a_{1}\right)   
    + O(z),\\
    T_{1}(z) & = \frac{1}{z} - \frac{1}{2} + O(z),
\end{align*}
and by matching the two constant terms, we see that $a_{1} = -2$.
The equation $V_{1}(z) = T_{1}(z)$ characterizes then the function
$\psi_{1}'(z)$ of the form
\begin{equation*}
    \psi_{1}'(z) = \frac{S_{2}(z)}{S_{1}^3(z)}
    - \frac{1}{z^{2}} + \frac{2}{z}.
\end{equation*}

Next we consider $V_{2}(z) = T_{2}(z)$ and obtain
\begin{align*}
    V_{2}(z) & = \left(-\frac{1}{2}+c_{1}\right) \frac{1}{z}
    + \left(-\frac{5}{12} - a_{1} + \frac{3 c_{1}}{2}
    + a_{2}\right) + O(z),\\
    T_{2}(z) & = \frac{3}{2 z} - \frac{11}{12}+ O(z),
\end{align*}
and thus, by matching the terms and using the values already
computed in the first-order approximation for $V_{n,m}^*$, $c_{1} =
2$ and $a_{2} = -\frac{11}{2}$. Then the function $\psi_{2}'(z)$ can
be characterized by equating $V_{2}(z) = T_{2}(z)$, which then gives
\begin{align*}
    \psi_{2}'(z) & = -\frac{5 S_{2}^{2}(z) S_{1}'(z)}{2 S_{1}^{5}(z)}
    + \frac{S_{3}(z) S_{1}'(z)}{S_{1}^{4}(z)}
    + \frac{3 S_{2}(z) S_{0}(z)}{S_{1}^{4}(z)}
    + \frac{S_{2}(z) S_{2}'(z)}{2 S_{1}^{4}(z)}
    + \frac{S_{0}(z)}{S_{1}^{3}(z)} - \frac{2}{S_{1}^{2}(z)}\\
    & \quad \mbox{} + \frac{1}{z^{3}} - \frac{3}{z^{2}}
    + \frac{11}{2 z}.
\end{align*}
All constants and functions here match with those obtained earlier
in previous paragraphs, and we can pursue the same calculations
further and obtain finer approximations. For example, we have $c_{2}
= \frac{7}{3}$. But the calculations are long and laborious.

Finally, it remains to determine the constant terms in the Taylor
expansion of the functions $\psi_{k}(z)$ by adjusting them to the
expansion of $V_{n,m}^*$ for small $m$. This yields $\psi_{1}(0) =
0$, and
\begin{equation*}
    \psi_{2}(0) = \frac{7}{12}.
\end{equation*}
This characterizes the function $\psi_{2}(z)$ in
Theorem~\ref{thm:var1} as follows.
\begin{align*}
    \psi_{2}(z) & = \frac{7}{12}
    + \int_{0}^{z}\bigg[ -\frac{5 S_{2}^{2}(t) S_{1}'(t)}
    {2 S_{1}^{5}(t)} + \frac{S_{3}(t) S_{1}'(t)}{S_{1}^{4}(t)}
    + \frac{3 S_{2}(t) S_{0}(t)}{S_{1}^{4}(t)}
    + \frac{S_{2}(t) S_{2}'(t)}{2 S_{1}^{4}(t)}
    + \frac{S_{0}(t)}{S_{1}^{3}(t)}\\
    & \qquad \qquad \qquad \mbox{} - \frac{2}{S_{1}^{2}(t)}
    + \frac{1}{t^{3}} - \frac{3}{t^{2}}
    + \frac{11}{2 t} \bigg] \dd t.
\end{align*}
In particular, the first few terms in the Taylor expansion of
$\psi_{2}(z)$ are given by
\begin{align*}
    \psi_{2}(z) & = \tfrac{7}{12} + \tfrac{239}{36} z
    - \tfrac{6283}{2880} z^{2} - \tfrac{4529}{3600} z^{3}
    + \tfrac{9283591}{1814400} z^{4}
    - \tfrac{137478949}{14112000} z^{5} + \cdots.
\end{align*}

\section*{Appendix. D. Closeness of the approximation 
\eqref{Vnm-star} for $V_{n,m}^*$: graphical representations}

\begin{figure}[!ht]
\begin{center}
\includegraphics[width=6cm]{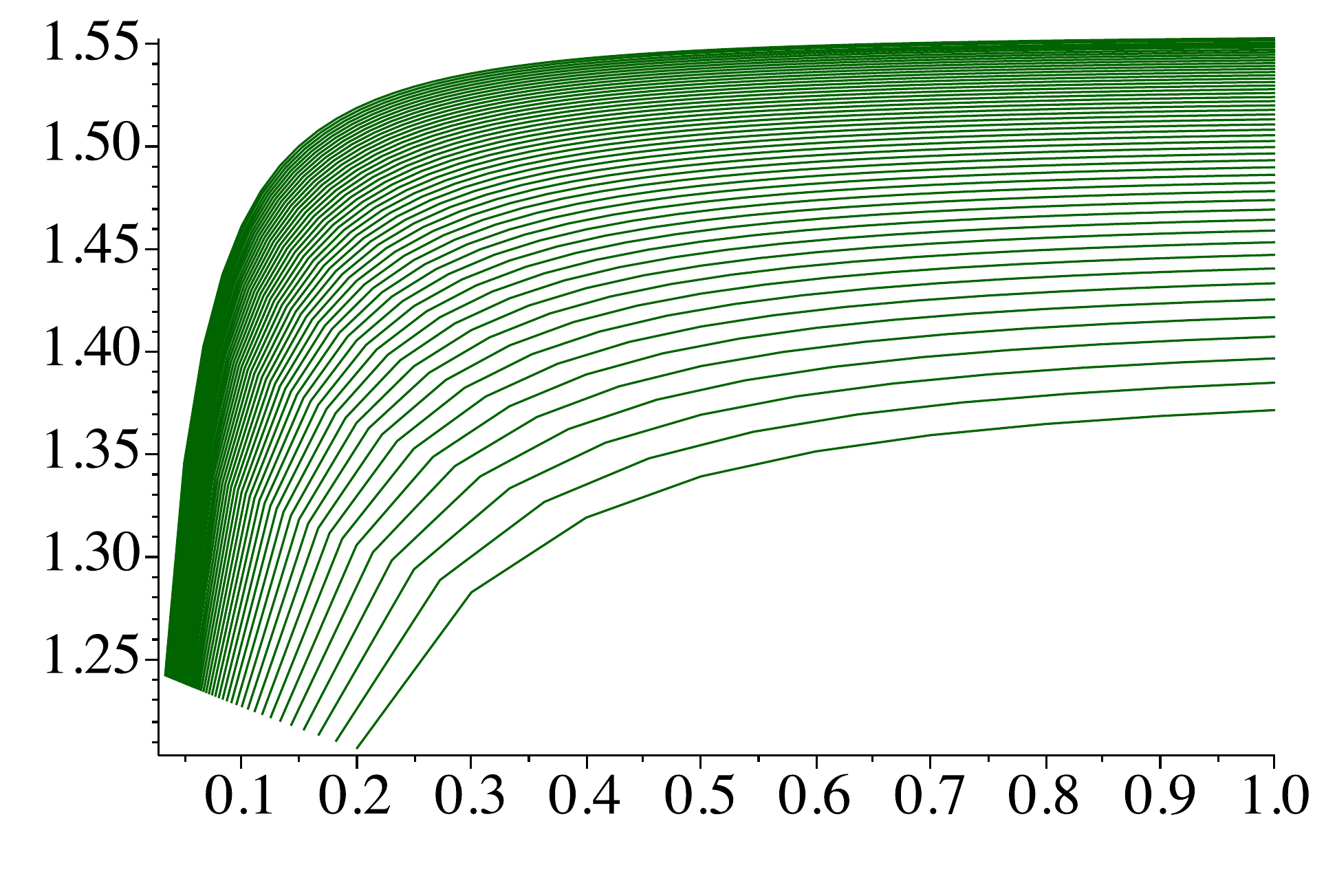}\quad
\includegraphics[width=6cm]{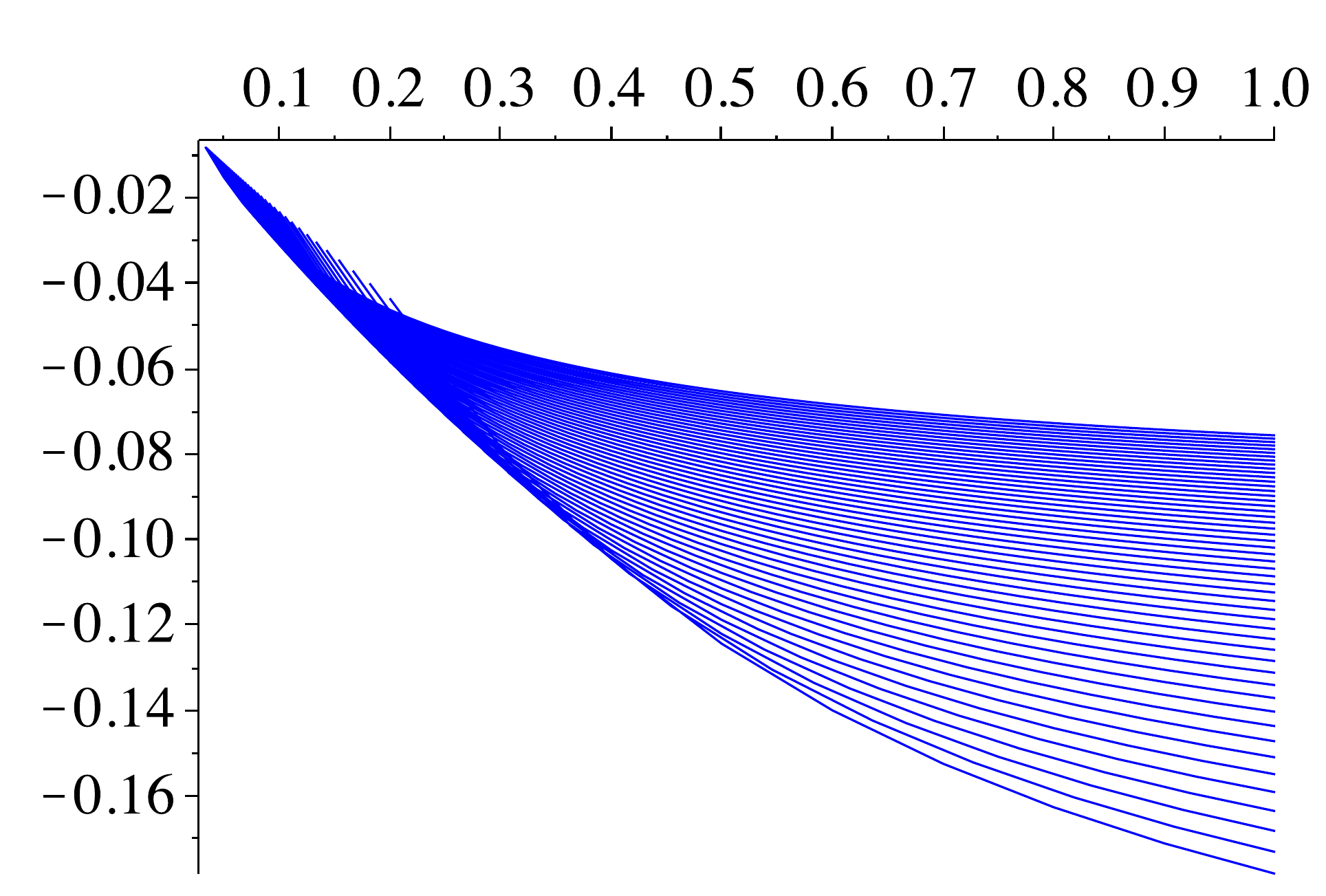}
\end{center}
\vspace*{-.6cm} \caption{\emph{Left: the sequence 
$V_{n,m}^*$ for $2\le m\le n$ and $n=10,\dots,60$; 
Right: the difference between $V_{n,m}^*-H_m^{(2)}$ 
for $n,m$ in the same ranges.}}\label{App:fig-V}
\end{figure}

\begin{figure}[!ht]
\begin{center}
\includegraphics[width=6cm]{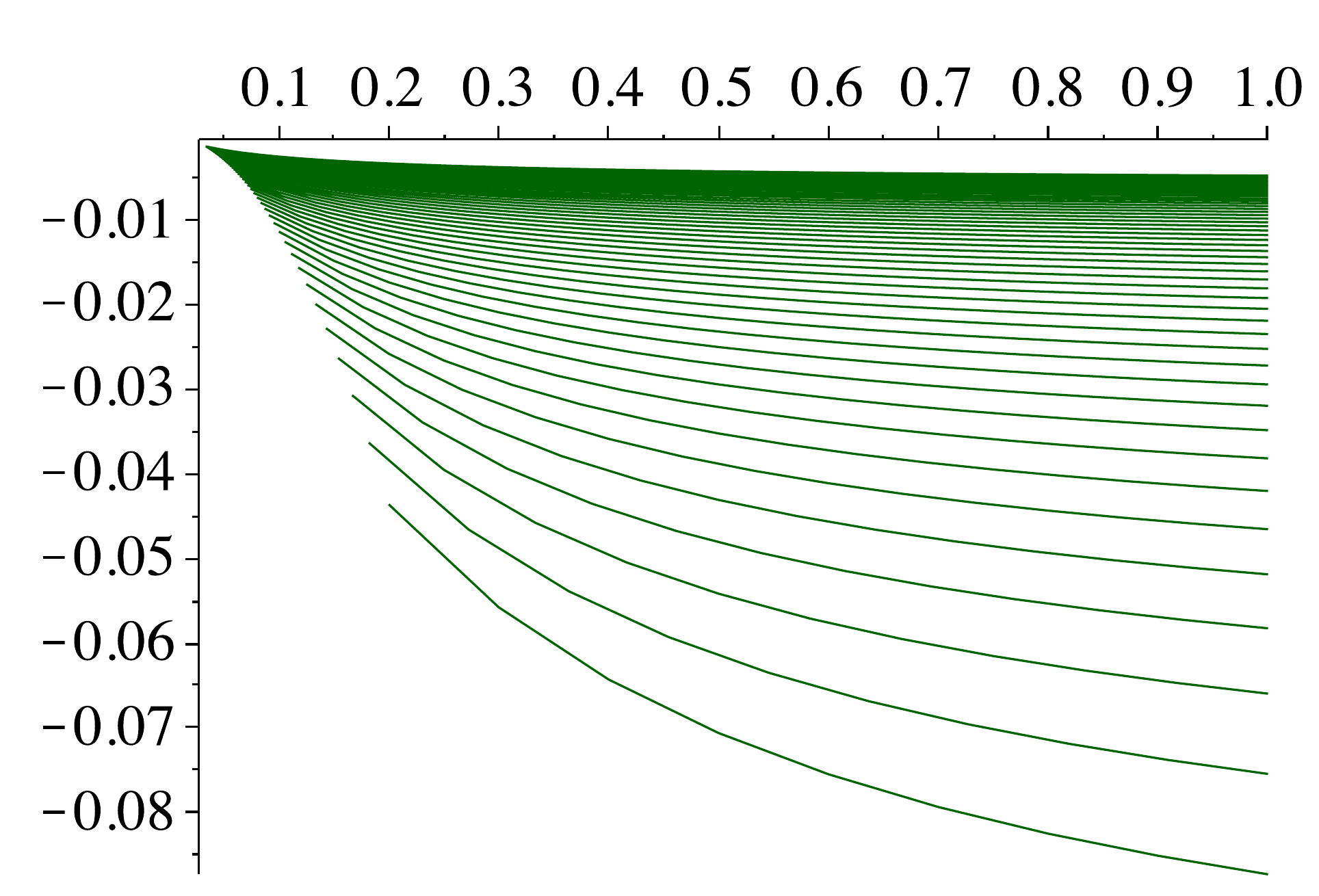}\quad
\includegraphics[width=6cm]{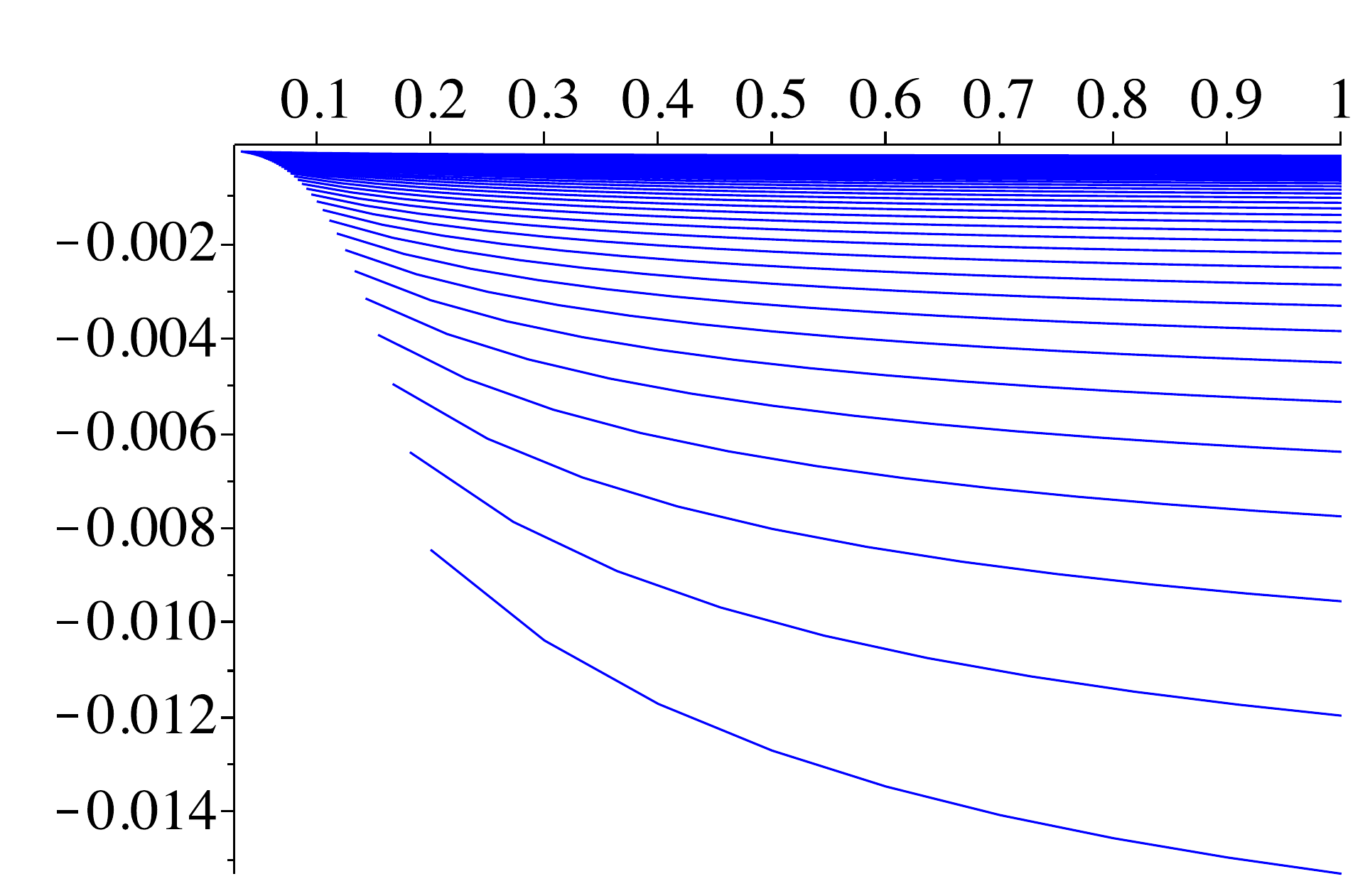}
\end{center}
\vspace*{-.6cm} \caption{\emph{The difference 
$V_{n,m}^*-\bigl(H_{m}^{(2)} + \frac{-2 H_{m}
+ \psi_{1}(\alpha) + 2H_{m}^{(2)}}{n}\bigr)$ (left) and
$V_{n,m}^* -\bigl( H_{m}^{(2)} + \frac{-2 H_{m}
+ \psi_{1}(\alpha) + 2H_{m}^{(2)}}{n}
+ \frac{-\frac{11}{2} H_{m} + \psi_{2}(\alpha)
+ \frac{7}{3} H_{m}^{(2)}}{n^{2}}\bigr)$ (right)
for $2\le m\le n$ and $n=10,\dots,60$.}}\label{App:fig-V2}
\end{figure}

\vspace*{15cm}

\end{document}